\documentclass[11pt,article]{article}

\usepackage{amsmath}
\usepackage{amsthm}
  \usepackage{paralist}
  \usepackage{graphics} %% add this and next lines if pictures should be in esp format
  \usepackage{epsfig} %For pictures: screened artwork should be set up with an 85 or 100 line screen
\usepackage{graphicx}
\usepackage{caption}
\usepackage{subcaption}
\usepackage{epstopdf}%This is to transfer .eps figure to .pdf figure; please compile your paper using PDFLeTex or PDFTeXify.
 \usepackage[colorlinks=true]{hyperref}
 \usepackage{multirow}
\input{amssym.tex}

\newtheorem{theorem}{Theorem}[section]
\newtheorem{corollary}{Corollary}
\newtheorem{lemma}[theorem]{Lemma}
\newtheorem{proposition}{Proposition}

 \numberwithin{equation}{section}
\newtheorem{remark}{Remark}

\newcommand{\keywords}

\def\bc{\begin{center}}       \def\ec{\end{center}}
\def\ba{\begin{array}}        \def\ea{\end{array}}
\def\be{\begin{equation}}     \def\ee{\end{equation}}
\def\bea{\begin{eqnarray}}    \def\eea{\end{eqnarray}}
\def\beaa{\begin{eqnarray*}}  \def\eeaa{\end{eqnarray*}}

\def\mathbb{\Bbb}

\begin{document}

\title{\bf Global existence and steady states of a two competing species Keller--Segel chemotaxis model}
\author{Chunyi Gai \thanks{ {\tt chunyi.gai@dal.ca}, Department of Mathematics and Statistics, Dalhousie University,6316 Coburg Road, Halifax, Nova Scotia, Canada B3H 4R2.}, Qi Wang \thanks{ {\tt qwang@swufe.edu.cn}, Department of Mathematics, Southwestern University of Finance and Economics, 555 Liutai Ave, Wenjiang, Chengdu, Sichuan 611130, China.}, Jingda Yan \thanks{ {\tt lz@2011.swufe.edu.cn}, Hanqing Advanced Institute of Economics and Finance, Renmin University of China, No. 59 Zhongguancun Street, Haidian District, Beijing 100872, China.}\\
}

\date{}
\maketitle

\abstract
We study a diffusive Lotka-Volterra competition system with advection under Neumann boundary conditions.  Our system models a competition relationship that one species escape from the region of high population density of their competitors in order to avoid competition.  We establish the global existence of bounded classical solutions to the system over one-dimensional finite domains.  For multi-dimensional domains, globally bounded classical solutions are obtained for a parabolic-elliptic system under proper assumptions on the system parameters.  These global existence results make it possible to study bounded steady states in order to model species segregation phenomenon.  We then investigate the one-dimensional stationary problem.  Through bifurcation theory, we obtain the existence of nonconstant positive steady states, which are small perturbations from the positive equilibrium; we also rigourously study the stability of these bifurcating solutions when diffusion coefficients of the escaper and its competitor are large and small respectively.  In the limit of large advection rate, we show that the reaction-advection-diffusion system converges to a shadow system involving the competitor population density and an unknown positive constant.  Existence and stability of positive nonconstant solutions to the shadow system have also been obtained through bifurcation theories.  Finally, we construct infinitely many single interior transition layers to the shadow system when crowding rate of the escapers and diffusion rate of their interspecific competitors are sufficiently small.  The transition-layer solutions can be used to model the interspecific segregation phenomenon.

\textbf{Keywords: Global existence, Lotka-Volterra competition system, steady state, bifurcation, stability, transition layer.}

\section{Introduction}
Mathematical analysis of reaction-diffusion systems has become more and more important in understanding the ecological behaviors of interacting species.  To study the population dynamics of two competing species over a homogeneous environment, we consider the following coupled reaction-advection-diffusion system
\begin{equation}\label{16}
\left\{
\begin{array}{ll}
u_t=\nabla \cdot(D_1 \nabla u+\chi u \phi(v)\nabla v)+(a_1-b_1u-c_1v)u,&x \in \Omega,~t>0,     \\
\tau v_t= D_2 \Delta v+(a_2-b_2u-c_2v)v,&x \in \Omega,~t>0,\\
\frac{\partial u}{\partial \textbf{n}}=\frac{\partial v}{\partial \textbf{n}}=0,& x \in \partial \Omega,~t>0,\\
u(x,0)=u_0(x) \geq 0,~v(x,0)=v_0(x) \geq 0,& x\in \Omega,
\end{array}
\right.
\end{equation}
where $\Omega$ is a bounded domain in $\mathbb{R}^N$, $N\geq1$;  $(u,v)=(u(x,t),v(x,t))$ are the population densities of two competing species at space-time $(x, t)$; $D_1$ and $D_2$ are positive constants that measure the tendency of random walks of the two species respectively;  $\chi>0$ and $\tau\geq0$ are constants, and $\phi(v)$ is a smooth function that reflects the variation of the advection flux with respect to the population density $v$;  the constants $a_i,b_i,c_i$, $i=1,2$ are nonnegative and ecologically, $a_1$ and $a_2$ reflect the intrinsic growth rates of the species, $b_1$ and $c_2$ measure the levels of intraspecific crowding, while $b_2$ and $c_1$ interpret the intensities of interspecific competition; moreover, $\textbf{n}$ is the unit outer normal to the boundary $\partial \Omega$ and the non-flux boundary condition means that the domain is an enclosed habitat.  We assume that the initial data $u_0$ and $v_0$ are not identically zero.

In this paper, we are concerned with the mathematical modeling of interspecific segregation by (\ref{16}).  For this purpose, we show that (\ref{16}) allows global existence of bounded classical solutions-see Theorem \ref{thm25} and Theorem \ref{thm27}.  We also study the existence and stability of nontrivial steady states to (\ref{16})-see Theorem \ref{thm31} and Theorem \ref{thm32}.  Moreover, we show that these nontrivial steady states can be approximated by a shadow system which has the transition layer structures that can be used to model the segregation phenomenon in interspecific competition-see Theorem \ref{thm42} and Theorem \ref{thm54}.

In the absence of advection term, i.e, for $\chi=0$, (\ref{16}) becomes
\begin{equation}\label{11}
\left\{
\begin{array}{ll}
u_t=D_1 \Delta u+ (a_1-b_1u-c_1v)u,&x \in \Omega,~t>0,     \\
v_t=D_2 \Delta v+(a_2-b_2u-c_2v)v,&x \in \Omega,~t>0,\\
\frac{\partial u}{\partial \textbf{n}}=\frac{\partial v}{\partial \textbf{n}}=0,& x \in \partial \Omega,~t>0,\\
u(x,0)=u_0(x) \geq 0,~v(x,0)=v_0(x) \geq 0,& x\in \Omega.
\end{array}
\right.
\end{equation}
It is easy to see that (\ref{11}) has four equilibria $(0,0),(\frac{a_1}{b_1},0),(0,\frac{a_2}{c_2})$ and $(\bar{u},\bar{v})$,
where
\[(\bar{u},\bar{v})=\left(\frac{a_1c_2-a_2c_1}{b_1c_2-b_2c_1},\frac{a_2b_1-a_1b_2}{b_1c_2-b_2c_1}\right);\]
moreover, $(\bar{u},\bar{v})$ is positive if and only if
\begin{equation}\label{12}
\frac{c_1}{c_2}<\frac{a_1}{a_2}<\frac{b_1}{b_2},\text{~or~}\frac{b_1}{b_2}<\frac{a_1}{a_2}<\frac{c_1}{c_2},
\end{equation}
which are referred as the weak and strong competition respectively.  Ecologically, the positive solution interprets coexistence of the competing species.

It is well known that the dynamics of (\ref{11}) is dominated by that of the ODEs if one of the diffusion rates is large.  Moreover, the dynamics of (\ref{11}) is rather simple and has been completely understood in the weak competition case $\frac{c_1}{c_2}<\frac{a_1}{a_2}<\frac{b_1}{b_2}$.  We summarize these results as follows.
\begin{theorem}[\cite{CHS,DR,LN}]\label{thm11}
Suppose that (\ref{12}) holds.  There exists a positive constant $D_0=D_0(a_i,b_i,c_i)$, $i=1,2$, such that (\ref{11}) has no nonconstant positive steady state if $\max \{D_1, D_2\}>D_0$; moreover, if $\frac{c_1}{c_2}<\frac{a_1}{a_2}<\frac{b_1}{b_2}$, the positive steady state $(\bar{u},\bar{v})$ of (\ref{11}) is asymptotically stable in the sense that, for any positive solution $(u(x,t),v(x,t))$ of (\ref{11}), $\lim_{t \rightarrow +\infty}\left(u(x,t),v(x,t)\right) \rightarrow (\bar{u},\bar{v})$, regardless of the initial data and the size of diffusion rates $D_1,D_2$.
\end{theorem}

However, the strong competition case $\frac{b_1}{b_2}<\frac{a_1}{a_2}<\frac{c_1}{c_2}$ is much more complicated and various interesting phenomena may occur.  Kishimoto and Weinberger showed in \cite{KiW} that, for $\Omega$ being any convex domain in $\mathbb{R}^N$, $N\geq1$, system (\ref{11}) has no nonconstant stable steady states.   On the other hand, Matano and Mimura \cite{MM} established the existence of a stable nonconstant positive steady state when $\Omega$ has a dumbbell shape with a narrow bar.  Similar results on positive nonconstant solutions can be found in \cite{KY,MEF} and the references therein.  We refer the reader to \cite{LN} for more detailed discussions on (\ref{11}) and \cite{E,Ni} for similar works over a spatially heterogeneous habitat.

From the viewpoint of mathematical modeling, it is not entirely reasonable to assume that individuals move around randomly, even in a homogeneous environment.  It ignores the population pressures from intraspecific crowding and interspecific competition.  Moreover, according to Theorem \ref{thm11}, diffusions alone are insufficient to induce nontrivial patterns of (\ref{11}) in almost all cases (at least when $\Omega$ is convex), therefore it can not be used to model the inhomogeneous population distributions of interacting species.  To study the directed movements of individuals due to mutual interactions, Shigesada, Kawasaki and Teramoto \cite{SKT} proposed the following system in 1979 to model the segregation phenomenon of two competing species
\begin{equation}\label{13}
\left\{
\begin{array}{ll}
u_t=\Delta[(D_1+\rho_{11}u+\rho_{12}v)u]+(a_1-b_1u-c_1v)u, &x \in \Omega,~t>0,     \\
v_t=\Delta[(D_2+\rho_{21}u+\rho_{22}v)v]+(a_2-b_2u-c_2v)v,& x \in \Omega,~t>0, \\
\frac{\partial u}{\partial \textbf{n}}=\frac{\partial v}{\partial \textbf{n}}=0,& x \in \partial \Omega,~t>0,\\
u(x,0)=u_0(x) \geq 0,~ v(x,0)=v_0(x) \geq 0,& x\in \Omega,
\end{array}
\right.
\end{equation}
where $\rho_{i,j}$, $i,j=1,2$, are nonnegative constants and the rest notations are the same as in (\ref{11}).  $\rho_{11},\rho_{22}$ are referred as the \textit{self-diffusions} and $\rho_{12},\rho_{21}$ the \textit{cross-diffusions}.  Ecologically, $\rho_{11},\rho_{22}$ represent the diffusion pressures due to the presence of conspecifics and $\rho_{12},\rho_{21}$ measure the diffusion pressures from the interspecific competitors.

Considerable work has been done on the qualitative analysis of the steady states regarding (\ref{13}).  First of all, in the absence of cross-diffusions, one can easily show that system (\ref{13}) does not exhibit diffusion-induced or self-diffusion-induced instability in the framework of Turing's analysis.  To study the effect of cross-diffusion on the pattern formations of (\ref{13}), Mimura and Kawasaki \cite{MK} made the first attempt in this direction, assuming $\rho_{11}=\rho_{21}=\rho_{22}=0$, in the weak competition case for $\Omega=(0,L)$.  They applied Lyapunov-Schmidt method to obtain positive nontrivial solutions of (\ref{13}) which are small perturbations from $(\bar{u},\bar{v})$.  Similar results were obtained by Matano and Mimura in \cite{MM} for the same system with $D_1=D_2$ through local bifurcation analysis.  For $\rho_{11}=\rho_{21}=\rho_{22}=0$, Mimura, et al applied singular perturbation and bifurcation methods in \cite{M1,MNTT} to show that (\ref{13}) over one-dimensional finite admits solutions with interior transition layer if $\rho_{12}$ is large and $d_1$ is sufficiently small in both weak and strong competition cases.

Great progress was made by Lou and Ni in \cite{LN,LN2} on the qualitative analysis of steady states of (\ref{13}) over multi-dimensional domains.  Their results can be briefly summarized as follows.  Both diffusion and self-diffusion have smoothing effect on the stationary problem of (\ref{13}) and only one of them is required large to suppress the formation of nonconstant steady states when the corresponding cross-diffusion cooperates.  However, cross-diffusion tends to support nonconstant steady states.  Moreover, they established the existence and limiting profiles of the nonconstant steady states as $\rho_{12}/D_1$ or $\rho_{21}/D_2$ approaches to infinity.  They also studied their shadow systems that have boundary spikes.  We refer the reader to \cite{KW,LNY,Ni,NWX,WX}, etc. for results and recent developments on these problems.

We consider in this paper an alternative reaction-diffusion system with advection terms in the following form

\begin{equation}\label{14}
\left\{
\begin{array}{ll}
u_t=\nabla \cdot(D_1 \nabla u+\chi \Phi_1(u,v)\nabla v)+f(u,v),&x \in \Omega,~t>0,     \\
v_t=\nabla \cdot(D_2 \nabla v+\chi \Phi_2(u,v)\nabla u)+g(u,v),&x \in \Omega,~t>0,\\
u(x,0)=u_0(x) \geq 0,~v(x,0)=v_0(x) \geq 0,& x\in \Omega,
\end{array}
\right.
\end{equation}
where $\chi$ is a constant that measures the strength of interference from interspecific competitors; $\Phi_1$ and $\Phi_2$, the so-called sensitivity potentials, are smooth functions that reflect the variations of directed dispersal magnitude with respect to the inter- and intra-species population densities.  If $\chi>0$, this system models a competition relationship that both species escape from the region of high concentration of their interspecific competitors, and if $\chi<0$, both species invade the habitat of their competitors.

To justify our choices on the constant $\chi$ and the sensitivity functions, we derive model (\ref{14}) from a macroscopic approach based on the conservation of species populations.  For species $u$, we have a transport equation in the following form
\begin{equation}\label{15}
u_t=-\nabla \cdot \textbf{J}+f,
\end{equation}
where $\textbf{J}$ is the total population flux and $f$ is the birth-death rate of the species.  We assume that the total flux is a superposition of the diffusion flux $\textbf{J}_{\text{diffusion}}$ from random walks and the competition flux $\textbf{J}_{\text{competition}}$ due to the interspecific population pressure.  The
diffusion flux is modelled by Fick's law $\textbf{J}_{\text{diffusion}}=-D_2 \nabla u$; on the other hand, if $u$ escapes the habitat of $v$, we assume that the competition flux of $u$ is in the negative direction of $\nabla v$, in which the population density of the species $v$ increases most rapidly;  moreover, we assume that the intensity of the competition flux depends on the population densities of both species, therefore it takes the form $\textbf{J}_{\text{competition}}=-\chi \Phi_1(u,v)\nabla v$, for some $\chi>0$;  similarly we have that $\chi<0$ if $u$ invades the habitat of $v$.  Hence the $u$-equation in (\ref{14}) follows from (\ref{15}) and we can also derive the $v$-equation of (\ref{14}) by the same reasonings.  For the sake of mathematical simplicity and to elucidate the effect of advection on the existence of positive solutions of (\ref{14}) and its stationary system, we choose $\Phi_1(u,v)=u\phi(v)$ and $\Phi_2(u,v) \equiv 0$ and take the classical Lotka-Volterra dynamics, then (\ref{14}) is reduced into (\ref{16}) and it models an escaping-competition relationship as we have discussed above.

To compare (\ref{14}) with the SKT model (\ref{13}), we rewrite the second system as
{\small\begin{equation*}
\left\{
\begin{array}{ll}
u_t=\nabla \cdot [(D_1+2\rho_{11}u+\rho_{12}v) \nabla u+\rho_{12}u\nabla v]+(a_1-b_1u-c_1v)u, &x \in \Omega,~t>0,     \\
v_t=\nabla \cdot[(D_2+\rho_{21}u+2\rho_{22}v)\nabla v+\rho_{21}v \nabla u]+(a_2-b_2u-c_2v)v,& x \in \Omega,~t>0, \\
\frac{\partial u}{\partial \textbf{n}}=\frac{\partial v}{\partial \textbf{n}}=0,& x \in \partial \Omega,~t>0,\\
u(x,0)=u_0(x) \geq 0,~ v(x,0)=v_0(x) \geq 0,& x\in \Omega.
\end{array}
\right.
\end{equation*}}It is easy to see that both diffusion rates in (\ref{13}) increase with respect to the population densities of both species.  Moreover, by our derivation of (\ref{14}), the advection terms interpret a dispersal strategy of avoiding population pressures from inter-species.  Moreover, both diffusion and advection in (\ref{13}) help to avoid intraspecific crowding and interspecific competition.  However, one needs to be cautious in applying (\ref{13}) to model an invading-competing relationship between $u$ and $v$, which can be simulated by taking $\chi<0$ in (\ref{14}).  For example, if $\rho_{ij}<0$, $i\neq j$, the diffusion rates in (\ref{13}) become negative for $u$ and $v$ being large and this makes (\ref{13}) an unrealistic choice for the modeling of species competition.  Therefore, some modifications of (\ref{13}) are required for the sake of mathematical analysis.

We also want to point out that the global existence and blow-up of solutions to the quasilinear models (\ref{13}) and (\ref{14}) are also important and challenging problems.  Considerable amount of works have been done on the global existence of (\ref{13})-see \cite{CLY,LNW} and the surveys in \cite{Ni} for relatively recent works in this direction.  It is commented by Cosner in the survey paper \cite{C} (Sec. 3.4) that, $u$ \emph{advects down} $\nabla v$ and vice-versa in (\ref{13}) and the negative feedbacks between $u$ and $v$ make finite time blow-ups of (\ref{13}) less likely.  However, it is unclear that if they are sufficient to suppress the blow-ups in (\ref{14}) if $\chi>0$.  Moreover, if $\chi<0$, the advection in (\ref{14}) is of Keller-Segel type and the positive stimulations between $u$ and $v$ make the blow-ups likely to occur during finite or infinite time period.  It is also worth pointing out that systems in the form of (\ref{14}) can be used as predator-prey models-see \cite{C} and the references therein for more discussions.

One of the most interesting phenomena in species competition is the segregation of inter-species, that is, $u$ and $v$ dominate two separated patches over $\Omega$.  Time-dependent systems may describe the species segregation phenomenon in terms of blow-up solutions, i.e, the $L_\infty$ norm of the solutions approaches to infinity over finite or infinite time period.  Then the segregation can be simulated by a $\delta$-function or a linear combination of $\delta$-functions that measure the species population densities.  Such attempts have been made on Keller-Segel chemotaxis systems that model the directed cellular movements along the gradient of certain chemicals.  See \cite{HP,HV,Nan} for works in this direction.  Though such blow-up solutions are evidently connected to the species segregation phenomenon, a $\delta-$funtion is not an optimal choice for this purpose since it challenges the rationality that population density can not be infinity in one way and it also brings difficulties to numerical simulations in another way.  An alternative approach is to show that the time-dependent system has global-in-time solutions which converges to bounded steady states with boundary spikes, transition layer, etc.  This approach has also been adopted for the chemotaxis models.  See \cite{HP,WX0} and the references therein for recent results and developments in this direction.

The remaining parts of this paper are organized as follows.  In Section \ref{section2}, we obtain the local existence of classical smooth solutions following the theories of Amann \cite{A,A1} for general quasilinear parabolic systems.  Then we proceed to study the $L^\infty$-bounds of the local solutions to establish the global existence.  The boundedness of $v$ is an immediate result of the Maximum Principles.  We employ semigroup arguments and the Moser-Alikakos iteration technique on appropriate $L^p$-estimate to establish the $L^\infty$-bound of $u$.  Global existence of (\ref{16}) with $\tau \geq0$ is completely proved for $\Omega=(0,L)$.  We also investigate the global existence for a parabolic-elliptic system (\ref{16}) with $\tau=0$ over a bounded multi-dimensional domain $\Omega \subset\mathbb{R}^N$, $N\geq 2$-see Theorem \ref{thm25} and Theorem \ref{thm27} for the one-dimensional and multi-dimensional global existence respectively.

In Section \ref{section3}, we study one-dimensional stationary solutions of (\ref{16}).  By using the classical Crandall-Rabinowitz bifurcation theories \cite{CR,CR2}, we obtain the existence and stability of nontrivial positive steady states of (\ref{16}) in Theorem \ref{thm31} and Theorem \ref{thm32}.  Section \ref{section4} is devoted to study the effect of large advection rate $\chi$.  We show that the stationary system of (\ref{16}) converges to a shadow system as $\chi \rightarrow \infty$ with $\chi/D_1 \rightarrow r\in(0,\infty)$, i.e, $\chi$ and $D_1$ being comparably large.  The shadow system consists of a single PDE of population density $v$ with an unknown positive constant $\lambda$ and an integral constraint-see (\ref{44}).  We then apply the Crandall-Rabinowitz bifurcation theories to investigate the existence and stability of nonconstant positive steady state to the shadow system.

In Section \ref{section5}, we carry out the analysis on the asymptotic behaviors of $v$ in the shadow system as its diffusion rate $D_2=\epsilon$ shrinks to zero.  It is shown that the shadow system admits solutions with an interior transition layer for $D_2=\epsilon>0$ and $b_1$ being sufficiently small-see Theorem \ref{thm54}.  The transition-layer solution is an approximation of a step function and it can be a reasonable modeling of the species segregation.  Finally, we include discussions and propose some interesting problems in Section \ref{section6}.

We remind our reader that we are interested in positive solutions to all the systems and we assume condition (\ref{12}) throughout this paper.  In the sequal, $C$ denotes a positive constant that may vary from line to line unless otherwise stated.

\section{Global-in-time solutions}\label{section2}
In this section, we study the global existence and boundedness of positive classical solutions $(u,v)$ to system (\ref{16}) over a bounded domain $\Omega\subset\mathbb{R}^N$, $N\geq1$.  We shall apply the well-known results of Amann \cite{A,A1} to obtain the local existence and then establish $L^\infty$-bounds of $u$ and $v$ for the global existence.  First of all, we convert the $u$ and $v$ equations of (\ref{16}) into the following integral forms
\begin{equation}\label{21}
\begin{split}
u(\cdot,t)=&\left. e^{D_1 (\Delta-1)t}u_0-\int_0^t \nabla \cdot e^{-D_1A(t-s)}\big( \chi u(\cdot,s) \phi(v(\cdot,s))  \nabla v(\cdot,s) \big) ds\right.\\
&\left.+\int_0^t  e^{-D_1A(t-s)} \big(D_1u(\cdot,s)+f(u(\cdot,s),v(\cdot,s)) \big) ds, \right. \\
v(\cdot,t)=&\left.e^{D_2 (\Delta-1)t}v_0+\int_0^t e^{D_2 (\Delta-1)t} \big(D_2v(\cdot,s)+g(u(\cdot,s),v(\cdot,s)) \big)ds, \right.
\end{split}
\end{equation}
where $f(u,v)=(a_1-b_1u-c_1v)u$ and $g(u,v)=(a_2-b_2u-c_2v)v$.

\subsection{Preliminaries and local solutions}

We present the local existence and uniqueness of the solution of (\ref{16}) as well as some preliminary results in this part.  Our first result goes as follows.
\begin{theorem}\label{thm21}
Let $\Omega\subset\mathbb{R}^N$, $N\geq1$, be a bounded domain with boundary $\partial \Omega \in C^3$.  Suppose that $\tau>0$ and $\phi\in C^5(\mathbb{R};\mathbb{R})$.

(i)  For any initial data $(u_0,v_0)\in C^0(\Omega)\times W^{1,p}(\Omega)$, $p>N$, system (\ref{16}) has a unique solution $(u(x,t),v(x,t))$ defined on $\bar \Omega\times [0,T_{\max})$ with $0<T_{\max} \leq \infty$ such that $(u(\cdot,t),v(\cdot,t)) \in C^0(\bar \Omega,[0,T_{\max}))\times C^0(\bar \Omega,[0,T_{\max}))$ and $(u,v)\in C^{2,1}(\bar \Omega,[0,T_{\max}))\times C^{2,1}(\bar \Omega,[0,T_{\max}))$.

(ii)  If $\sup_{s\in(0,t)}\Vert (u,v)(\cdot,s) \Vert_{L^\infty}$ is bounded for $t\in(0,T_{\max})$, then $T_{\max}=\infty$, i.e, $(u,v)$ is a global solution to (\ref{16}).  Furthermore, $(u,v)$ is a classical solution such that, $(u,v) \in C^\alpha((0,\infty),C^{2(1-\beta)}(\bar \Omega)\times C^{2(1-\beta)}(\bar \Omega))$ for any $0\leq \alpha \leq \beta \leq 1$.

(iii)  Suppose that $u_0>0$ and $\frac{a_2}{c_2}\geq v_0\geq 0$ on $\bar \Omega$.  Then $u>0$ and $\frac{a_2}{c_2}>v>0$ on $\bar \Omega \times (0,T_{\max})$.
\end{theorem}

\begin{remark}
If $\Omega=(0,L)$, then for any initial data $(u_0,v_0)\in H^1(0,L)\times H^1(0,L)$, (\ref{16}) has a unique solution $(u(x,t),v(x,t))$ defined on $[0,L]\times [0,T_{\max})$ with
$0<T_{\max} \leq \infty$ such that $(u(\cdot,t),v(\cdot,t)) \in C([0,T_{\max}),H^1(0,L)\times H^1(0,L))$ and $(u,v)\in C^{2+\alpha,1+\alpha}_{loc}$, $(x,t)\in [0,1]\times (0,T_{\max})$ with $0<\alpha<\frac{1}{4}$.
\end{remark}

\begin{proof}
Denote $\textbf{w}=(u,v)$ and (\ref{16}) becomes
\begin{equation}\label{22}
\left\{
\begin{array}{ll}
\textbf{w}_t=\nabla \cdot(A(\textbf{w})\nabla \textbf{w})+F(\textbf{w}),~x \in \Omega,~t>0,     \\
\textbf{w}(x,0)=(u_0,v_0),x\in \Omega; \frac{\partial \textbf{w}}{\partial \textbf{n}}=0,~x \in \partial \Omega,t>0,
\end{array}
\right.
\end{equation}
where
 \begin{equation*}A(\textbf{w})=\begin{pmatrix}
 D_1  &  \chi u\phi(v)  \\
  0           &  D_2
  \end{pmatrix} ,~~F(\textbf{w})=\begin{pmatrix}
(a_1-b_1u-c_1v)u \\
(a_2-b_2u-c_2v)v
  \end{pmatrix}.
   \end{equation*}
Since the eigenvalues of $A$ are positive, therefore (\ref{22}) is normally parabolic.  Then \emph{(i)} follows from Theorem 7.3 and 9.3 of \cite{A}.  Moreover, \emph{(ii)} follows from Theorem 5.2 in \cite{A1} since (\ref{22}) is a triangular system.

To prove \emph{(iii)}, we observe that $\underline{u}\equiv0$ and $\underline{v}\equiv0$ are strict sub-solutions and we have from the Strong Maximum Principle and Hopf's boundary point lemma that $u>0$ and $v>0$ on $\bar \Omega$.  On the other hand, $\frac{a_2}{c_2}$ is a super solution and it follows from the Comparison Principle that $v<\frac{a_2}{c_2}$ on $\bar \Omega\times (0,T_{\max})$.  This completes the proof of Theorem \ref{thm21}.
\end{proof}

\begin{lemma}\label{lem22}  Assume that $\Omega$ is a bounded domain in $\mathbb{R}^N$, $N\geq1$.  Let the initial data $(u_0,v_0) \in C^0(\bar \Omega)\times W^{1,p}(\Omega)$, $p>N$, be non-negative and $(u,v)$ be the unique positive solution of (\ref{16}).  Then
\begin{equation}\label{23}
\Vert u(\cdot,t) \Vert_{L^1(\Omega)} \leq C=C(\Vert u_0 \Vert_{L^1},a_1,b_1, \vert \Omega \vert), \forall t\in(0,T_{\max}),
\end{equation}
and
\begin{equation}\label{24}
0\leq v(x,t)\leq C= C(\Vert v_0 \Vert_{\infty},a_2,c_2),\forall t\in(0,T_{\max}).
\end{equation}
\end{lemma}
\begin{proof}
It is equivalent to show that $\int_\Omega u(x,t) dx$ is uniformly bounded in time $t$ since $u(x,t)>0$ according to Theorem \ref{thm21}.  We integrate the first equation in (\ref{16}) over $\Omega$ and have that
\[\frac{d}{dt} \int_\Omega u(x,t)dx=\int_\Omega f(u,v)u dx \leq a_1 \int_\Omega u(x,t) dx-b_1 \int_\Omega u^2(x,t) dx,\]
then we conclude from Gronwall's lemma that
\[\int_\Omega u(x,t)dx\leq e^{-b_1t}  \int_\Omega u_0(x)dx+C \vert \Omega \vert.\]
On the other hand, we denote $\bar{v}(t)$ as the solution of
\begin{equation}\label{25}
\frac{d \bar{v}(t)}{dt}=(a_2-c_2\bar{v})\bar{v},\bar{v}(0)=\max_{\Omega} v_0(x),
\end{equation}
then we solve differential inequality (\ref{25}) and obtain that
 \[\bar{v}(t)\leq \max \Big\{\Big(\frac{2a_2}{c_2}\Big), \Big(\Vert v_0 \Vert_\infty^{-1}+\frac{c_2 t}{2} \Big)^{-1} \Big\}.\]  Obviously, $\bar{v}(t)$ is a super-solution to the second equation of (\ref{16}).  Hence we have that $v(x,t)\leq \bar{v}(t)$ from the Maximum Principle and it implies (\ref{24}).
\end{proof}

Lemma \ref{lem22} provides the $L^\infty$-bound of $v$ and $L^1$-bound of $u$.  To establish the $L^\infty$-bound on $u$, we need to estimate $\Vert \nabla v \Vert_{L^p}$ for some large $p$.  To this end, we shall employ the well-known smoothing properties of the operator $-\Delta+1$, for which \cite{H}, \cite{HW} and \cite{Winkler} are good references.  Applying $\nabla$ on the $v$-equation in (\ref{21}), we have from the embeddings between on the analytic semigroups generated by $-\Delta+1$, for example, Lemma 1.3 in \cite{Winkler}, for all $1\leq p \leq q \leq \infty$, that there exists positive constants $C$ dependent on $\Vert v_0 \Vert _{W^{1,q}(\Omega)}$ and $\Omega$ such that, for all $T\in(0,\infty)$ and any $t\in(0,T)$
\begin{equation}\label{26}
 \Vert v(\cdot,t) \Vert _{W^{1,q}(\Omega)} \leq C\left(1 + \int_0^te^{-\nu(t-s)} (t-s)^{-\frac{1}{2}-\frac{N}{2}(\frac{1}{p}-\frac{1}{q})}   \Vert u(\cdot,s)+1 \Vert_{L^p}   ds\right),
\end{equation}
where $\nu$ is the first Neumann eigenvalue of $-\Delta$ in $\Omega$.  We have applied in (\ref{26}) the uniform boundedness of $\Vert v(\cdot,t)\Vert_{L^\infty}$ in (\ref{24}).

\subsection{Existence of global solutions in one-dimensional domain}

We now proceed to establish the $L^\infty$-estimate of $u$ under the assumption on the initial data $(u_0,v_0)$ in Lemma \ref{22}.  Then the existence of global-in-time classical solutions follows through \emph{(ii)} of Theorem \ref{thm21}.  In particular, we consider (\ref{16}) over one-dimensional domain $\Omega=(0,L)$.  The uniform boundedness of $\Vert u(\cdot,t) \Vert_{L^\infty}$ is a consequence of several lemmas.

\begin{lemma}\label{lem23}
Let $\Omega=(0,L)$ be a bounded interval in $\mathbb{R}^1$.  Assume that the nonnegative initial data $(u_0,v_0)\in H^1(0,L) \times H^1(0,L)$ and suppose that $\phi \in C^5(\mathbb{R},\mathbb{R})$.  Then there exists a positive constant $C$ dependent on the parameters of (\ref{16}) such that, for any $q\in(1,\infty)$
\begin{equation}\label{27}
\Vert v_x(\cdot,t) \Vert_{L^q} \leq C(q), \forall t\in(0,T_{\max}).
\end{equation}
\end{lemma}

\begin{proof}
We choose $p=N=1$ in (\ref{26}) and obtain
\begin{equation}\label{28}
\Vert v_x(\cdot,t) \Vert _{L^q(0,L)} \leq C\left(1+ \int_0^te^{-\nu(t-s)} (t-s)^{-\frac{1}{2}-\frac{1}{2}(1-\frac{1}{q})}   \Vert u(\cdot,s) \Vert_{L^1(0,L)}   ds\right),
\end{equation}
where $C$ depends on $\Vert v_0 \Vert_{W^{1,q}(0,L)}$.  On the other hand, we see that for any $q\in(1,\infty)$
\[\sup_{t\in(0,\infty)}\int_0^te^{-\nu(t-s)} (t-s)^{-1+\frac{1}{2q}} ds<C_0(q),\]
where $C_0$ is a positive constant independent of $q$, hence it follows from (\ref{28}) that
\begin{equation}\label{29}
\Vert  v_x(\cdot,t) \Vert _{L^q(0,L)} \leq C\Big(1+ \sup_{s\in(0,t)} \Vert u(\cdot,s) \Vert_{L^1(0,L)}\Big).
\end{equation}
Then (\ref{27}) is an immediate consequence of (\ref{23}) and (\ref{29}).
\end{proof}

\begin{lemma}\label{lem24}Suppose that the assumptions in Lemma \ref{lem23} hold and let $(u(x,t),v(x,\break t))$ be a positive classical solution of (\ref{16}).  Then for each $p\in(2,\infty)$, there exists a positive constant $C(p)$ such that
\begin{equation}\label{210}
\Vert u(\cdot,t) \Vert_{L^p}\leq C(p),\forall t\in(0,T_{\max}).
\end{equation}
\end{lemma}

\begin{proof}
For $p>2$, we multiply the first equation of (\ref{16}) by $u^{p-1}$ and integrate it over $(0,L)$, then we obtain from the integration by parts and the $L^\infty-$boundness of $v$ in (\ref{24}) that
{\small\begin{align}%\label{211}
&\frac{1}{p}\frac{d}{dt}\int_0^L u^p dx=\int_0^L u^{p-1}u_t dx  \nonumber \\
=&\int_0^L u^{p-1} (D_1 u_x+\chi u\phi(v) v_x)_x dx+\int_0^L (a_1-b_1u-c_1v)u^p dx      \nonumber \\
=& -\frac{4D_1(p-1)}{p^2}\int_0^L \vert  (u^\frac{p}{2})_x \vert^2 dx-\chi\int_0^L (u_x^{p-1})u\phi(v) v_x dx+\int_0^L (a_1-b_1u-c_1v)u^p dx \nonumber \\
\leq& -\frac{4D_1(p-1)}{p^2}\int_0^L \vert (u^\frac{p}{2})_x \vert^2 dx+\frac{2(p-1)\chi}{p}\int_0^L u^\frac{p}{2} \phi(v) (u^\frac{p}{2})_x v_x dx\nonumber\\&\ \ \ \ \ \ \ \ \ \ -C_1\int_0^L u^{p+1}dx+C_0 \nonumber  %\\
\end{align}}{\small\begin{align}\label{211}\leq& -\frac{4D_1(p-1)}{p^2}\int_0^L \vert (u^\frac{p}{2})_x \vert^2 dx+C_2\int_0^L u^\frac{p}{2} \vert (u^\frac{p}{2})_x\vert~ \vert v_x\vert dx -C_1\int_0^L  u^{p+1}dx+C_0   \\
\leq& -\frac{4D_1(p-1)}{p^2}\Vert (u^\frac{p}{2})_x \Vert^2_{L^2}+C_2\Vert (u^\frac{p}{2})_x \Vert_{L^2} \Vert u^\frac{p}{2} \Vert_{L^{2(p+1)/p}} \Vert v_x\Vert_{L^{2(p+1)}}-C_1\Vert u \Vert^{p+1}_{L^{p+1}}+C_0 \nonumber \\ \nonumber
=& -\frac{4D_1(p-1)}{p^2}\Vert (u^\frac{p}{2})_x \Vert^2_{L^2}+C_2\Vert (u^\frac{p}{2})_x \Vert_{L^2}\Vert u \Vert^\frac{p}{2}_{L^{p+1}} \Vert v_x\Vert_{L^{2(p+1)}}-C_1\Vert u \Vert^{p+1}_{L^{p+1}}+C_0, \nonumber \\ \nonumber
\end{align}}%\begin{equation}\label{211}
%\begin{split}
%&\frac{1}{p}\frac{d}{dt}\int_0^L u^p dx=\int_0^L u^{p-1}u_t dx\\
%=&\int_0^L u^{p-1} (D_1 u_x+\chi u\phi(v) v_x)_x dx+\int_0^L (a_1-b_1u-c_1v)u^p dx \\
%=& -\frac{4D_1(p-1)}{p^2}\int_0^L \vert  (u^\frac{p}{2})_x \vert^2 dx-\chi\int_0^L (u_x^{p-1})u\phi(v) v_x dx+\int_0^L (a_1-b_1u-c_1v)u^p dx  \\
%\leq&-\frac{4D_1(p-1)}{p^2}\int_0^L \vert (u^\frac{p}{2})_x \vert^2 dx+\frac{2(p-1)\chi}{p}\int_0^L u^\frac{p}{2} \phi(v) (u^\frac{p}{2})_x v_x dx-C_1\int_0^L u^{p+1}dx \\
%\leq&-\frac{4D_1(p-1)}{p^2}\int_0^L \vert (u^\frac{p}{2})_x \vert^2 dx+C_2\int_0^L u^\frac{p}{2} \vert (u^\frac{p}{2})_x\vert~ \vert v_x\vert dx -C_1\int_0^L  u^{p+1}dx   \\
%\leq&-\frac{4D_1(p-1)}{p^2}\Vert (u^\frac{p}{2})_x \Vert^2_{L^2}+C_2\Vert (u^\frac{p}{2})_x \Vert_{L^2} \Vert u^\frac{p}{2} \Vert_{L^{2(p+1)/p}} \Vert v_x\Vert_{L^{2(p+1)}}-C_1\Vert u \Vert^{p+1}_{L^{p+1}}  \\
%\leq&-\frac{4D_1(p-1)}{p^2}\Vert (u^\frac{p}{2})_x \Vert^2_{L^2}+C_2\Vert (u^\frac{p}{2})_x \Vert_{L^2}\Vert u \Vert^\frac{p}{2}_{L^{p+1}} \Vert v_x\Vert_{L^{2(p+1)}}-C_1\Vert u \Vert^{p+1}_{L^{p+1}},  \\
%\end{split}
%\end{equation}
where $C_0$, $C_1$ and $C_2$ are positive constants that depend on $p$ and may vary from line to line.  We have applied the H\"older's inequality in the last inequality and used the fact $\Vert u^\frac{p}{2} \Vert_{L^{\frac{2(p+1)}{p}}}=\Vert u \Vert^\frac{p}{2}_{L^{p+1}}$ in the last identity of (\ref{211}).  Moreover, in light of (\ref{27}), we can further estimate (\ref{211}) by the Young's inequality
\[C_2\Vert (u^\frac{p}{2})_x \Vert_{L^2}\Vert u \Vert^\frac{p}{2}_{L^{p+1}} \Vert v_x\Vert_{L^{2(p+1)}}\leq \frac{4D_1(p-1)}{p^2}\Vert (u^\frac{p}{2})_x \Vert^2_{L^2} +C_3 \Vert u \Vert^p_{L^{p+1}}, \]
where $C_3$ is a positive constant that may depend on $p$, then we see that (\ref{211}) implies that
\begin{equation}\label{212}
\frac{1}{p}\frac{d}{dt}\int_0^L u^p dx \leq C_3\Vert u \Vert^p_{L^{p+1}}-C_1 \Vert u\Vert^{p+1}_{L^{p+1}}+C_0.
\end{equation}
Denoting $y_p(t)=\int_0^L u^p dx$, we apply the H\"older's in (\ref{212}) to have that
\[y_p'(t)\leq -\tilde{C}_1y_p^\frac{p+1}{p}+\tilde{C}_0,~y_p(0)=\Vert u_0 \Vert^p_{L^p}.\]
Solving this differential inequality gives us that $y_p(t)\leq C(p)$ for all $t\in (0,T_{\max})$.  This completes the proof of (\ref{210}).
\end{proof}
By taking $p$ sufficiently large but fixed in (\ref{26}), we see that the following result is a quick implication of Lemma \ref{lem24}.
\begin{corollary}\label{cor1}
Under the conditions in Lemma \ref{lem22}, we have
\begin{equation}\label{213}
\Vert v_x(\cdot,t) \Vert_{L^\infty} \leq C,\forall t\in(0,T_{\max}),
\end{equation}
where $C$ is a positive constant dependent on the parameters in (\ref{16}).
\end{corollary}

Now we are ready to present our first main result concerning the global existence of bounded positive classical solutions to (\ref{16}).
\begin{theorem}\label{thm25}
Let $\Omega=(0,L)$ and $\phi \in C^5(\mathbb{R},\mathbb{R})$.  Then for any positive initial data $(u_0,v_0)\in H^1(0,L)\times H^1(0,L)$, system (\ref{16}) has a unique bounded positive solution $(u(x,t),v(x,t))$ defined on $[0,L]\times [0,\infty)$ such that $(u(\cdot,t),v(\cdot,t)) \in C([0,\infty),H^1(0,L)\times H^1(0,L))$ and $(u,v)\in C^{2+\alpha,1+\alpha}_{loc}([0,1]\times (0,\infty))$ for any $0<\alpha<\frac{1}{4}$.
\end{theorem}

\begin{proof}To prove this Theorem, we need to show that $\Vert u(\cdot,t)\Vert_{L^\infty}$ is uniformly bounded for all $t\in(0,T_{\max})$, then we must have that $T_{\max}=\infty$ and Theorem \ref{thm25} follows from \emph{(ii)} of Theorem \ref{thm21} and Corollary \ref{cor1}.  Moreover, one can apply parabolic boundary $L^p$ estimates and Schauder estimates to show that $u_t,v_t$ and all spatial partial derivatives of $u$ and $v$ up to order two are bounded on $[0,L] \times (0,\infty)$, and then $(u,v)$ have the regularities as stated in Theorem \ref{thm25}.

Without loss of our generality, we assume in (\ref{24}) and (\ref{213}) that $\Vert \phi(v) \Vert_{L^\infty}\leq 1 \text{ and } \Vert v_x \Vert_{L^\infty} \leq 1$.  Through the same calculations that lead to the proof of Lemma \ref{lem24} and using the fact $u\geq0,v\geq0$, we obtain
\begin{eqnarray}\label{214}
&&\frac{1}{p}\frac{d}{dt}\!\!\int_0^L\!\!\!\! u^p dx\leq \!\!-\frac{4D_1(p\!\!-\!\!1)}{p^2}\!\!\int_0^L \!\!\!\! \vert (u^\frac{p}{2})_x \vert^2 dx\!\!+\!\!\frac{2(p\!\!-\!\!1)\chi}{p}\!\!\int_0^L \!\!\!\!u^\frac{p}{2} \vert (u^\frac{p}{2})_x \vert dx\!\!+\!\!a_1\!\!\!\!\int_0^L\!\! u^p dx  \nonumber \\
&& \leq-\frac{4D_1(p-1)}{p^2}\!\!\int_0^L \!\!\!\!\vert (u^\frac{p}{2})_x \vert^2 dx\!\!+\!\!\frac{(p-1)\chi}{p}\!\!\int_0^L\!\!\!\! \Big(\frac{2D_1}{p\chi}\vert (u^\frac{p}{2})_x \vert^2\!\!+\!\!\frac{p\chi}{2D_1}u^p \Big) dx\!\!+\!\!a_1\!\!\int_0^L\!\!\!\! u^p dx\nonumber \\
&& \leq-\frac{2D_1(p-1)}{p^2}\!\!\int_0^L \!\!\vert (u^\frac{p}{2})_x \vert^2 dx+\Big(\frac{(p-1)\chi^2}{2D_1}+a_1\Big)\!\!\int_0^L\!\! u^p dx,
\end{eqnarray}
where we have used a Young's inequality in third line of (\ref{214}).  To estimate $\int_0^L u^p$, we shall apply in (\ref{214}) the following estimate (P. 63 in \cite{LSU} and Corollary 1 in \cite{CKWW} with $d=1$) due to Gagliardo-Ladyzhenskaya-Nirenberg inequality and the Young's inequality that, for any $u\in H^1(0,L)$ and any $\epsilon>0$
\[\Vert u^\frac{p}{2}\Vert^2_{L^2(0,L)}\leq \epsilon \Vert (u^\frac{p}{2})_x \Vert^2_{L^2(0,L)}+K \big(1+ \epsilon^{-\frac{1}{2}}\big)\Vert u^\frac{p}{2}\Vert^2_{L^1(0,L)},\]
where $K$ only depends on $L$.  Choosing $\epsilon=\frac{2D_1^2(p-1)}{p^2((p-1)\chi^2+2a_1D_1)}$ such that $\frac{D_1(p-1)}{p^2\epsilon}=\frac{(p-1)\chi^2}{2D_1}+a_1$, we obtain that
\begin{eqnarray}\label{215}
&&\Big(\frac{(p-1)\chi^2}{2D_1}+a_1\Big)\!\!\int_0^L\!\! u^p dx
= \frac{2D_1 (p-1)}{p^2\epsilon}  \!\!\int_0^L\!\! u^p dx-\Big(\frac{(p-1)\chi^2}{2D_1}+a_1\Big)\!\!\int_0^L \!\!u^p dx \nonumber   \\
\!\!\!\!\!\!\!\!\!\!\leq\!\!\!\!\!\!\!\!\!\! && \frac{2D_1 (p-1)}{p^2} \int_0^L \vert (u^\frac{p}{2})_x \vert^2 dx+\frac{2D_1(p-1)K \big(1+ \epsilon^{-\frac{1}{2}}\big)}{p^2\epsilon}\Big(\int_0^L u^\frac{p}{2} dx\Big)^2  \\
&&-\Big(\frac{(p-1)\chi^2}{2D_1}+a_1\Big)\int_0^L u^p dx.\nonumber\\  \nonumber
\end{eqnarray}
It follows from (\ref{214}) and (\ref{215}) that
\[\frac{d}{dt}\int_0^L u^p dx\leq - \frac{p(p-1)\chi^2}{2D_1} \int_0^L u^p dx +\frac{2D_1(p-1)K \big(1+ \epsilon^{-\frac{1}{2}}\big)}{p \epsilon}\Big(\int_0^L u^\frac{p}{2} dx\Big)^2,p\geq2.\]
On the other hand, we can choose $p_0$ large such that $\epsilon=\frac{2D_1^2(p-1)}{p^2\big((p-1)\chi^2+2a_1D_1\big)}>\big(\frac{D_1}{p\chi}\big)^2$ and $\frac{1+\epsilon^{-\frac{1}{2}}}{\epsilon}<\frac{1+\frac{p\chi}{D_1}}{D_1^2/p^2\chi^2}$ for all $p\geq p_0$, then we have that
\begin{equation}\label{216}
\frac{d}{dt}\!\!\int_0^L\!\! u^p dx\leq - \frac{p(p-1)\chi^2}{2D_1} \!\!\int_0^L\!\! u^p dx +\frac{2p(p-1)K\chi^2(1+\frac{p\chi}{D_1})}{D_1}\Big(\!\!\int_0^L\!\! u^\frac{p}{2}dx\Big)^2.
\end{equation}
Denote $\kappa=\frac{p(p-1)\chi^2}{2D_1}$.  For each $T\in(0,\infty)$, we solve the differential inequality (\ref{216}) for all $t\in(0,T)$ and obtain that
\begin{eqnarray}\label{217}
\int_0^L u^p dx \!\!\!\!\! &\leq&\!\!\!\!\!  e^{-\kappa t} \int_0^L u_0^p dx+\frac{2p(p-1)K\chi^2}{D_1}\Big(1+\frac{p\chi}{D_1}\Big)\int_0^t e^{-\kappa(t-s)}\Big(\int_0^L u^\frac{p}{2} dx\Big)^2ds \nonumber\\
\!\!\!\!\! &\leq&\!\!\!\!\!  \int_0^L u_0^p dx+4K\left(1+\frac{p\chi}{D_1}\right) \sup_{t\in(0,T)}\Big(\int_0^L u^\frac{p}{2}dx\Big)^2,p\geq p_0 \\ \nonumber
\end{eqnarray}
We now employ the Moser-Alikakos iteration \cite{A0} to establish $L^\infty$-estimate of $u$.  To this end, we introduce the function
\[M(p)=\max\big\{\Vert u_0 \Vert_{L^\infty},~\sup_{t\in(0,T)}\Vert u(\cdot,t) \Vert_{L^p}  \big\},\]
then it follows from (\ref{217}) that
\[M(p)\leq  \left(4K+\frac{4Kp\chi}{D_1}\right)^\frac{1}{p}M(p/2),\forall p>p_0.\]
Taking $p=2^i$, $i=1,2,...$ and choosing $p_0=2^{i_0}$ for $i_0$ large, we have
\begin{eqnarray}\label{218}
M(2^i)\!\!\!\!\! &\leq&\!\!\!\!\! \Big(\!4K\!+\!\frac{4K\chi 2^i}{D_1}\!\Big)^{2^{-i}}\!\!\!\!M(2^{i-1})\!\leq\! \Big(\!4K\!+\!\frac{4K\chi 2^i}{D_1}\!\Big)^{2^{-i}} \!\!\!\Big(\!4K\!+\!\frac{4K\chi 2^{i-1}}{D_1}\!\Big)^{2^{-(i-1)}}\!\!\!\!\!M(2^{i-2}) \nonumber\\
\!\!\!\!\! &\leq&\!\!\!\!\! M(2^{i_0}) \prod_{j=i_0+1}^i \Big(4K+\frac{4K\chi 2^j}{D_1}\Big)^{2^{-j}}\leq M(2^{i_0}) \prod_{j=i_0+1}^i \Big(4K2^j+\frac{4K\chi 2^j}{D_1}\Big)^{2^{-j}} \nonumber \\
\!\!\!\!\! &\leq&\!\!\!\!\!  M(2^{i_0}) \Big(4K+\frac{4K\chi }{D_1}\Big) \prod_{j=i_0+1}^i (2^j)^{2^{-j}} \leq M(2^{i_0}) \Big(4K+\frac{4K\chi }{D_1}\Big) 2^{\sum\limits_{j=i_0+1}^i j2^{-j}} \\
\!\!\!\!\! &\leq&\!\!\!\!\!  C\Big(1+\frac{\chi }{D_1}\Big) M(2^{i_0}), \nonumber
\end{eqnarray}
where $C$ is a constant that only depends on $L$ and $M(2^{i_0})$ is bounded for all $t\in(0,\infty)$ in light of (\ref{23}) and (\ref{210}).  Sending $i\rightarrow \infty$ in (\ref{218}), we finally conclude from Lemma \ref{lem24} and (\ref{218}) that
\begin{equation}\label{219}
\Vert u(\cdot,t) \Vert_{L^\infty}\leq C,\forall t\in[0,\infty),
\end{equation}
and this completes the proof of Theorem \ref{thm25}.
\end{proof}

\subsection{Global solutions of parabolic-elliptic system in $N$-dimensional domain}

In this section, we establish the global existence of (\ref{16}) for $\Omega$ being a bounded multi-dimensional domain in $\mathbb{R}^N$, $N\geq 2$.  In particular, we consider model (\ref{16}) with $\tau=0$, which approximates an competition relationship that $v$ diffuses much faster than $u$.  Same as the analysis for the 1D domain, our global existence result is a consequence of several Lemmas.  Our first step is to establish the $\Vert u(\cdot,t) \Vert_{L^p}$-bounds for $p$ large in the following lemma.
\begin{lemma}\label{lem26}
Let $\Omega$ be a bounded domain in $\mathbb{R}^N$, $N\geq2$ with $\partial \Omega\in C^3$ and suppose that $\tau=0$ in (\ref{16}).  Moreover, we assume that $\phi \in C^5(\mathbb{R},\mathbb{R})$ and $\phi'(v)\leq0$ for all $v\geq0$.  Let $(u(x,t),v(x,t))$ be a positive classical solution of (\ref{16}).  Then there exists $C^*>0$ dependent of $a_2$, $c_2$ and $\Vert v_0 \Vert_{L^\infty}$ such that if $\frac{b_1D_2}{b_2\chi}>C^*$, for any $p>\max\{\frac{N}{2},1\}$,
\begin{equation}\label{220}
\Vert u(\cdot,t) \Vert_{L^p}\leq C(p),\forall t\in(0,T_{\max}),
\end{equation}
where $C=C(p)$ is a positive constant that also depends on $N$ and $\Omega$.
\end{lemma}

\begin{proof}
For any $p>\max\{\frac{N}{2},1\}$, we test the first equation of (\ref{16}) by $u^{p-1}$ and then integrate it over $\Omega$ by parts.  Then we have from $\phi'(v)\leq0$ and $L^\infty-$boundness of $v$ in (\ref{24}) that
\begin{eqnarray}%\label{221}
&&\frac{1}{p}\frac{d}{dt}\int_\Omega u^p  \!\!=\!\!\int_\Omega\!\! u^{p-1}u_t\!\!=\!\!\int_\Omega\!\! u^{p-1}\nabla \cdot (D_1 \nabla u+\chi u\phi(v)\nabla v)\!\!+\!\!\int_\Omega \!\! (a_1\!\!-\!\!b_1u\!\!-\!\!c_1v)u^p   \nonumber \\
\!\!\!\!\! &=&\!\!\!\!\!  -\frac{4D_1(p-1)}{p^2}\int_\Omega \vert \nabla u^\frac{p}{2} \vert^2 dx-\chi\int_\Omega \nabla u^{p-1}  u\phi(v)\nabla v +\int_\Omega (a_1-b_1u-c_1v)u^p dx \nonumber% \\
\end{eqnarray}\begin{eqnarray}\label{221}
\!\!\!\!\! &\leq&\!\!\!\!\!  -\chi\int_\Omega \nabla u^{p-1} u\phi(v)\nabla v+\int_\Omega (a_1-b_1u-c_1v)u^p  \\
\!\!\!\!\! &=&\!\!\!\!\!  \frac{(p-1)\chi}{p}\int_\Omega u^p (\phi'(v)\vert\nabla v\vert ^2+\phi(v)\Delta v)+\int_\Omega (a_1-b_1u-c_1v)u^p  \nonumber \\
\!\!\!\!\! &\leq&\!\!\!\!\! \frac{(p-1)\chi}{p}\int_\Omega u^p \phi(v)\Delta v+\int_\Omega (a_1-b_1u-c_1v)u^p.  \nonumber \\ \nonumber
\end{eqnarray}
Since $\tau=0$, we can substitute $\Delta v=-(a_2-b_2u-c_2v)v/D_2$ into (\ref{221}) and collect
\begin{eqnarray}\label{222}
&&\frac{1}{p}\frac{d}{dt}\int_\Omega u^p dx\leq\frac{(p-1)\chi}{p}\int_\Omega u^p \phi(v)\Delta v+\int_\Omega (a_1-b_1u-c_1v)u^p  \\ \nonumber
\!\!\!\!\! &=&\!\!\!\!\! -\frac{(p-1)\chi}{pD_2}\int_\Omega u^p \phi(v)(a_2-b_2u-c_2v)v+\int_\Omega (a_1-b_1u-c_1v)u^p  \nonumber \\
\!\!\!\!\! &\leq&\!\!\!\!\! -\Big(b_1-\frac{(p-1)\chi \Vert \phi(v) \Vert_{L^\infty} b_2}{pD_2}\Big)\int_\Omega u^{p+1}+ \Big(a_1+\frac{(p-1)\chi c_2}{pD_2}\Vert \phi(v)v^2 \Vert_{L^\infty}\Big)\int_\Omega u^p, \nonumber \\ \nonumber
\end{eqnarray}
For $\frac{b_1D_2}{b_2\chi}$ being large, we have that (\ref{222}) implies
\begin{equation}\label{223}
\frac{d}{dt}\int_\Omega u^p dx \leq -C_1\int_\Omega u^{p+1} dx+C_2\int_\Omega u^p dx,
\end{equation}
where $C_1$ and $C_2$ are positive constants independent of $p$.  Then (\ref{220}) is an immediate consequence of (\ref{223}).
\end{proof}

\begin{corollary}\label{cor2}
Under the conditions in Lemma \ref{lem26}, there exists a positive constant $C$ such that
\begin{equation}\label{224}
\Vert \nabla v(\cdot,t)\Vert_{L^\infty(\Omega)}<C,\forall t\in(0,T_{\max}).
\end{equation}
\end{corollary}
\begin{proof}
(\ref{224}) is a quick implication of Lemma \ref{lem26} and (\ref{26}).
\end{proof}

Now we prove the $L^\infty$-bound of $u$ and present our main result on the existence of global solutions for the parabolic-elliptic model.
\begin{theorem}\label{thm27}
Suppose that $\tau=0$ in (\ref{16}) and the initial data $(u_0,v_0)\in C^(\bar \Omega)\times W^{1,p}$, $p>N$, $u$, $v\geq0$, $\not \equiv0$.  Under the assumptions in Lemma \ref{lem26}, (\ref{16}) admits a unique globally bounded classical solution $(u(x,t),v(x,t))$ for all $(x,t)\in\Omega \times (0,\infty)$; both $u$ and $v$ are nonnegative on $\Omega$ for all $t\in(0,\infty)$.
\end{theorem}
\begin{proof}
Similar as the proof of Theorem \ref{thm21}, we can show the existence of unique classical local solutions $(u(x,t),v(x,t))$ for $(x,t)\in \Omega \times (0,T_{\max})$.   Moreover, by the same Moser-Alikakos iteration estimate that leads to the proof of (\ref{219}), we can prove that $\Vert u(\cdot,t) \Vert_{L^\infty(\Omega)}$ for all $t\in(0,T_{\max})$, therefore $T_{\max}=\infty$ and the classical solutions are global.  Finally, the nonnegativity of the classical solutions follows from the Strong Maximum Principle and Hopf's boundary point lemma.
\end{proof}

We should mention that the largeness assumption on $b_1D_2/b_2\chi$ is made out of the mathematical analysis and both the self-competition rate $b_1$ and diffusion rate $D_2$ have the stabilizing effect to exclude finite time blow-ups.  There is clearly global existence in (\ref{16}) for $\chi=0$ for bounded domain $\Omega$ in arbitrary dimensions, then we expect this is still true for $\chi$ being small from the standard perturbation theory.  On the other hand, it is also worth pointing out that, the restriction on $b_1D_1/b_2\chi$ is necessary in this sense that the largeness of $b_1$ may be insufficient to guarantee the existence of global solutions for the general multi-dimensional parabolic-parabolic system (\ref{14}).

\section{Existence and stability of nonconstant positive steady states}\label{section3}

From the viewpoint of mathematical modeling, it is interesting and important to understand whether or not two competing species can form spatial segregation eventually.  For this purpose, we focus on the stationary system of (\ref{16}) and study the existence of nonconstant positive steady states that exhibit striking patterns such as boundary or transition layers.  In particular, we consider the species competition model over a one-dimensional habitat and we study the stationary solutions to the following strongly-coupled system in this section,
\begin{equation}\label{31}
\left\{
\begin{array}{ll}
u_t=(D_1 u'+\chi u\phi(v) v')'+(a_1-b_1u-c_1v)u ,& x \in (0,L),t>0,\\
v_t=D_2  v''+(a_2-b_2u-c_2v)v ,&x \in (0,L), t>0,\\
u(x,0)=u_0(x),~v(x,0)=v_0(x),&x\in(0,L),\\
u'(x)=v'(x)=0, & x=0,L,t>0.
\end{array}
\right.
\end{equation}
We remind the reader that classical solutions of (\ref{31}) exist globally and are uniformly bounded for all $t\in(0,\infty)$.

Unlike random movements, directed dispersals have the effect of destabilizing the spatially homogeneous solutions.  Then spatially inhomogeneous solutions may arise through bifurcation as the homogeneous one becomes unstable.  To study the regime under which spatial patterns arise in (\ref{31}), we first implement the standard linearized stability analysis at $(\bar{u},\bar{v})$.  Let $(u,v)=(\bar u,\bar v)+(U,V)$, where $U$ and $V$ are small perturbations from ($\bar u,\bar v$), then we arrive at the following system of $(U,V)$
\begin{equation*}
\left\{
\begin{array}{ll}
U_t\approx  (D_1   U' +\chi \bar u \phi(\bar v)V')'-b_1\bar{u}U-c_1\bar u V ,  & x \in(0,L),~t>0,\\
V_t \approx D_2  V'' -b_2 \bar vU-c_2\bar vV, & x \in(0,L),~t>0,\\
U'(x)=V'(x)=0, & x=0,L,t>0.
\end{array}
\right.
\end{equation*}
According to the standard linearized stability analysis, we see that the stability of $(\bar{u},\bar{v})$ can be determined by the eigenvalues of the following matrix,
 \begin{equation}\label{32}
\begin{pmatrix}
 -D_1 \Lambda -b_1\bar{u} & -\chi\bar{u} \phi(\bar{v})\Lambda-c_1\bar u \\
 -b_2 \bar{v}              & -D_2 \Lambda-c_2\bar v
  \end{pmatrix} ,
   \end{equation}
where $\Lambda=\big(\frac{k\pi}{L}\big)^2>0$, $k=1,2,...$, are the $k$-th eigenvalues of $-\frac{d^2}{dx^2}$ on $(0,L)$ under the Neumann boundary conditions.  We have the following result on the linearized instability of $(\bar{u},\bar{v})$ to (\ref{31}).
\begin{proposition}\label{prop1}
The constant solution $(\bar{u},\bar{v})$ of (\ref{31}) is unstable if
\begin{equation}\label{33}
\chi>\chi_0=\min_{k\in \mathbb N^+} \frac{\big( D_1(\frac{k\pi}{L})^2+b_1\bar{u} \big)\big(D_2(\frac{k\pi}{L})^2+c_2\bar{v} \big)-b_2c_1\bar{u}\bar{v} }{b_2(\frac{k\pi}{L})^2\phi(\bar{v})\bar{u}\bar{v}}.
\end{equation}
\end{proposition}

\begin{proof}For each $k\in \mathbb N^+$, the stability matrix (\ref{32}) becomes
 \begin{equation}\label{34}
 H_k=
\begin{pmatrix}
 -D_1 \big(\frac{k\pi}{L}\big)^2 -b_1\bar{u} & -\chi\bar{u}\phi(\bar{v})\big(\frac{k\pi}{L}\big)^2-c_1\bar u \\
 -b_2 \bar{v}              & -D_2 \big(\frac{k\pi}{L}\big)^2-c_2\bar v
  \end{pmatrix} ,
   \end{equation}
Then $(\bar{u},\bar{v})$ is unstable if $H_k$ has an eigenvalue with positive real part for some $k \in \mathbb N^+$.  It is easy to see that the characteristic polynomial of (\ref{34}) takes the form
\[p(\lambda)=\lambda^2+\text{Tr} \lambda+\text{Det},\]
where \[\text{Tr}=\big(D_1+D_2\big)\Big(\frac{k\pi}{L}\Big)^2+b_1\bar u+c_2\bar v>0,\text{~and~}\]
\[\text{Det}=\big(D_1\big(\frac{k\pi}{L}\big)^2+b_2\bar{u}\big)\big(D_2\big(\frac{k\pi}{L}\big)^2 +c_2\bar v\big)-\big(\chi\bar{u}\phi(\bar{v})\big(\frac{k\pi}{L}\big)^2+c_1\bar u\big)b_2 \bar{v},\]
then $p(\lambda)$ has a positive root if and only if $p(0)=\text{Det}<0$.  Hence (\ref{33}) readily follows and this finishes the proof of the proposition.
\end{proof}

$(\bar u,\bar v)$ changes its stability as $\chi$ cross $\chi_0$.  We note that in the strong competition case $\frac{b_1}{b_2}<\frac{a_1}{a_2}<\frac{c_1}{c_2}$, $\chi_0$ in (\ref{33}) is negative if $D_1$ and $D_2$ are sufficiently small.  It is well known that $(\bar u,\bar v)$ is unstable if $\chi=0$ in (\ref{31}) in this case.  Moreover, the semi-steady states $(\frac{a_1}{b_1},0)$ and $(0,\frac{a_2}{c_2})$ are global attractors.   We also want to remark that Proposition \ref{prop1} holds for multi-dimensional domain $\Omega\subset \mathbb{R}^N$ with $N\geq2$, with $\big(\frac{k\pi}{L}\big)^2$ replaced by the $k$-th eigenvalue of $-\Delta$ under the Neumann boundary condition.

\subsection{Positive solutions through bifurcation}

The linearized instability of $(\bar u,\bar v)$ in (\ref{31}) is insufficient to guarantee the existence of spatially inhomogeneous steady states.  As we have shown above, the advection term $\chi u \phi(v) \nabla v$ has the effect of destabilizing $(\bar u,\bar v)$, which becomes unstable if $\chi$ surpasses $\chi_0$.  Then we are concerned if a stable spatially inhomogeneous steady states of (\ref{31}) may emerge through bifurcations as $\chi$ increases.  Clearly, the emergence of spatially inhomogeneous solutions is due to the effect of large advection rate $\chi$ and we refer this as advection-induced patterns in the sense of Turing's instability.

In this section, we carry out bifurcation analysis to seek non-constant positive solutions to the following stationary reaction-advection-diffusion system,
\begin{equation}\label{35}
\left\{
\begin{array}{ll}
(D_1 u'+\chi u\phi(v) v')'+(a_1-b_1u-c_1v)u=0,& x \in (0,L),\\
D_2  v''+(a_2-b_2u-c_2v)v=0,&x \in (0,L), \\
u'(x)=v'(x)=0, & x=0,L.
\end{array}
\right.
\end{equation}
In order to apply the bifurcation theory of Crandall and Rabinowitz \cite{CR}, we first introduce the Hilbert space
\begin{equation}\label{36}
\mathcal{X}=\{w \in H^2(0,L) ~\vert w'(0)=w'(L)=0\}.
\end{equation}
Then by taking $\chi$ as the bifurcation parameter, we rewrite (\ref{35}) in the abstract form
\[\mathcal{F}(u,v,\chi)=0,~(u,v,\chi) \in \mathcal{X}  \times \mathcal{X} \times \mathbb{R},\]
where
\begin{equation}\label{37}
\mathcal{F}(u,v,\chi) =\left(
 \begin{array}{c}
(D_1 u'+\chi u\phi(v) v')'+(a_1-b_1u-c_1v)u\\
D_2  v''+(a_2-b_2u-c_2v)v
 \end{array}
 \right).
 \end{equation}

It is easy to see that $\mathcal{F}(\bar{u},\bar{v},\chi)=0$ for any $\chi \in \mathbb{R}$ and $\mathcal{F}: \mathcal{X}  \times \mathbb{R}  \times
\mathbb{R}  \rightarrow \mathcal{Y} \times \mathcal{Y}$ is analytic for $\mathcal{Y}=L^2(0,L)$;  moreover, we can have from straightforward calculations
that, for any fixed $(u_0,v_0) \in \mathcal{X} \times \mathcal{X}$, the Fr\'echet derivative of $\mathcal{F}$ is given by\\
\vspace{-2mm}
\hspace{1cm}$D_{(u,v)}\mathcal{F}(u_0,v_0,\chi)(u,v)$
\begin{equation}\label{38}
=\left(
\begin{array}{c}
D_1u''+\chi(\phi(v_0)uv'_0+\phi(v_0)u_0v'+\phi'(v_0)u_0v'_0v)'+D\mathcal{F}_2\\
D_2 v''+(a_2-b_2u_0-2c_2v_0)v-b_2uv_0
\end{array}
\right),
\end{equation}
where $D\mathcal{F}_2=(a_1-2b_1u_0-c_1v_0)u-c_1u_0v$.

Denoting $\textbf{u}=(u,v)^\text{T}$, we can write (\ref{38}) as
\[D_{(u,v)}\mathcal{F}(u_0,v_0,\chi)(u,v)=\textbf{A}_0(\textbf{u}) \textbf{u}''+\textbf{A}_1(\textbf{u}) \textbf{u}'+\textbf{F}_0(x,\textbf{u}),\]
where
$$\textbf{A}_0(\textbf{u})=\begin{pmatrix}
D_1 & \chi u_0\phi(v_0)  \\
0 &  D_2
\end{pmatrix},$$
$$\textbf{A}_1(\textbf{u})=\begin{pmatrix}
\chi v'_0\phi(v_0) & \chi\big((u_0 \phi(v_0))'+u_0 v'_0\phi'(v_0)\big)  \\
0 &  0
\end{pmatrix}$$
and
$\textbf{F}_0=\begin{pmatrix}
\chi (v'_0\phi(v_0))'u+\chi (u_0 v'_0\phi'(v_0))'v+(a_1-2b_1u_0-c_1v_0)u-c_1u_0v \\
(a_2-b_2u_0-2c_2v_0)v-b_2uv_0
\end{pmatrix}$, then we see that operator (\ref{38}) is elliptic; moreover it is strongly elliptic and satisfies the Agmon's condition according to Remark 2.5 of case 2 with $N=1$ in Shi and Wang \cite{SW}.  Hence by Theorem 3.3 and Remark 3.4 of \cite{SW}, $D_{(u,v)}\mathcal{F}(u_0,v_0,\chi)$ is a Fredholm operator with zero index.

For bifurcations to occur at $(\bar u,\bar v,\chi)$, we need the Implicit Function Theorem to fail on $\mathcal{F}$, therefore $\mathcal{N}\big(D_{(u,v)}\mathcal{F}(\bar{u},\bar{v},\chi)\big) \neq \{0\}$ and this implies that there exists a nontrivial solution $(u,v)$ to the following problem
\begin{equation}\label{39}
\left\{
\begin{array}{ll}
D_1 u''+\chi\bar{u}\phi(\bar{v})v''-b_1\bar{u}u-c_1\bar{u}v=0,&x\in(0,L),\\
D_2  v''-b_2\bar{v}u-c_2\bar{v}v=0,&x\in(0,L),\\
u'(x)=v'(x)=0,&x=0,L.
\end{array}
\right.
\end{equation}
We write $u$ and $v$ into their eigen-expansions
\[u(x)=\sum_{k=0}^\infty t_k \cos \frac{k\pi x}{L},~v(x)=\sum_{k=0}^\infty s_k \cos \frac{k\pi x}{L},\]
and substitute them into (\ref{39}) to obtain that
\begin{equation}\label{310}
\begin{pmatrix}
-D_1 \big(\frac{k \pi}{L}\big)^2-b_1\bar{u} & -\chi \big(\frac{k \pi}{L}\big)^2 \bar{u}\phi(\bar{v})-c_1\bar{u}   \\
-b_2\bar{v} & -D_2\big(\frac{k \pi}{L}\big)^2-c_2\bar{v}
\end{pmatrix}
\begin{pmatrix}
t_k\\
~~\\
s_k
\end{pmatrix}=\begin{pmatrix}
0\\
~~\\
0
\end{pmatrix}.
\end{equation}
We want to establish the condition such that (\ref{310}) admits nonzero solutions $(t_k,s_k)$ for some $k\in N$.  $k=0$ can be easily ruled out, otherwise we must have from the fact $B\neq C$ that $t_1=s_1=0$.  For $k\in \mathbb N^+$, (\ref{310}) has nonzero solutions if and only if its coefficient matrix is singular, which implies that
\begin{equation}\label{311}
\chi=\chi_k=\frac{\big( D_1(\frac{k\pi}{L})^2+b_1\bar{u} \big)\big(D_2(\frac{k\pi}{L})^2+c_2\bar{v} \big)-b_2c_1\bar{u}\bar{v} }{b_2(\frac{k\pi}{L})^2\phi(\bar{v})\bar{u}\bar{v}}.
\end{equation}
We readily see that $\chi_k>0$ if and only if
\begin{equation}\label{312}
\Big( D_1\big(\frac{k\pi}{L}\big)^2+b_1\bar{u} \Big)\Big(D_2\big(\frac{k\pi}{L}\big)^2+c_2\bar{v} \Big)>b_2c_1\bar{u}\bar{v},k\in \mathbb N^+,
\end{equation}
and we shall assume (\ref{312}) from now on.  Moreover, we have that the null space $\mathcal{N}(D_{(u,v)}\mathcal{F}(\bar{u},\bar{v},\chi_k))$ is one-dimensional and it has a span
\[\mathcal{N}(D_{(u,v)}\mathcal{F}(\bar{u},\bar{v},\chi_k))=\text{span}\left\{ (\bar{u}_k, \bar{v}_k) \right\},\]
where
\begin{equation}\label{313}
(\bar{u}_k,\bar{v}_k)=(Q_k,1)\cos\frac{k\pi x}{L},\text{~with~}Q_k=-\frac{D_2(\frac{k\pi}{L})^2+c_2\bar{v}}{b_2\bar{v}},k\in \mathbb N^+.
\end{equation}

Having the potential bifurcation values $\chi_k$ in (\ref{311}), we now verify that local bifurcation does occur at $(\bar{u},\bar{v},\chi_k)$ for each $k\in \mathbb N^+$, which establishes nonconstant positive solutions to (\ref{31}) in the following theorem.
\begin{theorem}\label{thm31}
Suppose that $\phi\in C^2(\mathbb{R} ,\mathbb{R})$ and $\phi(v)>0$ for all $v>0$.  Assume that the conditions (\ref{12}), (\ref{312}) hold and
\begin{equation}\label{314}
(b_1c_2-b_2c_1)\bar u \bar v\neq k^2j^2D_1D_2\big(\frac{\pi}{L}\big)^4,k\neq j,
\end{equation}
for all positive different integers $k,j\in \mathbb N^+$, where $(\bar u,\bar v)$ is the positive equilibrium of (\ref{35}).  Then for each $k\in \mathbb N^+$, there exists a constant $\delta>0$ and continuous functions
$s\in(-\delta, \delta):\rightarrow (u_k(s,x),v_k(s,x),\chi_k(s)) \in \mathcal{X}\times \mathcal{X} \times \mathbb{R}^+$ such that,
\begin{equation}\label{315}
\chi_k(0)=\chi_k+O(s),~(u_k(s,x),v_k(s,x))=(\bar{u},\bar{v})+s(Q_k,1)\cos \frac{k\pi x}{L} +o(s),
\end{equation}
and $(u_k(s,x),v_k(s,x))-(\bar{u},\bar{v})-s(Q_k,1)\cos \frac{k\pi x}{L} \in \mathcal{Z}$, where
\begin{equation}\label{316}
\mathcal{Z}=\big\{(u,v)\in \mathcal{X} \times \mathcal{X}~ \big \vert \int_0^L u \bar u_k+v\bar v_k dx=0\big\}.
\end{equation}
with $(\bar u_k,\bar v_k)$ and $Q_k$ defined in (\ref{313}); moreover, $(u_k(s,x),v_k(s,x))$ solves system (\ref{35}) and all nontrivial solutions of (\ref{35}) near the bifurcation point ($\bar{u},\bar{v}, \chi_k)$ must stay on the curve $\Gamma_k(s)=(u_k(s), v_k(s), \chi_k(s))$, $s\in(-\delta,\delta)$.
\end{theorem}

\begin{proof}
To apply the local theory of Crandall and Rabinowitz in \cite{CR}, we have checked all but the following transversality condition,
\begin{equation}\label{317}
\frac{d}{d \chi} \left(D_{(u,v)}\mathcal{F}(\bar{u},\bar{v},\chi)\right)(\bar u_k,\bar v_k)\vert_{\chi=\chi_k} \notin \mathcal{R}(D_{(u,v)}\mathcal{F}(\bar{u},\bar{v},\chi_k)),
\end{equation}
where $(\bar u_k,\bar v_k)$ is defined in (\ref{313}) and $\mathcal{R}(\cdot)$ denotes the range of the operator.  If (\ref{317}) fails, there must exist a nontrivial solution $(u,v)$ to the following problem
\begin{equation}\label{318}
\left\{\!\!
\begin{array}{ll}
D_1 u''+\chi_k\phi(\bar{v})\bar{u}v''-b_1\bar{u}u-c_1\bar{u}v
=-\!\big(\frac{k\pi }{L}\big)^2\phi(\bar{v})\bar{u} \cos \frac{k\pi x}{L} ,&x\in(0,L),\\
D_2  v''-b_2\bar{v}u-c_2\bar{v}v=0,&x\in(0,L),\\
u'(x)=v'(x)=0,&x=0,L.
\end{array}
\right.
\end{equation}
Multiplying the both equations in (\ref{318}) by $\cos \frac{k\pi x}{L}$ and then integrating them over $(0,L)$ by parts, we obtain
\begin{equation}\label{319}
\!\!\begin{pmatrix}
\!\!-D_1 \big(\frac{k \pi}{L}\big)^2\!\!\!-b_1\bar{u} & \!\!\!\!-\chi_k \big(\frac{k \pi}{L}\big)^2 \bar{u}\phi(\bar{v})\!\!-c_1\bar{u}   \\
~~\\
-b_2\bar{v} & \!\!-D_2\big(\frac{k \pi}{L}\big)^2\!\!-c_2\bar{v}
\end{pmatrix}\!\!
\!\begin{pmatrix}
\int_0^L u \cos \frac{k \pi x}{L} dx\\
~~\\
\int_0^L v \cos \frac{k \pi x}{L} dx
\end{pmatrix}\!\!=\!\!\begin{pmatrix}
\!-\frac{(k\pi)^2 \bar u \phi(\bar v)}{2L}\\
~~\\
0
\end{pmatrix}\!\!.
\end{equation}

The coefficient matrix of this system is singular because of (\ref{311}), hence we reach a contradiction in (\ref{319}) and this proves (\ref{317}).  Moreover, we must have that $\chi_k \neq \chi_j$ for all integers $k\neq j$, then (\ref{314}) follows from easy calculations and this finishes the proof of Theorem \ref{thm31}.
\end{proof}

\subsection{Stability analysis of the bifurcation branches near ($\bar{u},\bar{v},\chi_k$)}
We now study the stability or instability of the spatially inhomogeneous patterns $(u_k(s,x),\break v_k(s,x))$ established in Theorem \ref{thm31}.  Here the stability or instability is that of $(u_k(s,x),v_k(s,x))$ viewed as equilibrium of system (\ref{31}).  To this end, we first determine the turning direction of each bifurcation branch $\Gamma_k(s)$ around $(\bar u,\bar v,\chi_k)$, $k\in \mathbb N^+$.

The operator $\mathcal{F}$ defined in (\ref{37}) is $C^4$-smooth if $\phi$ is $C^5$-smooth, hence by Theorem 1.18 from \cite{CR}, $(u_k(s,x),v_k(s,x),\chi_k(s))$ in (\ref{315}) is a $C^3$-smooth function of $s$ and we can write the following expansions
\begin{equation}\label{320}
\left\{
\begin{array}{ll}
u_k(s,x)=\bar{u}+sQ_k\cos \frac{k\pi x}{L}+s^2\varphi_1(x)+s^3\varphi_2(x)+o(s^3),\\
v_k(s,x)=\bar{v}+s\cos \frac{k\pi x}{L}+s^2\psi_1(x)+s^3\psi_2(x)+o(s^3),\\
\chi_k(s)=\chi_k+sK_1+s^2K_2+o(s^2),
\end{array}
\right.
\end{equation}
where for $i=1,2$, $(\varphi_i,\psi_i)\in \mathcal{Z}$ defined in (\ref{316}) and $K_i$ is a constant.  $o(s^3)$ in the first two equations of (\ref{320}) are taken with respect to the $\mathcal{X}$-topology.   On the other hand, we have the following fact from Taylor expansions for $\phi\in C^5$
\begin{equation}\label{321}
\phi(v_k(s,x))=\phi(\bar{v})+s\phi'(\bar{v})\cos \frac{k\pi x}{L}+s^2\Big(\phi'(\bar{v}) \psi_1+\frac{\phi''(\bar{v})}{2}\cos^2\frac{k\pi x}{L}\Big)+o(s^3).
\end{equation}
Substituting (\ref{320}) and (\ref{321}) into (\ref{35}) and equating the $s^2$-terms, we collect the following system
\begin{equation}\label{322}
\left\{
\begin{array}{ll}
D_1 \varphi''_1+\chi_k\phi(\bar{v})\bar{u}\psi''_1 =(b_1\varphi_1+c_1\psi_1)\bar{u}+K_1\phi(\bar{v})\bar{u}(\frac{k\pi}{L})^2 \cos \frac{k\pi x}{L}+R_k ,\\
D_2\psi''_1=(b_2\varphi_1+c_2\psi_1)\bar{v}+(b_2Q_k+c_2)\cos^2 \frac{k \pi x}{L},\\
\varphi'_1(x)=\psi'_1(x)=0,x=0,L,
\end{array}
\right.
\end{equation}
where
\[R_k=\chi_k \Big(\frac{k\pi}{L} \Big)^2  \Big(\phi'(\bar{v})\bar{u}+Q_k \phi(\bar{v})\Big)\cos \frac{2k\pi x}{L}+(b_1Q_k+c_1)Q_k \cos^2 \frac{k \pi x}{L}.\]
Multiplying the first equation of (\ref{322}) by $\cos \frac{k\pi x}{L}$ and integrating it over $(0,L)$ by parts, we have
\begin{eqnarray}\label{323}
\frac{k^2\pi^2\phi(\bar{v})\bar{u}K_1}{2L}&=&-\big(\chi_k\phi(\bar{v})\bar{u}+c_1\bar{u} \big) \int_0^L \psi_1 \cos \frac{k\pi x}{L} dx\\
&&-\big(D_1\big(\frac{k\pi}{L}\big)^2+b_1\bar{u}\big) \int_0^L \varphi_1 \cos \frac{k\pi x}{L} dx\nonumber.
\end{eqnarray}
Multiplying the second equation of (\ref{322}) by $\cos \frac{k\pi x}{L}$, we have from the integration by parts that
\begin{equation}\label{324}
b_2 \bar{v}\int_0^L \varphi_1 \cos \frac{k\pi x}{L} dx+\Big(D_2\Big(\frac{k\pi}{L}\Big)^2+c_2\bar{v}\Big) \int_0^L \psi_1 \cos \frac{k\pi x}{L} dx=0.
\end{equation}
On the other hand, since $(\varphi_1,\psi_1)\in \mathcal{Z}$, we have from (\ref{316}) that
\begin{equation}\label{325}
Q_k\int_0^L \varphi_1 \cos \frac{k\pi x}{L} dx+\int_0^L \psi_1 \cos \frac{k\pi x}{L} dx=0,
\end{equation}
where $Q_k=-\frac{D_2 (\frac{k\pi}{L} )^2+c_2\bar{v}}{b_2 \bar{v}}$.  Solving (\ref{324}) and (\ref{325}) leads us to
\[(1+Q_k^2)\int_0^L \varphi_1 \cos \frac{k\pi x}{L} dx=0,\]
which implies that
\[\int_0^L \varphi_1 \cos \frac{k\pi x}{L} dx=\int_0^L \psi_1 \cos \frac{k\pi x}{L} dx=0,~\forall k\in \mathbb N^+,\]
therefore it follows from (\ref{323}) that $K_1=0$.

To determine the sign of $K_2$, we equate the $s^3$-terms in (\ref{35}) and collect
\begin{equation}\label{326}
\left\{
\begin{array}{ll}
D_1 \varphi''_2=K_2\phi(\bar{v})\bar{u}(\frac{k\pi}{L})^2\cos\frac{k\pi x}{L}+(2b_1Q_k\varphi_1+c_1\varphi_1+c_1Q_k\psi_1)\cos \frac{k\pi x}{L}\\
\hspace{1.4cm}+(b_1\varphi_2+c_1\psi_2)\bar{u}-\chi_k A_3,\\
D_2\psi''_2=(b_2\varphi_2+c_2\psi_2)\bar{v}+(b_2\varphi_1+(b_2Q_k+2c_2)\psi_1)\cos\frac{k \pi x}{L},\\
\varphi'_2(x)=\psi'_2(x)=0,x=0,L,
\end{array}
\right.
\end{equation}
where
\begin{eqnarray}
A_3&=&\phi(\bar{v})\bar{u}\psi''_2 -\Big(\phi(\bar{v}) \varphi'_1+\big(\phi(\bar{v})Q_k+2\phi'(\bar{v})\bar{u}\big)\psi'_1\Big)\Big(\frac{k\pi}{L} \Big)\sin \frac{k\pi x}{L}\nonumber\\
&&-\Big(\big(\phi(\bar{v}) \varphi_1+\phi'(\bar{v}) \bar{u}\psi_1\big)\Big(\frac{k\pi}{L}\Big)^2-\big(\phi(\bar{v}) Q_k+\phi'(\bar{v})\bar{u} \big)\psi''_1  \Big)\cos \frac{k\pi x}{L}\nonumber\\
&&+\Big(2\phi'(\bar{v})Q_k+\phi''(\bar{v})\bar{u} \Big)\Big(\frac{k\pi}{L} \Big)^2 \sin^2 \frac{k\pi x}{L} \cos \frac{k\pi x}{L} \nonumber\\
&&-\Big(\phi'(\bar{v})Q_k+\frac{1}{2}\phi''(\bar{v})\bar{u}\Big)\Big(\frac{k\pi}{L} \Big)^2  \cos^3 \frac{k\pi x}{L}\nonumber.
\end{eqnarray}
We have used in (\ref{326}) the Taylor expansion (\ref{321}) with $\phi$ replaced by $\phi'$.

Multiplying the first equation in (\ref{326}) by $\cos \frac{k\pi x}{L}$, we conclude from the integration by parts and straightforward calculations that
\begin{eqnarray}\label{327}
&&-\frac{\phi(\bar{v})\bar{u}(k\pi)^2}{2L}K_2\nonumber \\
=\!\!\!\!\!\!\!&\Big(&\!\!\!\!\!\!\!b_1\bar{u}\!+\!D_1(\frac{k\pi}{L})^2\Big)\!\!\!\int_0^L\!\!\!\! \varphi_2 \cos \!\frac{k\pi x}{L} dx\!+\!\Big(\!c_1\bar{u}\!+\!\chi_k\phi(\bar{v})\bar{u}\big(\frac{k\pi}{L}\big)^2\Big)\!\!\! \int_0^L\!\!\!\! \psi_2 \cos \!\frac{k\pi x}{L} dx\nonumber\\
\!\!\!\!\!&+&\!\!\!\!\!\frac{1}{2}\Big(\!2b_1Q_k\!+\!c_1\!+\!\chi_k\phi(\bar{v})\big(\frac{k\pi}{L}\big)^2\Big)\!\!\! \int_0^L \!\!\!\!\varphi_1 dx \!+\!\frac{1}{2}\Big(\!c_1Q_k\!+\!\chi_k\phi'(\bar v)\bar u\big(\frac{k\pi}{L}\big)^2\Big)\!\!\!\int_0^L\!\!\!\!\psi_1 dx\nonumber\\
&+&\!\!\!\!\!\frac{1}{2}\Big(\!2b_1Q_k\!+\!c_1\!-\!\chi_k\phi(\bar{v})\big(\frac{k\pi}{L}\big)^2 \Big)\!\!\!\int_0^L\!\!\!\! \varphi_1 \cos \!\frac{2k\pi x}{L}dx \\
&+&\!\!\!\!\!\frac{1}{2}\Big(\!c_1Q_k\!+\!\chi_k\big(2\phi(\bar{v})Q_k\!+\!\bar{u}\phi'(\bar{v})\big)\big(\frac{k\pi}{L}\big)^2\Big)\!\!\!\int_0^L\!\!\!\! \psi_1 \cos \!\frac{2k\pi x}{L} dx.\nonumber\\
&+&\!\!\!\!\!\!\Big(\phi'(\bar{v})Q_k\!+\!\frac{1}{2}\phi''(\bar{v})\bar{u}\Big) \frac{k^2\pi^2\chi_k}{8L}\!\nonumber
\end{eqnarray}
On the other hand, we test the second equation of (\ref{326}) by $\cos \frac{k\pi x}{L}$ over $(0,L)$ to obtain
\begin{eqnarray}\label{328}
&&b_2\bar{v}\int_0^L \varphi_2 \cos \frac{k\pi x}{L} dx+\Big(D_2\big(\frac{k\pi}{L}\big)^2+c_2\bar{v}\Big)\int_0^L \psi_2 \cos \frac{k\pi x}{L} dx \nonumber \\
&=&-b_2 \int_0^L \varphi_1\cos^2 \frac{k\pi x}{L} dx-(b_2Q_k+2c_2)\int_0^L \psi_1 \cos^2 \frac{k\pi x}{L} dx,
\end{eqnarray}
Since $(\varphi_2,\psi_2)\in \mathcal{Z}$, we obtain from (\ref{313}), (\ref{316}) and (\ref{328}) that
\begin{eqnarray}\label{329}
 \int_0^L \varphi_2 \cos \frac{k\pi x}{L} dx=&-&\frac{1}{(1+Q_k^2)\bar{v}} \int_0^L \varphi_1 \cos^2 \frac{k\pi x}{L} dx\nonumber\\
&-&\frac{(b_2Q_k+2c_2)}{b_2(1+Q_k^2)\bar{v}} \int_0^L \psi_1 \cos^2 \frac{k\pi x}{L} dx,
\end{eqnarray}
and
\begin{eqnarray}\label{330}
 \int_0^L \psi_2 \cos \frac{k\pi x}{L} dx&=&\frac{Q_k}{(1+Q_k^2)\bar{v}} \int_0^L \varphi_1 \cos^2 \frac{k\pi x}{L} dx\nonumber\\
&&+\frac{(b_2Q_k+2c_2)Q_k}{b_2(1+Q_k^2)\bar{v}} \int_0^L \psi_1 \cos^2 \frac{k\pi x}{L} dx.
\end{eqnarray}
Therefore, thanks to (\ref{328})--(\ref{330}), the $K_2$ equation (\ref{327}) becomes
\begin{eqnarray}\label{331}
&&-\frac{\phi(\bar{v})\bar{u}(k\pi)^2}{2L}K_2\nonumber \\
=\!\!\!\!\!\!&\Big(&\!\!\!\!\!\!\phi'(\bar{v})Q_k\!+\!\frac{1}{2}\phi''(\bar{v})\bar{u}\Big) \frac{k^2\pi^2\chi_k}{8L}+\!\frac{1}{2}\Big(2b_1Q_k\!+\!c_1\!+\!\chi_k\phi(\bar{v})\big(\frac{k\pi}{L}\big)^2\!-\!\!B\Big)\!\!\!\int_0^L\!\!\!\varphi_1dx \nonumber\\
&+&\!\!\!\!\!\frac{1}{2}\Big(c_1Q_k\!+\!\chi_k\bar{u}\phi'(\bar{v})\big(\frac{k\pi}{L}\big)^2\!-\!\frac{B\big(b_2Q_k\!+\!2c_2\big)}{b_2}\Big)\!\!\int_0^L\!\!\!\psi_1 dx \nonumber\\
&+&\!\!\!\!\!\frac{1}{2}\Big(2b_1Q_k\!+\!c_1\!-\!\chi_k\phi(\bar{v})\big(\frac{k\pi}{L}\big)^2\!-\!B\Big)\!\!\int_0^L\!\!\!\varphi_1\cos\frac{2k\pi x}{L} dx \\
&+&\!\!\!\!\!\frac{1}{2}\Big(c_1Q_k\!+\!\chi_k\big(\frac{k\pi}{L}\big)^2\!\!\big(2\phi(\bar{v})Q_k\!+\!\bar{u}\phi'(\bar{v})\big)\!-\!\frac{B\big(b_2Q_k\!\!+\!\!2C_2\big)}{b_2}\Big)\!\!\int_0^L\!\!\!\psi_1\cos\frac{2k\pi x}{L} dx, \nonumber
\end{eqnarray}
where in (\ref{331})
\begin{equation*}
B=\frac{D_1\big(\frac{k\pi}{L}\big)^2+b_1\bar{u}-Q_k\big(c_1\bar{u}+\chi_k\big(\frac{k\pi}{L}\big)^2\bar{u}\phi(\bar{v})\big)}
{\bar{v}\big(1+Q_k^2\big)}.
\end{equation*}
Moreover, in the light of (\ref{311}) and (\ref{314}), we conclude from (\ref{331})
\begin{eqnarray}\label{338}
\frac{k^2\pi^2\phi(\bar{v})\bar{u}}{2L}K_2&=&B_0+B_1\int_0^L\varphi_1 dx+B_2\int_0^L\psi_1 dx\\
&&+B_3\int_0^L\varphi_1\cos\frac{2k\pi x}{L} dx+B_4\int_0^L\psi_1\cos\frac{2k\pi x}{L} dx, \nonumber
\end{eqnarray}
where we have used the following notations in (\ref{338}),
\begin{flalign}
B_0=&-\frac{L}{16}\Big(2\phi'(\bar{v})Q_k+\bar{u}\phi''(\bar{v})\Big)\Big(\frac{Q_k}{\bar{u}\phi(\bar{v})}\big(\frac{k\pi}{L}\big)^2D_1
+\frac{b_1Q_k+c_1}{\phi(\bar{v})}\Big),&\nonumber\\
B_1=&\Big(\frac{Q_k}{2\bar{u}}+\frac{1}{2\bar{v}}\Big)\big(\frac{k\pi}{L}\big)^2D_1-\frac{b_1Q_k}{2}+\frac{b_1\bar{u}}{2\bar{v}},&\nonumber\\
B_2=&\Big(\frac{\phi'(\bar{v})Q_k}{2\phi(\bar{v})}+\frac{b_2Q_k+2c_2}{2b_2\bar{v}}\Big)\big(\frac{k\pi}{L}\big)^2D_1-\frac{c_1Q_k}{2}&\nonumber\\
&+\frac{\bar{u}\phi'(v)\big(b_1Q_k+c_1\big)}{2\phi(\bar{v})}+\frac{b_1\bar{u}\big(b_2Q_k+2c_2\big)}{2b_2\bar{v}},&\nonumber\\
B_3=&\Big(\frac{1}{2\bar{v}}-\frac{Q_k}{2\bar{u}}\Big)\big(\frac{k\pi}{L}\big)^2D_1-\frac{3b_1Q_k+2c_1}{2}+\frac{b_1\bar{u}}{2\bar{v}},&\nonumber\\ \nonumber
\end{flalign}
and
\begin{flalign}
B_4=&\Big(\frac{\bar{u}\phi'(\bar{v})+2\phi(\bar{v})Q_k}{2\bar{u}\phi(\bar{v})}Q_k+\frac{b_2Q_k+2c_2}{2b_2\bar{v}}\Big)
\big(\frac{k\pi}{L}\big)^2D_1+\frac{b_1\bar{u}}{2b_2\bar{v}}\Big(b_2Q_k+2c_2\Big)  &\nonumber\\
&+\frac{\big(\bar{u}\phi'(\bar{v})+2\phi(\bar{v})Q_k\big)\big(b_1Q_k+c_1\big)}{2\phi(\bar{v})}.&\nonumber
\end{flalign}
Hence, we need to evaluate the following integrals to calculate $K_2$ in (\ref{338})
\begin{equation*}
\int_0^L\varphi_1 dx,\int_0^L\psi_1 dx,\int_0^L\varphi_1\cos\frac{2k\pi x}{L} dx, \text{~and~}\int_0^L\psi_1\cos\frac{2k\pi x}{L}dx.
\end{equation*}
To find the first two integrals, we integrate both equations in (\ref{322}) over $(0,L)$, and then because of $K_1=0$, we obtain from straightforward calculations that
\begin{equation}\label{332}
\int_0^L\varphi_1 dx=\frac{c_1\bar u L(b_2Q_k+c_2)-c_2\bar v L(b_1Q_k^2+c_1Q_k)}{2\bar u\bar v(b_1c_2-b_2c_1)}
\end{equation}
and
\begin{equation}\label{333}
\int_0^L\psi_1 dx=\frac{b_2\bar vL(b_1Q_k^2+c_1Q_k)-b_1\bar u L(b_2Q_k+c_2)}{2\bar u\bar v(b_1c_2-b_2c_1)}.
\end{equation}
To find the last two integrals, we multiply both two equations in (\ref{322}) by $\cos\frac{2k\pi x}{L}$ and integrate them over $(0,L)$ by parts.  Then again since $K_1=0$, we have from straightforward calculations that
\begin{eqnarray}\label{334}
\int_0^L\varphi_1\cos\frac{2k\pi x}{L} dx=\frac{|\mathcal{A}_1|}{|\mathcal{A}|},~~ \int_0^L\psi_1\cos\frac{2k\pi x}{L} dx=\frac{|\mathcal{A}_2|}{|\mathcal{A}|}
\end{eqnarray}
where
\begin{eqnarray}\label{335}
|\mathcal{A}|\!\!\!\!\!&=&\!\!\!\!\!\Big(D_1\big(\frac{2k\pi}{L}\big)^2\!+\!b_1\bar{u}\Big)\Big(D_2\big(\frac{2k\pi}{L}\big)^2\!+\!c_2\bar{v}\Big)\!-\!
b_2\bar{u}\bar{v}\Big(\chi_k\phi(\bar{v})\big(\frac{2k\pi}{L}\big)^2\!+\!c_1\Big), \\
|\mathcal{A}_1|\!\!\!\!\!&=&\!\!\!\!\!\frac{\Big(2\chi_k\big(k\pi\big)^2\big(\bar{u}\phi'(\bar{v})+Q_k\phi(\bar{v})\big)+L^2\big(b_1Q_k^2+c_1Q_k\big)\Big)\Big(D_2
\big(\frac{2k\pi}{L}\big)^2+c_2\bar{v}\Big)}{4L}\nonumber\\
&&+\frac{L\Big(b_2Q_k+c_2\Big)\Big(\chi_k\bar{u}\phi(\bar{v})\big(\frac{2k\pi}{L}\big)^2+c_1\bar{u}\Big)}{4},\label{336}\\
|\mathcal{A}_2|\!\!\!\!\!&=&\!\!\!\!\!\frac{b_2\bar{v}\Big(2\chi_k\big(k\pi\big)^2\big(\bar{u}\phi'(\bar{v})+Q_k\phi(\bar{v})\big)+L^2
\big(b_1Q_k^2+c_1\big)\Big)}{4L}\label{337}\\
&&-\frac{L\Big(b_2Q_k+c_2\Big)\Big(D_1\big(\frac{2k\pi}{L}\big)^2+b_1\bar{u}\Big)}{4}\nonumber
\end{eqnarray}

We observe that $K_2$ in (\ref{338}) is extremely complicated and it is very hard to obtain the sign of $K_2$ which determines the stability of bifurcating solution $(u_k(s,x),v_k(s,x),\chi_k)$.  On the other hand, as we shall see in the coming sections, (\ref{35}) admits nontrivial positive solutions $(u,v)$ that have interior transition layers if $D_1$ is sufficiently large and $D_2$ is sufficiently small, therefore, throughout the rest of this section, for the purpose of mathematical modeling of species segregation as well as the simplicity of calculations, we assume that $\min\{D_1,\frac{1}{D_2}\}$ is sufficiently large and now we present the following results on the sign of $K_2$.

\begin{proposition}\label{prop2}Assume that the assumptions in Theorem \ref{thm31} hold.  For each $k\geq 1$, there exists a large $D^*=D^*(a_i,b_i,c_i)>0$ such that for all $D_1,D_2$ with $\min \{D_1,\frac{1}{D_2}\}>D^*$, we have the following results about $K_2$ in (\ref{320}),

(i).  when $\frac{\phi'(\bar v)}{\phi(\bar v)}-\frac{c_2}{b_2\bar{u}}=0$, $K_2>0$ if $\frac{\phi''(\bar{v})}{\phi(\bar{v})}>\frac{2c_2^2}{b^2_2\bar{u}}$ and $K_2<0$ if $\frac{\phi''(\bar{v})}{\phi(\bar{v})}<\frac{2c_2^2}{b^2_2\bar{u}}$;

(ii).  when $\frac{\phi'(\bar v)}{\phi(\bar v)}-\frac{c_2}{b_2\bar{u}}\ne0$, $K_2>0$ if $D_1D_2 (\frac{k\pi}{L} )^4<\frac{(b_1c_2-b_2c_1)\bar{u}\bar{v}}{4}$ and $K_2<0$ if $D_1D_2 (\frac{k\pi}{L})^4>\frac{(b_1c_2-b_2c_1)\bar{u}\bar{v}}{4}$.
\end{proposition}
\begin{proof}
For $D_1\rightarrow+\infty$ and $D_2\rightarrow0^+$, we have $Q_k=-\frac{c_2}{b_2}+O(D_2)$ and the asymptotic expansions
\begin{eqnarray}
B_0\!\!\!\!\!&=&\!\!\!\!\!D_1\Big(\frac{L}{16}\frac{c_2}{b_2\bar{u}\phi(\bar{v})}\big(\phi''(\bar{v})\bar{u}-\frac{2\phi'(\bar{v})c_2}{b_2}\big)
\big(\frac{k\pi}{L}\big)^2+O(1/D_1)\Big),\nonumber\\
B_1\!\!\!\!\!&=&\!\!\!\!\!D_1\Big(\big(\frac{1}{2\bar{v}}-\frac{c_2}{2b_2\bar{u}}\big)\big(\frac{k\pi}{L}\big)^2+O(1/D_1)\Big),\nonumber\\
B_2\!\!\!\!\!&=&\!\!\!\!\!D_1\Big(\big(\frac{c_2}{2b_2\bar{v}}-\frac{c_2\phi'(\bar{v})}{2b_2\phi(\bar{v})}\big)\big(\frac{k\pi}{L}\big)^2
+O(1/D_1)\Big),\nonumber\\ B_3\!\!\!\!\!&=&\!\!\!\!\!D_1\Big(\big(\frac{1}{2\bar{v}}+\frac{c_2}{2b_2\bar{u}}\big)\big(\frac{k\pi}{L}\big)^2+O(1/D_1)\Big),\nonumber\\ \nonumber
\end{eqnarray}
and
\begin{eqnarray}
B_4\!\!\!\!\!&=&\!\!\!\!\!D_1\Big(\big(\frac{c^2_2}{b_2^2\bar{u}}+\frac{c_2}{2b_2\bar{v}}-\frac{c_2\phi'(\bar{v})}{2b_2\phi(\bar{v})}
\big)\big(\frac{k\pi}{L}\big)^2+O(1/D_1)\Big).\nonumber
\end{eqnarray}
Thanks to these expansions, we see that the $K_2$ equation (\ref{338}) implies
\begin{eqnarray}\label{339}
&&\frac{\bar{u}\phi(\bar{v})L}{2D_1} K_2 \nonumber\\
=\!\!\!\!\!&& \frac{L}{16}\frac{c_2}{b_2\bar{u}\phi(\bar{v})}\Big(\phi''(\bar{v})\bar{u}-\frac{2\phi'(\bar{v})c_2}{b_2}\Big)
+\Big(\frac{1}{2\bar{v}}-\frac{c_2}{2b_2\bar{u}}\Big)\int_0^L\!\!\!\varphi_1 dx\nonumber\\
&&-\Big(\frac{c_2}{2b_2\bar{v}}-\frac{c_2\phi'(\bar{v})}{2b_2\phi(\bar{v})}\Big)\frac{b_2}{c_2}
\int_0^L\!\!\!\varphi_1 dx+\Big(\frac{1}{2\bar v}+\frac{c_2}{2b_2\bar{u}}\Big)\int_0^L\!\!\!\varphi_1\cos\frac{2k\pi x}{L} dx\nonumber\\
&&-\Big(\frac{c^2_2}{b_2^2\bar{u}}+\frac{c_2}{2b_2\bar{v}}-\frac{c_2\phi'(\bar{v})}{2b_2\phi(\bar{v})}\Big)
\frac{b_2}{c_2}\int_0^L\!\!\!\varphi_1\cos\frac{2k\pi x}{L} dx+O(1/D_1)\!\!\int_0^L\!\!\!\varphi_1\cos\frac{2k\pi}{L} dx\nonumber\\
&&+O(1/D_1)+O(D_2)\nonumber\\
&=&\!\!\!\!\!\frac{L}{16}\frac{c_2}{b_2\bar{u}\phi(\bar v)}\Big(\phi''(\bar{v})\bar{u}-\frac{2\phi'(\bar{v})c_2}{b_2}\Big)+O(1/D_1)\int_0^L\varphi_1\cos\frac{2k\pi x}{L} dx\nonumber \\
&&+\frac{1}{2}\Big(\frac{\phi'(\bar v)}{\phi(\bar v)}-\frac{c_2}{b_2\bar{u}}\Big) \Big(\int_0^L\varphi_1 dx+\int_0^L\varphi_1\cos\frac{2k\pi x}{L} dx\Big)\\
&&+O(1/D_1)+O(D_2) \nonumber
\end{eqnarray}
where we have used in (\ref{339}) the facts
\[\int_0^L\psi_1 dx=\big(-\frac{c_2}{b_2}+O(D_2)\big)\int_0^L\varphi_1 dx,\]
from (\ref{333}) and
\begin{equation*}
\int_0^L\psi_1\cos\frac{2k\pi x}{L} dx=\Big(-\frac{b_2}{c_2}+O(D_2)\Big)\int_0^L\varphi_1\cos\frac{2k\pi x}{L} dx.
\end{equation*}
from (\ref{334}), (\ref{335}) and (\ref{337}).  On the other hand, we have that (\ref{332})--(\ref{334}) become
\begin{equation*}
\int_0^L\varphi_1 dx=-\frac{c^2_2L}{2b^2_2\bar{u}}+O(D_2),
\end{equation*}
\begin{equation*}
\int_0^L\varphi_1\cos\frac{2k\pi x}{L}dx=\frac{D_1 \Big(\frac{c_2^2 \bar v L}{b_2}\big(\frac{c_2}{b_2\bar{u}}-\frac{\phi'(\bar{v})}{\phi(\bar{v})}\big)\big(\frac{k\pi}{L}\big)^2+O(1/D_1)\Big)}{24D_1D_2
\big(\frac{k\pi}{L}\big)^4-6\big(b_1c_2-b_2c_1\big)\bar{u}\bar{v}},
\end{equation*}
We now divide our discussions into the following two cases.  If $\frac{\phi'(\bar v)}{\phi(\bar v)}-\frac{c_2}{b_2\bar{u}}=0$, we see in (\ref{339}) that
\[\frac{\bar{u}\phi(\bar{v})L}{2D_1}K_2=\frac{L}{16}\frac{c_2}{b_2}\Big(\frac{\phi''(\bar v)}{\phi(\bar{v})}-\frac{2c^2_2}{b^2_2\bar{u}^2}\Big)+O(1/D_1)+O(D_2),\]
therefore, \emph{(i)} follows immediately.  If $\frac{\phi'(\bar v)}{2\phi(\bar v)}-\frac{c_2}{2b_2\bar{u}}\neq0$, we have that
\begin{equation}\label{340}
\frac{\bar{u}\phi(\bar v)L}{2D_1}K_2 =
-\frac{D_1\Big(c^2_2\bar{v}L\big(\frac{k\pi}{L}\big)^2\big(\frac{c_2}{b_2\bar{u}}-\frac{\phi'(\bar{v})}{\phi(\bar{v})}\big)^2
+O(1/D_1)\Big)}{12b_2\Big(4D_1D_2\big(\frac{k\pi}{L}\big)^4-\big(b_1c_2-b_2c_1\big)\bar{u}\bar{v}\Big)}.
\end{equation}
We conclude that \emph{(ii)} is an immediate consequence of (\ref{340}).
\end{proof}
\begin{remark}\label{remark2}
According to Theorem \ref{thm31}, we have that $4D_1D_2\big(\frac{k\pi}{L}\big)^4-\big(b_1c_2-b_2c_1\big)\bar{u} \bar{v}$ is always nonzero since $\chi_{k}\neq \chi_{2k}$.  We also point out that for the strong competition case $\frac{b_1}{b_2}<\frac{a_1}{a_2}<\frac{c_1}{c_2}$ in Proposition \ref{prop2}, $K_2$ is always negative if $\min\{D_1,\frac{1}{D_2}\}$ is large, independent on the size of $D_1D_2$.
\end{remark}

We continue with the stability analysis of the bifurcating solutions $(u_k(s,x),v_k(s,\break x),\chi_k(s))$ with $s\in (-\delta,\delta)$.  This branch of solutions will be asymptotically stable if the real part of any eigenvalue $\lambda$ of the following problem is negative:
\begin{equation}\label{341}
D_{(u,v)}\mathcal{F}(u_k(s,x),v_k(s,x),\chi_k(s))(u,v)=\lambda(u,v),~(u,v)\in \mathcal{X} \times \mathcal{X}.
\end{equation}
Taking $s=0$, we see that $\lambda=0$ is a simple eigenvalue of $D_{(u,v)}\mathcal{F}(\bar{u},\bar{v},\chi_k)$ with eigenspace $\mathcal{N}\big(D_{(u,v)}\mathcal{F}(\bar{u},\bar{v},\chi_k)\big)=\{(Q_k,1)\cos \frac{k\pi x}{L}\}$; moreover, following the same analysis that leads to (\ref{317}), one can also prove that $(Q_k,1)\cos \frac{k\pi x}{L} \not\in\mathcal{R}\big(D_{(u,v)}\mathcal{F}(\bar{u},\break\bar{v},\chi_k)\big)$.  We now present the following stability results.

\begin{theorem}\label{thm32}
Suppose that $\chi_{k_0}=\min_{k\in\mathbb N^+}\chi_k$ of (\ref{311}), then the following statements hold:

(i). $(u_k(s,x),v_k(s,x))$, $s\in(-\delta,\delta)$, is unstable for all positive integers $k\neq k_0$;

(ii). Suppose that $\min \{D_1,\frac{1}{D_2}\}$ is large, then \emph{(iia)}: when $\frac{\phi'(\bar v)}{\phi(\bar v)}-\frac{c_2}{b_2\bar{u}}=0$, $(u_{k_0}(s,x),v_{k_0}(s,x))$ is stable if $\frac{\phi''(\bar{v})}{\phi(\bar{v})}>\frac{2c_2^2}{b^2_2\bar{u}}$ and $(u_{k_0}(s,x),v_{k_0}(s,x))$ is unstable if $\frac{\phi''(\bar{v})}{\phi(\bar{v})}<\frac{2c_2^2}{b^2_2\bar{u}}$; \emph{(iib)}: when $\frac{\phi'(\bar v)}{\phi(\bar v)}-\frac{c_2}{b_2\bar{u}}\ne0$, $(u_{k_0}(s,x),v_{k_0}(s,x))$ is stable if $D_1D_2 (\frac{{k_0}\pi}{L} )^4\break<\frac{(b_1c_2-b_2c_1)\bar{u}\bar{v}}{4}$ and $(u_{k_0}(s,x),v_{k_0}(s,x))$ is unstable if $D_1D_2 (\frac{{k_0}\pi}{L})^4>\frac{(b_1c_2-b_2c_1)\bar{u}\bar{v}}{4}$.
\end{theorem}

\begin{proof}
To prove \emph{(i)}, we first study the limit of (\ref{341}) for all $k\in \mathbb N^+$ as $s\rightarrow 0$, i.e, the following eigenvalue problem
\begin{equation}\label{341a}
\left\{
\begin{array}{ll}
D_1 u''+\bar \chi_k \phi(\bar{v})\bar{u}v''-b_1\bar{u}u-c_1\bar{u}v=\lambda u,& x\in(0,L),  \\
D_2 v''-b_2\bar{v}u-c_2\bar{v}v=\lambda v,& x\in(0,L),  \\
u'(x)=v'(x)=0,&x=0,L.
\end{array}
\right.
\end{equation}
Multiplying (\ref{341a}) by $\cos\frac{k\pi x}{L}$ and then iterating it over $(0,L)$ by parts, we have that
\[\begin{pmatrix}\!\!
-D_1 \!\big(\frac{k \pi}{L}\big)^2\!\!-\!b_1\bar{u}-\lambda & -\chi_k \big(\frac{k \pi}{L}\big)^2 \!\phi(\bar{u},\bar{v})\!-\!c_1\bar{u}   \\
-b_2\bar{v} & -D_2\big(\frac{k \pi}{L}\big)^2\!\!-\!c_2\bar{v}-\lambda
\!\!\end{pmatrix}\!
\!\!\begin{pmatrix}\!
\int_0^L \!\!{u}\cos\frac{k\pi x}{L}dx \\
\int_0^L\!\! {v}\cos\frac{k\pi x}{L}dx
\end{pmatrix}\!\!=\!\!\begin{pmatrix}
0\\
0
\end{pmatrix}.
\]
$\lambda$ is an eigenvalue of (\ref{341}) if and only if $p(\lambda)=\lambda^2+\text{Tr}\lambda+\text{Det}=0$, where
\[\text{Tr}=\big(D_1+D_2\big)\Big(\frac{k\pi}{L}\Big)^2+b_1\bar u+c_2\bar v>0\]
and
\[\text{Det}=\big(D_1\big(\frac{k\pi}{L}\big)^2+b_2\bar{u}\big)\big(D_2\big(\frac{k\pi}{L}\big)^2 +c_2\bar v\big)-\big(\chi\bar{u}\phi(\bar{v})\big(\frac{k\pi}{L}\big)^2+c_1\bar u\big)b_2 \bar{v}.\]
It is easy to see that $p(\lambda)=0$ has a positive root $\lambda>0$ if $\text{Det}>0$, or $\chi=\chi_{k_0}=\min_{k\in \mathbb N^+} \chi_k$.  Therefore, from the standard eigenvalue perturbation theory in \cite{Ka}, (\ref{341}) always has a positive root for small $s$ if $k\neq k_0$.  This finishes the instability result \emph{(i)}.

We now proceed to study the stability of $(u_{k_0}(s,x),v_{k_0}(s,x))$.  According to the analysis above, we see that (\ref{341a}) with $k=k_0$ has one negative eigenvalue and a zero eigenvalue.  To show that $(u_{k_0}(s,x),v_{k_0}(s,x))$ is stable, we need to show that both eigenvalues of (\ref{341}) are negative (have a negative real part) for $s\neq 0$.  According to Corollary 1.13 in \cite{CR2}, there exist an interval $I$ with $\chi_{k_0}\in I$ and $C^1$-smooth functions $(\chi,s):I\times (-\delta,\delta) \rightarrow (\mu(\chi),\lambda(s))$ with $\lambda(0)=0$ and $\mu(\chi_{k_0})=0$ such that, $\lambda(s)$ is an eigenvalue of (\ref{341}) and $\mu(\chi)$ is an eigenvalue of the following eigenvalue problem
\begin{equation}\label{342}
D_{(u,v)}\mathcal{F}(\bar{u},\bar{v},\chi)(u,v)=\mu(u,v),~(u,v)\in \mathcal{X} \times \mathcal{X}.
\end{equation}
Moreover, $\lambda(s)$ is the only eigenvalue of (\ref{341}) in any fixed neighbourhood of the origin of the complex plane and the same assertion can be made about $\mu(\chi)$.  We also know from \cite{CR2} that the eigenfunction of (\ref{342}) can be represented by $\big(u(\chi,x),v(\chi,x)\big)$, which depends on $\chi$ smoothly and is uniquely determined by $\big(u(\chi_{k_0},x),v(\chi_{k_0},x)\big)=\big( Q_{k_0} \cos \frac{k_0\pi x}{L},\cos \frac{k_0\pi x}{L} \big)$ and $\big(u(\chi,x),v(\chi,x)\big)$ $-\break\big( Q_{k_0}\cos \frac{k_0\pi x}{L},\cos \frac{k_0\pi x}{L} \big) \in \mathcal{Z}$, where $Q_{k_0}$ and $\mathcal{Z}$ are defined (\ref{313}) and (\ref{316}) respectively.
We have from (\ref{39}) that (\ref{342}) is equivalent to
\begin{equation}\label{343}
\left\{
\begin{array}{ll}
D_1 u''+\chi \phi(\bar{v})\bar{u}v''-b_1\bar{u}u-c_1\bar{u}v=\mu u,& x\in(0,L),  \\
D_2 v''-b_2\bar{v}u-c_2\bar{v}v=\mu v,& x\in(0,L),  \\
u'(x)=v'(x)=0,&x=0,L.
\end{array}
\right.
\end{equation}
Differentiating (\ref{343}) with respect to $\chi$ and then taking $\chi=\chi_{k_0}$, we have
\begin{equation}\label{344}
\left\{\!\!\!
\begin{array}{ll}
D_1\dot{u}''\!\!-\!\phi(\bar{v})\bar{u}\Big(\cos\frac{k_0\pi x}{L}\Big)''\!\!+\chi_{k_0}\phi(\bar{v})\bar{u}\dot{v}''\!\!-\!b_1\bar{u}\dot{u}\!-\!c_1\bar{u}\dot{v}=\dot{\mu}(\chi_{k_0}) Q_{k_0} \cos\frac{k_0\pi x}{L},  \\
D_2 \dot{v}''\!\!-b_2\bar{v}\dot{u}-c_2\bar{v}\dot{v}=\dot{\mu}(\chi_{k_0}) \cos\frac{k_0\pi x}{L},\\
\dot{u}'(x)=\dot{v}'(x)=0,~x=0,L,
\end{array}
\right.
\end{equation}
where the dot-notation $\dot{}$ in (\ref{344}) denotes the differentiation with respect to $\chi$ evaluated at $\chi=\chi_{k_0}$ and in particular $\dot{u}=\frac{\partial u(\chi,x)}{\partial \chi}\big\vert_{\chi=\chi_{k_0}}$, $\dot{v}=\frac{\partial v(\chi,x)}{\partial \chi}\big\vert_{\chi=\chi_{k_0}}$.

Multiplying both equations in (\ref{344}) by $\cos\frac{k \pi x}{L}$ and integrating them over $(0,L)$ by parts, we arrive at the following system
\[\begin{pmatrix}\!\!
-D_1 \!\big(\frac{k \pi}{L}\big)^2\!\!-\!b_1\bar{u} & -\chi_{k_0} \big(\frac{k \pi}{L}\big)^2 \!\phi(\bar{u},\bar{v})\!-\!c_1\bar{u}   \\
-b_2\bar{v} & -D_2\big(\frac{k \pi}{L}\big)^2\!\!-\!c_2\bar{v}
\!\!\end{pmatrix}\!
\!\!\begin{pmatrix}\!
\int_0^L \!\!\dot{u}\cos\frac{k_0\pi x}{L}dx \\
~~\\
\int_0^L\!\! \dot{v}\cos\frac{k_0\pi x}{L}dx
\end{pmatrix}\]\[ \!\!=\!\!\begin{pmatrix}
\!\Big(\dot{\mu}(\chi_{k_0})Q_{k_0}\!-\!\phi(\bar{v})\bar{u}\big(\frac{k_0\pi }{L} \big)^2\Big)\! \frac{L}{2}\\
~~\\
\dot{\mu}(\chi_{k_0})\frac{L}{2}
\!\!\!\!\end{pmatrix}
\]
We see from (\ref{311}) that the coefficient matrix is singular, therefore if the algebraic system is solvable, we must have that $\frac{D_1 (\frac{k_0 \pi}{L})^2+b_1\bar{u}}{b_2\bar{v}}=\frac{\dot{\mu}(\chi_{k_0})Q_{k_0}-\phi(\bar{v})\bar{u} (\frac{k_0\pi }{L})^2}{\dot{\mu}(\chi_{k_0})}$, which together with (\ref{313}), implies that
\[\dot{\mu}(\chi_{k_0})=-\frac{b_2\phi(\bar{v})\bar{u}\bar{v}\big(\frac{k_0\pi }{L} \big)^2}{(D_1+D_2)\big(\frac{k_0\pi }{L} \big)^2+b_1\bar{u}+c_2\bar{v}}<0.\]
By Theorem 1.16 in \cite{CR2}, the functions $\lambda(s)$ and $-s\chi'_{k_0}(s)\dot{\mu}(\chi_{k_0})$ have the same zeros and the same sign near $s=0$, and for $\lambda(s) \neq0$,
\[\lim_{s\rightarrow 0}\frac{-s\chi'_{k_0}(s)\dot{\mu}(\chi_{k_0})}{\lambda(s)}=1,\]
therefore, this shows that $\text{sgn} (\lambda(s))=\text{sgn}(K_2)$ in light of $K_1=0$.

Thanks to Proposition \ref{prop2} and the analysis that leads to part \emph{(i)}, the instability statements follow right away from the positive sign of $\lambda(s)$.  To show the stability part, we first observe that, following the same calculations that lead to (\ref{316}), as $s\rightarrow 0$, (\ref{342}) has no nonzero eigenvalues with nonpositive real parts if $K_2>0$.  Then it follows from the standard perturbation theory that all eigenvalues of (\ref{341}) have no positive real part in a small neighborhood of the origin of the complex plane.  This completes the proof of Theorem \ref{thm32}.
\end{proof}

\begin{figure}[!htb]
\minipage{.4\textwidth}
  \includegraphics[width=1.2in]{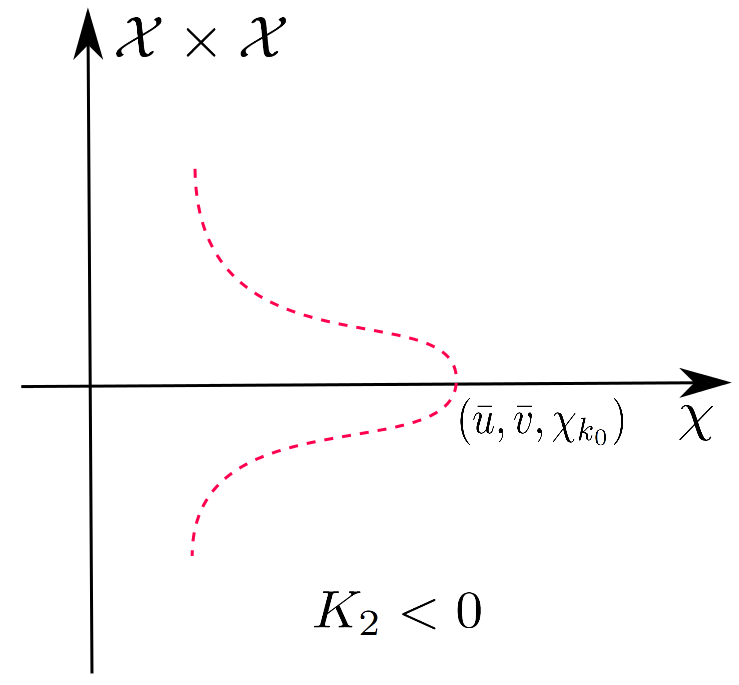}
%  \caption*{An illustration of the bifurcation branch for stable bifurcation branch}\label{fig:awesome_image1}
\endminipage
\minipage{.4\textwidth}
  \includegraphics[width=1.2in]{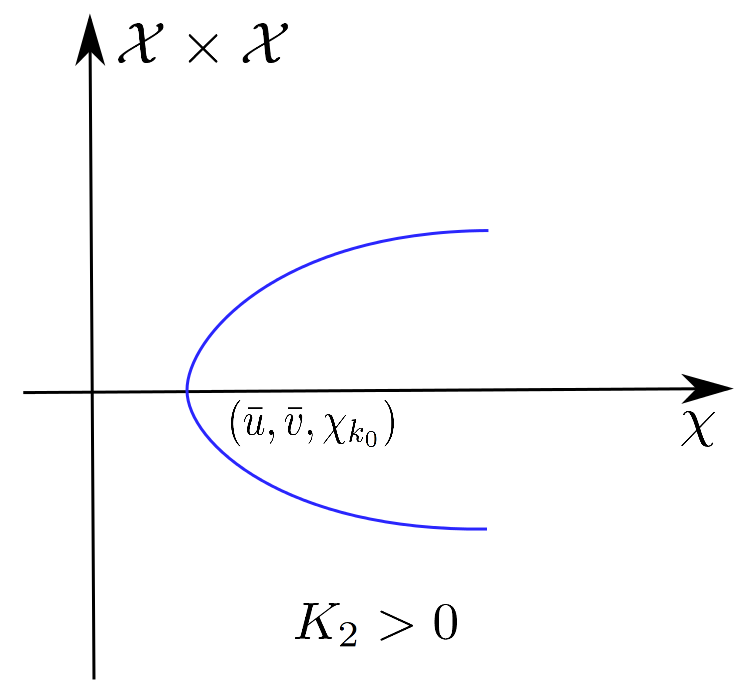}
%  \caption*{An illustration of the bifurcation branch for unstable bifurcation branch}\label{fig:awesome_image2}
\endminipage
\caption{Pitchfork bifurcations are illustrated.  The solid line represents stable bifurcating solutions $(u_{k_0}(s,x),v_{k_0}(s),\chi_{k_0}(s))$ and the dashed line represents unstable solutions $(u_{k_0}(s,x),v_{k_0}(s,x),\chi_{k_0}(s))$.}\label{fig1}
\end{figure}

Theorem \ref{thm32} gives a wave mode selection mechanism of system (\ref{11}).  If the spatial pattern $(u_k(s,x),v_k(s,x))$ is stable, then $k$ must be the integer that minimizes the bifurcation value $\chi_k$, that is, if a bifurcation branch is stable, then it must be the first branch counting from the left to the right.  Given initial data being small perturbations from $(\bar u,\bar v)$, spatially inhomogeneous patterns can emerge through this mode and finally develop into stable patterns with interesting structures, such as interior spikes, transition layers, etc.  However, rigourous analysis to this end is out of the scope of our paper.

If the interval length $L$ is small, we see that
\[\chi_k=\frac{\big( D_1(\frac{k\pi}{L})^2+b_1\bar{u} \big)\big(D_2(\frac{k\pi}{L})^2+c_2\bar{v} \big)-b_2c_1\bar{u}\bar{v} }{b_2(\frac{k\pi}{L})^2\phi(\bar{v})\bar{u}\bar{v}}\approx \frac{D_1D_2(\frac{k\pi}{L})^2}{b_2\phi(\bar v)\bar u\bar v}\]
and $\chi_1=\min_{k\in \mathbb N^+} \chi_k$, then Theorem \ref{thm32} indicates that the monotone solution $(u_1(s,x),v_1(s,x))$ is the only stable pattern.  Actually, we see that, if $L$ increases, then the minimizer of $\chi_k$, denoted by $k_0$, also increases.  This indicates that small domain only supports monotone stable solutions, while large domain supports non-monotone stable solutions.  Actually, given a monotone solution of (\ref{35}), one can construct its non-monotone steady states by reflection and periodic extension of the monotone solution at the boundary points $-2L,-L,0,L,2L,...$

From the view point of biology, the perceived intensity of a stimulus should have a saturation effect on the strength of stimulus.  This can be modeled mathematically by choosing $\phi''<0$ in (\ref{35}).  According to Proposition \ref{prop3} and Theorem \ref{thm32}, if $\min \{D_1,\frac{1}{D_2}\}$ is large, in the strong competition case $\frac{b_1}{b_2}<\frac{a_1}{a_2}<\frac{c_1}{c_2}$, the small-amplitude solutions $(u_k(s,x),v_k(s,x))$ are unstable independent of the size of $D_1D_2$.  Moreover, the advection rate $\chi$ tends to destabilize the constant solution $(\bar u,\bar v)$ as we have seen in Proposition \ref{prop1}, therefore, we are motivated to study the solutions to (\ref{35}) that have large amplitudes as $\chi$ increases, with $\min \{D_1,\frac{1}{D_2}\}$ being large.  In the weak competition case, we see that the bifurcating solutions $(u_k(s,x),v_k(s,x))$ are unstable for all $k$ large.

\section{Effect of large advection rate $\chi$}\label{section4}
This section is devoted to study the asymptotic behaviors of positive solutions $(u,v)$ to (\ref{35}) if the advection rate $\chi$ is large.  The first step of our analysis is to present the following a prior estimates.
\begin{lemma}\label{lem41}
Assume that $\phi(v)\in C^2(\mathbb{R},\mathbb{R})$.
Let $(u,v)$ be any positive solution of (\ref{35}).  Then we have that
\begin{equation}\label{41}
\max_{x\in[0,L]} v(x)\leq\frac{a_2}{c_2},
\end{equation}
and
\begin{equation}\label{42}
\int_0^Lu^2dx \leq \frac{a_1^2L}{b_1^2}.
\end{equation}
%\begin{equation}
%u(x)\leq u(0)e^{-\frac{\chi \Vert\Phi(v) \Vert_\infty}{D_1}x}+\frac{C(L)}{\chi \Vert\Phi(v) \Vert_\infty}\Big(e^{-\frac{\chi \Vert\Phi(v) \Vert_\infty}{D_1}L}+L \Big);
%\end{equation}
%moreover, there exists $C$ independently of $\chi$ and $D_1$ such that
%\begin{equation}
%\Vert u(x) \Vert _{H^2(0,L)}, \Vert v(x) \Vert _{H^2(0,L)} \leq C.
%\end{equation}
\end{lemma}
\begin{proof}
First of all, we see that (\ref{41}) follows from the Maximum Principles.  We integrate the $u$-equation in (\ref{35}) over $(0,L)$ and obtain
\[\int_0^L(a_1-b_1u-c_1v)udx=0.\]
Then we have from the Young's inequality that
\begin{equation*}\label{43}
b_1\int_0^L u^2 dx+c_1 \int_0^Luvdx=a_1\int_0^Ludx \leq \frac{b_1}{2}\int_0^L u ^2dx+\frac{a_1^2L}{2b_1}.
\end{equation*}
This implies (\ref{42}) and it completes the proof of Lemma \ref{lem41}.
\end{proof}
We now study the asymptotic behaviors of $(u,v)$ by passing the advection rate $\chi$ to infinity in (\ref{35}).  It seems that one needs the largeness of $D_1$ in order to establish nontrivial patterns.  On the other hand, we have assumed that the diffusion rate $D_2$ is sufficiently small for our stability analysis in Section 3, and here we treat $D_1$ and $D_2$ separately and then as $\chi\rightarrow \infty$, we have the following results.

\begin{theorem}\label{thm42}
Suppose that $\phi\in C^2(\mathbb{R},\mathbb{R})$.  Let $(u_i,v_i)$ be positive solutions of (\ref{35}) with $(D_{1,i},D_{2,i},\chi_i)=(D_1,D_2,\chi)$ and $\frac{\chi_i}{D_{1,i}}=r_i$.  Assume that $\chi_{i} \rightarrow \infty$, $D_{2,i}\rightarrow D_2 \in(0,\infty)$ and $r_i \rightarrow r\in(0,\infty)$ as $i\rightarrow \infty$, then there exists a nonnegative constant $\lambda_\infty$ such that
\[u_ie^{r_i\Phi(v_i)} \rightarrow \lambda_\infty \text{~uniformly on~}[0,L];\] moreover, after passing to a subsequence if necessary as $i \rightarrow \infty$
\[(u_i,v_i)\rightarrow (\lambda_\infty  e^{-r\Phi(v_\infty)},v_\infty) \text{ in }C^1([0,L])\times C^1([0,L]),\]
where $v_\infty=v_\infty(x)$ satisfies the following shadow system
\begin{equation}\label{44}
\left\{
\begin{array}{ll}
D_2 v''_\infty+(a_2-b_2 \lambda_\infty e^{-r\Phi(v_\infty)}-c_2 v_\infty)v_\infty=0, & x \in(0,L),\\
\int_0^L (a_1-b_1 \lambda_\infty e^{-r\Phi(v_\infty)}-c_1 v_\infty) e^{-r\Phi(v_\infty)} dx=0,\\
v_\infty'(0)=v_\infty'(L)=0.
\end{array}
\right.
\end{equation}
\end{theorem}

\begin{proof} We see from (\ref{42}) and the $v$-equation in (\ref{35}) that, $\Vert v_i \Vert_{H^2(0,L)}$ is uniformly bounded for all $\chi_i$ and $D_{1,i}$.  By Sobolev embedding theorem, we can show that $v_i \rightarrow v_\infty$ in $C^1([0,L])$ as $i\rightarrow \infty$, after passing to a subsequence if necessary.
Dividing the $u$-equation in (\ref{35}) by $D_{1,i}$ and then integrating it over $(0,x)$, we obtain that
\[u_i'+r_i u_i\Phi'(v_i)=-\frac{1}{D_{1,i}} \int_0^x (a_1-b_1u_i-c_1v_i)u_idx.\]
Denoting $w_i=u_i e^{r_i\Phi(v_i)}$, we can easily show from Young's inequality that
\begin{equation}\label{45}
\Big \vert e^{-r_i\Phi(v_i)}w_i' \Big \vert\leq \frac{1}{D_{1,i}} \int_0^x \Big \vert(a_1-b_1u_i-c_1v_i)u_i\Big \vert dx \leq \frac{1}{D_{1,i}} \int_0^L 2b_1u_i^2+M dx,
\end{equation}
where $M$ is a positive constant independent of $D_{1,i}$.  Sending $i$ to $\infty$ in (\ref{45}), we conclude from (\ref{41}), (\ref{42}) and the Neumann boundary conditions that $w'_{i} \rightarrow 0$ uniformly.  Therefore, we must have that $w_{i}=u_ie^{r \Phi(v_i)}$ converges to some nonnegative constant $\lambda_\infty$.  Moreover we can show from elliptic regularity theories that $v_\infty$ is smooth and it satisfies the shadow system (\ref{44}).
\end{proof}
\begin{remark}
If $r_i\rightarrow r=0$, we see that the convergence in Theorem \ref{thm42} still holds and $u_i$ converges to the constant $\lambda_\infty$.  Moreover, (\ref{44}) with $r=0$ has only trivial solution $v_\infty$, therefore $(u_i,v_i)$ can only be trivial if $\chi_i \rightarrow \infty$ with $D_{1,i}$ being relatively larger.  Same conclusions can be made if $r_i\rightarrow r=\infty$.  Hence we only consider the limiting process with the diffusion rate $D_1$ and the advection rate $\chi$ being comparably large.
\end{remark}
We want to point out that, if $b_1\neq0$, the boundedness of $\Vert u \Vert_{L^2}$ in (\ref{42}) is required to show the convergence in Theorem \ref{thm42}.  However, if $b_1=0$, the statements in Theorem \ref{thm42} still hold though $\Vert u \Vert_{L^2}$ may be unbounded.  Actually, if $b_1=0$, (\ref{35}) implies
\[\int_0^L udx =\frac{c_1}{a_1} \int_0^L uvdx.\]
Integrating the $v$-equation over $(0,L)$, we collect that
\begin{equation*}\label{46}
\int_0^L uv dx=\frac{1}{b_2}\int_0^L(a_2-c_2v)vdx,
\end{equation*}
which is uniformly bounded in $D_1$, therefore $\Vert u \Vert_{L^1}$ is uniformly bounded.  Taking $b_1=0$ in (\ref{35}), we can also show that $w_i$ converges to a constant as $D_{1,i}$ converges to infinity.

\subsection{Existence of nonconstant positive solutions to the shadow system}

Now we proceed to study the existence of nonconstant solutions of (\ref{44}).  One can easily show that (\ref{44}) has only trivial solution if $D_2$ is large.  On the other hand, since small diffusion rate $D_2$ tends to support $v_\infty$ that has an interior transition layer, we put $D_2=\epsilon$ and denote $(v_\infty,\lambda_\infty)=(v_\epsilon,\lambda_\epsilon)$, then (\ref{44}) becomes
\begin{equation}\label{47}
\left\{
\begin{array}{ll}
\epsilon v_\epsilon''+\Big(a_2-b_2 \lambda_\epsilon e^{-r\Phi(v_\epsilon )}-c_2 v_\epsilon\Big)v_\epsilon =0, ~ x \in(0,L),\\
\int_0^L (a_1-b_1 \lambda_\epsilon e^{-r\Phi(v_\epsilon)}-c_1 v_\epsilon) e^{-r\Phi(v_\epsilon)} dx=0,\\
v_\epsilon(x)>0,x\in(0,L);~v_\epsilon'(0)=v_\epsilon'(L)=0.
\end{array}
\right.
\end{equation}
Obviously (\ref{47}) has a unique positive trivial solution
\[(\bar v,\bar \lambda)=\Big(\frac{a_2b_1-a_1b_2}{b_1c_2-b_2c_1},\frac{a_2b_1-a_1b_2}{b_1c_2-b_2c_1} e^{-r\Phi(\bar v)} \Big),\]
provided with (\ref{12}).  First of all, we establish the existence of nonconstant positive solutions to (\ref{47}) through the bifurcation analysis by taking $\epsilon$ as the bifurcation parameter.  Similarly as the bifurcation analysis in Section 3, we rewrite (\ref{47}) into the following abstract form
\[\mathcal{T}(v,\lambda,\epsilon)=0,~(v,\lambda,\epsilon)\in \mathcal{X}\times \mathbb{R}^+ \times\mathbb{R}^+,\]
where $\mathcal{X}$ is defined in (\ref{36}).  Then $\mathcal{T}(\bar{v},\bar{\lambda},\epsilon)=0$ for any $\epsilon \in \mathbb{R}$ and $\mathcal{T}$ is analytic from $\mathcal{X} \times \mathbb{R}^+ \times \mathbb{R}^+$ to $\mathcal{Y} \times \mathbb{R}$, where $\mathcal{Y}=L^2(0,L)$;  moreover, it follows through straightforward calculations that, for any fixed $(v_0,\lambda_0) \in \mathcal{X} \times \mathbb{R}$, the Fr\'echet derivative of $\mathcal{T}$ is given by
\begin{equation}\label{48}
\begin{split}
&D_{(v,\lambda)}\mathcal{T}(v_0,\lambda_0,\epsilon)(v,\lambda)\\
 =&\left(
 \begin{array}{c}
\epsilon  v''\!+\!\big(a_2\!-\!2c_2v_0\!-\!b_2e^{-r\Phi(v_0)}(\lambda_0\!-\!\lambda_0v_0r\Phi'(v_0))\big)v\!-\!b_2e^{-r\Phi(v_0)}v_0\lambda \\
\int_0^L(b_1\lambda_0 r\Phi'(v_0)e^{-2r\Phi(v_0)}\!-\!c_1e^{-r\Phi(v_0)})v\!-\!b_1e^{-2r\Phi(v_0)}\lambda dx
 \end{array}\right).
 \end{split}
\end{equation}
Similar as the arguments for (\ref{38}), we can show that $D_{(v,\lambda)}\mathcal{T}(v_0,\lambda_0,\epsilon): \mathcal{X}  \times \mathbb{R}^+ \rightarrow \mathcal{Y} \times \mathbb{R}$ is a Fredholm operator with zero index.  For bifurcation to occur at $(\bar v,\bar \lambda)$, we need to check the following necessary condition,
\begin{equation}\label{49}
\mathcal{N}(D_{(v,\lambda)}\mathcal{T}(v,\lambda,\epsilon))\not =\{0\},
\end{equation}
where $\mathcal{N}$ denotes the null space.  First, we claim that if $(v,\lambda)\in \mathcal{N}(D_{(v,\lambda)}\mathcal{T}(\bar{v},\bar{\lambda},\epsilon))$, then $\lambda=0$.  In fact, assume that $(v,\lambda)$ satisfies the following system
\begin{equation}\label{410}
\left\{
\begin{array}{ll}
\epsilon v''+\big((a_2-c_2\bar{v})\bar v r \Phi'(\bar v)-c_2\bar v\big)v-b_2\bar v e^{-r\Phi(\bar v)}\lambda=0,~x \in (0,L),\\
\int_0^L\big(-c_1+(a_1-c_1\bar v)r\Phi'(\bar v)\big)v-e^{-r\Phi(\bar v)}b_1\lambda dx=0,\\
v'(0)=v'(L)=0.
\end{array}
\right.
\end{equation}
Integrating the first equation in (\ref{410}) over $(0,L)$, we have
\[((a_2-c_2\bar{v})r\Phi'(\bar v)-c_2)\int_0^Lvdx=b_2\lambda e^{-r\Phi(\bar v)}L.\]
The second equation in (\ref{410}) implies that
\[((a_1-c_1\bar v)r\Phi'(\bar v)-c_1)\int_0^Lvdx=b_1\lambda e^{-r\Phi(\bar v)}L.\]
If $\lambda \neq 0$, we equate the coefficients of the two equations above to obtain
\[\big((a_2-c_2\bar{v})r\Phi'(\bar v)-c_2\big)b_1-\big((a_1-c_1\bar{v})r\Phi'(\bar v)-c_1\big)b_2=0,\]
which implies through direct calculation that
\[\bar{v}=\frac{a_2b_1-a_1b_2}{b_1c_2-b_2c_1}-\frac{1}{r\Phi'(\bar v)},\]
however, this contradicts to the fact that
\[\bar v=\frac{a_2b_1-a_1b_2}{b_1c_2-b_2c_1},\]
therefore $\lambda=0$ as claimed and (\ref{410}) becomes
\begin{equation}\label{411}
\left\{
\begin{array}{ll}
\epsilon  v''+\big((a_2-c_2\bar{v})r\Phi'(\bar v)\bar v-c_2\bar v\big)v=0,~x \in (0,L),\\
\int_0^L \big((a_1-c_1\bar{v})r\Phi'(\bar v)-c_1\big)v dx=0,\\
~v'(0)=v'(L)=0.
\end{array}
\right.
\end{equation}
It is easy to see that (\ref{411}) has nonzero solutions if and only if
\begin{equation}\label{412}
\frac{(a_2-c_2\bar{v})r\Phi'(\bar v)-c_2}{\epsilon}\bar v=\big(\frac{n \pi}{L}\big)^2,~ n\in \mathbb N^+.
\end{equation}
First of all, $n=0$ can be easily ruled out in (\ref{412}) since $\frac{a_1}{a_2}\neq \frac{c_1}{c_2}$.  Otherwise, we must have that $v$ is a constant from the first equation in (\ref{411}) and $v\equiv0$ thanks to the integral constraint, which contradicts the condition (\ref{49}).  For $n\neq0$, the local bifurcation might occur at $(\bar{v},\bar{\lambda},\epsilon_n)$ with
\begin{equation}\label{413}
\epsilon_n=\frac{\big((a_2-c_2\bar{v})r\Phi'(\bar v)-c_2\big)\bar v}{(n\pi/L)^2}>0,~ n \in \mathbb N^+,
\end{equation}
provided that $(a_2-c_2\bar v)r\Phi'(\bar v)>c_2>0$.  Furthermore the null space
\[\mathcal{N}(D_{(v,\lambda)}\mathcal{T}(\bar{v},\bar{\lambda},\epsilon_n))=\text{span}\big\lbrace ( \cos \frac{n\pi x}{L}, 0)  \big\rbrace,~n\in \mathbb N^+,\]
and it has $\dim \mathcal{N}(D_{(v,\lambda)}\mathcal{T}(\bar{v},\bar{\lambda},\epsilon_n))=1.$

Having the potential bifurcation values, we can now proceed to verify in the following theorem that the local bifurcation does occur at $(\bar{v},\bar{\lambda},\epsilon_n)$.
\begin{theorem}\label{thm43}
Suppose that the conditions (\ref{12}) and $c_2<(a_2-c_2\bar v)r\Phi'(\bar v)$ are satisfied.  For each $n\in \mathbb N^+$, there exist $\delta>0$ and continuous functions $s\in(-\delta, \delta):\rightarrow (v_n(x,s),\lambda_n(s),\epsilon_n(s)) \in \mathcal{X} \times \mathbb{R}^+ \times \mathbb{R}^+$, with
\begin{equation}\label{414}
\epsilon_n(0)=\epsilon_n,~(v_n(x,s),\lambda_n(s))=(\bar{v},\bar{\lambda})+s(\cos \frac{n\pi x}{L},0) +o(s),
\end{equation}
such that $(v_n(x,s),\lambda_n(s))$ solves the system (\ref{44}).  Moreover, all nontrivial solutions of (\ref{44}) near ($\bar{v},\bar{\lambda},\epsilon_n)$ take the form in (\ref{414}).
\end{theorem}
\begin{proof}
To make use of the local bifurcation analysis of Crandall and Rabinowitz \cite{CR}, we only need to show the transversality condition
\begin{equation}\label{415}
\frac{d}{d \epsilon} \left(D_{(v,\lambda)}\mathcal{T}(\bar{v},\bar{\lambda},\epsilon)\right)(v_n,\lambda_n)\Big\vert_{\epsilon=\epsilon_n} \notin \mathcal{R}(D_{(v,\lambda)}\mathcal{T}(\bar{v},\bar{\lambda},\epsilon_n)),
\end{equation}
where
\[\frac{d}{d \epsilon} \left(D_{(v,\lambda)}\mathcal{T}(\bar{v},\bar{\lambda},\epsilon)\right)(v_n,\lambda_n)\Big\vert_{\epsilon=\epsilon_n}=\left(
 \begin{array}{c}
(\cos\frac{n\pi x}{L})''\\
0 \end{array}
 \right). \]
If not and we suppose that there exists a nontrivial solution $v \in \mathcal{X}$ to the following problem
\begin{equation}\label{416}
\left\{
\begin{array}{ll}
\epsilon v''+\big((a_2-c_2\bar{v})\bar v r \Phi'(\bar v)-c_2\bar v\big)v=(\cos\frac{n\pi x}{L})'',x\in(0,L),\\
\int_0^L(-c_1+(a_1-c_1\bar v)r\Phi'(\bar v))v-e^{-r\Phi(v_0)}b_1\lambda dx=0,\\
v'(0)=v'(L)=0.
\end{array}
\right.
\end{equation}
By the same analysis that leads to the claim under (\ref{49}), we have that $\lambda=0$ in (\ref{416}), and it becomes
\begin{equation}\label{417}
\left\{
\begin{array}{ll}
\epsilon v''+\big((a_2-c_2\bar{v}) r \Phi'(\bar v)-c_2 \big)\bar vv=(\cos\frac{n\pi x}{L})'',~~x\in(0,L),\\
v'(0)=v'(L)=0.
\end{array}
\right.
\end{equation}
However, this reaches a contradiction to the Fredholm Alternative since $\cos\frac{n\pi x}{L}$ is in the kernel of the operator on the left hand side of (\ref{417}).  Then we have verified (\ref{415}) and thus conclude the proof of Theorem \ref{thm43}.
\end{proof}
According to Theorem \ref{thm43}, shadow system (\ref{47}) always admits positive solutions as long as $\epsilon$ is small.
\subsection{Stability of bifurcating solutions from $(\bar{v},\bar{\lambda},\epsilon_n)$}

In this section, we proceed to investigate the stability or instability of the spatially inhomogeneous solution $(v_n(s,x),\lambda_n(s,x))$ that bifurcates from $(\bar{v},\bar{\lambda})$ at $\epsilon=\epsilon_n$. Here again the stability refers to that of the bifurcating solution taken as an equilibrium to the time-dependent system of (\ref{44}).  Similar as the analysis in Section 3, we write the following expansions
\begin{equation}\label{418}
\left\{
\begin{array}{ll}
v_n(s,x)=\bar{v}+s\cos \frac{n \pi  x}{L}+s^2\varphi_2(x)+s^3\varphi_3(x)+o(s^3),\\
\lambda_n(s,x)=\bar{\lambda}+s^2\bar{\lambda}_2+s^3\bar{\lambda}_3+o(s^3),\\
\epsilon_n(s)=\epsilon_n+s\mathcal{K}_1+s^2\mathcal{K}_2+o(s^2),
\end{array}
\right.
\end{equation}
where $\varphi_i \in \mathcal{X}$ satisfies $\int_0^L \varphi_i\cos \frac{n\pi x}{L} dx=0$ for $i=2,3$ and $\bar{\lambda}_2, \mathcal{K}_1, \mathcal{K}_2$ are positive constants to be determined.

For notational simplicity, we denote in (\ref{47}),
\[f(\lambda,v)=\Big(a_2-b_2 \lambda_\epsilon e^{-r\Phi(v_\epsilon )}-c_2 v_\epsilon\Big)v_\epsilon
\]
and
\[g(\lambda,v)=(a_1-b_1 \lambda_\epsilon e^{-r\Phi(v_\epsilon)}-c_1 v_\epsilon) e^{-r\Phi(v_\epsilon)};\]  moreover, we introduce the notations
\[\bar f_v=\frac{\partial f(v,\lambda)}{\partial v}\vert_{(v,\lambda)=(\bar v,\bar\lambda)},~~\bar f_\lambda=\frac{\partial f(v,\lambda)}{\partial \lambda}\vert_{(v,\lambda)=(\bar v,\bar\lambda)}\]
and in the same manner we can define $\bar f_{v\lambda}, \bar f_{vvv}, \bar g_{v}, \bar g_{\lambda}, \bar g_{v\lambda}, \bar g_{vv}$, etc.

Substituting (\ref{418}) into (\ref{47}) and collecting the $s^2$-terms, we obtain
\begin{equation}\label{419}
\epsilon_n\varphi_2''-\mathcal{K}_1\bigg(\frac{n\pi}{L}\bigg)^2\cos\frac{n\pi x}{L}+\bar{f}_v\varphi_2+\bar{f}_\lambda
\bar{\lambda}_2+\frac{\bar{f}_{vv}}{2}\cos^2\frac{n\pi x}{L}=0.
\end{equation}
Multiplying (\ref{419}) by $\cos\frac{n\pi x}{L}$ and integrating over $(0,L)$ by parts, we have
\begin{equation}\label{420}
\frac{n^2\pi^2}{L}\mathcal{K}_1=\Big(-\epsilon_n\big(\frac{n\pi}{L}\big)^2+\bar{f}_v\Big)\int_0^L\varphi_2\cos\frac{n\pi x}{L}=0,
\end{equation}
therefore $\mathcal{K}_1=0$ and the bifurcation branch around $(\bar{v},\bar{\lambda},\epsilon_n)$ is of pitch-fork type for all $n\in \mathbb N^+.$

Equating the $s^3$-terms from (\ref{47}) gives
\begin{eqnarray}\label{421}
\epsilon_n\varphi_3''+\bar{f}_v\varphi_3+\bar{f}_\lambda\bar{\lambda}_3-\mathcal{K}_2\bigg(\frac{n\pi}{L}\bigg)^2\cos\frac{n\pi x}{L}\\ \nonumber
+\bigg(\bar{f}_{vv}\varphi_2+\bar{f}_{v\lambda}\bar{\lambda}_2\bigg)\cos\frac{n\pi x}{L}+\frac{\bar{f}_{vvv}}{6}\cos^3\frac{n\pi x}{L}=0. \nonumber
\end{eqnarray}
Testing (\ref{421}) by $\cos\frac{n\pi x}{L}$, we obtain from straightforward calculations that
\begin{equation}\label{422}
\frac{n^2\pi^2}{2L}\mathcal{K}_2=\frac{\bar{f}_{vv}}{2}\Big(\int_0^L\varphi_2\cos\frac{2n\pi x}{L} dx+\int_0^L\varphi_2 dx\Big)+\frac{\bar{f}_{v\lambda}\bar{\lambda}_2 L}{2}+\frac{\bar{f}_{vvv}L}{16}.
\end{equation}
Multiplying (\ref{419}) by $\cos\frac{2n\pi x}{L}$ and integrating it over $(0,L)$ by parts, then thanks to $\mathcal{K}_1=0$ in (\ref{420}), we conclude from straightforward calculations that
\begin{equation}\label{423}
\int_0^L\varphi_2\cos\frac{2n\pi x}{L} dx=\frac{\bar{f}_{vv}L}{24\bar{f}_v},
\end{equation}
where we have used the fact that $\bar{f}_v=\epsilon_n\big(\frac{n\pi}{L}\big)^2$.
Integrating (\ref{419}) over $(0,L)$ by parts, we have
\begin{equation}\label{424}
\bar{f}_v\int_0^L\varphi_2 dx+\bar{f}_\lambda\bar{\lambda}_2L+\frac{\bar{f}_{vv}L}{4}=0.
\end{equation}
Moreover, we integrate the $s^2$-terms from the second equation of (\ref{47}) over $(0,L)$
\begin{equation}\label{425}
\bar{g}_v\int_0^L\varphi_2 dx+\bar{g}_\lambda\bar{\lambda}_2L+\frac{\bar{g}_{vv}L}{4}=0.
\end{equation}
Then we conclude from (\ref{424}) and (\ref{425}) that
\begin{eqnarray}\label{426}
\int_0^L\varphi_2dx=-\frac{\big(\bar{f}_{vv}\bar{g}_\lambda-\bar{g}_{vv}\bar{f}_\lambda\big)L}{4\big(\bar{f}_v
\bar{g}_\lambda-\bar{g}_v\bar{f}_\lambda\big)},~~\bar{\lambda}_2=-\frac{\bar{f}_v\bar{g}_{vv}-\bar{g}_v\bar{f}_{vv}}{4\big(\bar{f}_v\bar{g}_\lambda-\bar{g}_v\bar{f}_
\lambda\big)}.
\end{eqnarray}
Substituting (\ref{423}) and (\ref{426}) into (\ref{422}), we see that $\mathcal{K}_2$ equation becomes
\begin{eqnarray}\label{427}
\frac{n^2\pi^2}{2L}\mathcal{K}_2 \!\!\!\!\!&=& \!\!\!\!\! \frac{\bar{f}_{vv}}{2}\bigg(\frac{\bar{f}_{vv}L}{24\bar{f}_v}-\frac{\big(\bar{f}_{vv}\bar{g}_\lambda\!-\!\bar{g}_{vv}
\bar{f}_\lambda\big)L}{4\big(\bar{f}_v\bar{g}_\lambda-\bar{g}_v\bar{f}_\lambda\big)}\bigg)\!-\!\frac{\bar{f}_{v\lambda}\big(
\bar{f}_v\bar{g}_{vv}-\bar{g}_v\bar{f}_{vv}\big)L}{8\big(\bar{f}_v\bar{g}_\lambda-\bar{g}_v\bar{f}_
\lambda\big)}\!+\!\frac{\bar{f}_{vvv}L}{16}\nonumber\\
&=&\!\!\!\!\!\frac{\Big(\bar{g}_{vv}\big(\bar{f}_{vv}\bar{f}_\lambda-
\bar{f}_v\bar{f}_{v\lambda}\big)\!-\!\bar{f}_{vv}\big(\bar{f}_{vv}\bar{g}_\lambda\!-\!\bar{g}_v\bar{f}_{v\lambda}\big)\!\Big)L}{8\big(\bar{f}_v\bar{g}_\lambda\!-\!\bar{g}_v\bar{f}_\lambda\big)}
\!+\!\frac{\bar{f}_{vv}^2L}{48\bar{f}_v}\!+\!\frac{\bar{f}_{vvv}L}{16}.
\end{eqnarray}
To evaluate $\mathcal{K}_2$ in (\ref{427}), we derive the following partial derivatives of $f$ and $g$ at $(\bar{v},\bar{\lambda})$ through straightforward calculations
\begin{equation}\label{428}
\bar{f}_v=\big(a_2-c_2\bar{v}\big)r\bar{v}\Phi'(\bar{v})-c_2\bar{v},~\bar{f}_\lambda=-b_2\bar{v}e^{-r\Phi(\bar{v})},
\end{equation}
\begin{equation}\label{429}
\bar{f}_{vv}=\big(a_2-c_2\bar{v}\big)\big(-\bar{v}\big(\Phi'(\bar{v})\big)^2r^2+\big(2\Phi'(\bar{v})+\bar{v}
\Phi''(\bar{v})\big)r\big)-2c_2,
\end{equation}
\begin{equation}\label{430}
\bar{f}_{v\lambda}=b_2e^{-r\Phi(\bar{v})}\big(\bar{v}\Phi'(\bar{v})r-1\big)
\end{equation}
and
\begin{equation}\label{431}
\bar{f}_{vvv}=\big(a_2\!-\!c_2\bar{v}\big)\Big(\bar{v}\Phi'^3(\bar{v})r^3\!-\!3\Phi'(\bar{v})\big(\Phi'(\bar{v})\!+\!
\Phi''(\bar{v})\big)r^2\!+\!\big(3\Phi''(\bar{v})\!+\!\bar{v}\Phi'''(\bar{v})\big)r\Big),
\end{equation}
together with
\begin{equation}\label{432}
 \bar{g}_v=\Big(\big(a_1-c_1\bar{v}\big)\Phi'(\bar{v})r-c_1\Big)e^{-r\Phi(\bar{v})},~ \bar{g}_\lambda=-b_1e^{-2r\Phi(\bar{v})}
\end{equation}
and
\begin{equation}\label{433}
\bar{g}_{vv}=re^{-r\Phi(\bar{v})}\big(-3\big(\Phi'(\bar{v})\big)^2\big(a_1-c_1\bar{v}\big)r+\big(a_1-c_1\bar{v}\big)\Phi''(\bar{v})+2c_1\Phi'(
\bar{v})\big).
\end{equation}
We now proceed to calculate $\mathcal{K}_2$ in (\ref{427}) with (\ref{428})--(\ref{433}) as follows.  First of all, in the light of the fact $b_1(a_2-c_2\bar{v})=b_2(a_1-c_1\bar{v})$, we have from (\ref{428}) and (\ref{432}) that
\[\bar{f}_v\bar{g}_\lambda-\bar{g}_v\bar{f}_\lambda=\big(b_1c_2-b_2c_1\big)\bar{v}e^{-2r\Phi(\bar{v})};\]
moreover, we have
\begin{eqnarray}\label{434}
&&\bar{g}_{vv}\big(\bar{f}_{vv}\bar{f}_\lambda-
\bar{f}_v\bar{f}_{v\lambda}\big)
-\bar{f}_{vv}\big(\bar{f}_{vv}\bar{g}_\lambda-\bar{g}_v\bar{f}_{v\lambda}\big)\nonumber\\
&=&\bigg(\bar{f}_{vv}\big(b_1\bar{f}_{vv}e^{-r\Phi(\bar{v})}-b_2\bar{v}\bar{g}_{vv}\big)+\bar{f}_{v\lambda} \Big(\bar{f}_{vv}\big(\big(a_1-c_1\bar{v}
\big)r\Phi'(\bar{v})-c_1\big)\nonumber\\
&&-\bar{g}_{vv}e^{r\Phi(\bar{v})}\big(\big(a_2-c_2\bar{v}\big)r\Phi'(\bar{v})-c_2\big)\bar{v}\Big)\bigg)e^{-r
\Phi(\bar{v})}\\
&=& \Big(\big(b_1\bar{f}_{vv}-b_2\bar{v}e^{r\Phi(\bar{v})}\bar{g}_{vv}\big)\big(\bar{f}_{vv}e^{-r\Phi(\bar{v})}
+\bar{f}_{v\lambda}e^{-r\Phi(\bar{v})}\Phi'(\bar{v})\bar{\lambda} r
\big)\nonumber\\
&&+\bar{f}_{v\lambda}\big(c_2\bar{v}e^{r\Phi(\bar{v})}\bar{g}_{vv}-c_1\bar{f}_{vv}\big)\Big)e^{-r\Phi(\bar{v})}; \nonumber
\end{eqnarray}
furthermore, we apply (\ref{429}), (\ref{430}) and (\ref{433}) and obtain in (\ref{434}) that
\begin{eqnarray*}\label{435}
&&\bar{f}_{vv}e^{-r\Phi(\bar{v})}+\bar{f}_{v\lambda}e^{-r\Phi(\bar{v})}\Phi'(\bar{v})\bar{\lambda} r\nonumber\\
&=&\Big(\big(a_2-c_2\bar{v}\big)\big(-\bar{v}\big(\Phi'(\bar{v})\big)^2r^2+\big(2\Phi'(\bar{v})+\bar{v}
\Phi''(\bar{v})\big)r\big)-2c_2+b_2\bar{\lambda} \nonumber\\
&&\big(r\bar{v}\Phi'(\bar{v})-1\big)\Phi'(\bar{v})re^{-r\Phi(\bar{v})}\Big)e^{-r\Phi(\bar{v})}\nonumber\\
&=&\Big(-2c_2+\big(a_2-c_2\bar{v}\big)\big(r\Phi'(\bar{v})+r\bar{v}\Phi''(\bar{v})\big)\Big)e^{-r\Phi(\bar{v})},
\end{eqnarray*}
then we conclude from straightforward calculations with (\ref{428}) and (\ref{429})
\begin{eqnarray}\label{436}
&&\frac{\bar{g}_{vv}\big(\bar{f}_{vv}\bar{f}_\lambda-
\bar{f}_v\bar{f}_{v\lambda}\big)-\bar{f}_{vv}\big(\bar{f}_{vv}\bar{g}_\lambda-\bar{g}_v\bar{f}_{v\lambda}\big)}
{\bar{f}_v\bar{g}_\lambda-\bar{g}_v\bar{f}_\lambda}\nonumber\\
&=&\frac{1}{\big(b_1c_2-b_2c_1\big)\bar v}\Big(
\big(b_1\bar{f}_{vv}-b_2\bar{v}e^{r\Phi(\bar{v})}\bar{g}_{vv}\big)\big(-2c_2+\big(a_2-c_2\bar{v}\big)\big(r\Phi'
(\bar{v})\nonumber\\
&&+r\bar{v}\Phi''(\bar{v})\big)\big)+\bar{f}_{v\lambda}e^{r\Phi(\bar{v})}\big(c_2\bar{v}\bar{g}_{vv}e^{r\Phi(\bar{v})}
-c_1\bar{f}_{vv}
\big)\Big).
\end{eqnarray}
Substituting (\ref{436}) into (\ref{427}), we arrive at the following form for $\mathcal{K}_2$
\begin{equation}\label{437}
\begin{split}
\frac{n^2\pi^2}{2L}\mathcal{K}_2 &=\frac{L}{8\bar{v}\big(b_1c_2-b_2c_1\big)}\Bigg(\big(b_1
\bar{f}_{vv}-b_2\bar{v}e^{r\Phi(\bar{v})}\bar{g}_{vv}\big)\big(-2c_2+\big(a_2-c_2\bar{v}\big)\big(r\Phi'(\bar{v})\\
&+r\bar{v}\Phi''(\bar{v})\big)\big)+c_2\bar{v}e^{2r\Phi(\bar{v})}\bar{f}_{v\lambda}\bar{g}_{vv}+\Big(\frac{\bar{v}
\big(b_1c_2-b_2c_1\big)\bar{f}_{vv}}{6f_v}-c_1e^{r\Phi(\bar{v})}\bar{f}_{v\lambda}\Big)\bar{f}_{vv}\\
&+\frac{\bar{v}\big(b_1c_2-b_2c_1\big)\bar{f}_{vvv}}{2}\Bigg).
\end{split}
\end{equation}

$\mathcal{K}_2$ in (\ref{437}) is extremely complicated and it is difficult to evalute its sign in order to determine the turning direction of each bifurcation branch, therefore, for the simplicity of calculations, we choose $\Phi(v)=v$ and consider (\ref{47}) in the strong competition case with $b_1=0$.  Actually, we see that if $b_1=0$, both $(\bar u, \bar v)$ and the small-amplitude bifurcating solutions $(u_k(s,x),v_k(s,x))$ are unstable, therefore, we are motivated to study the solutions that have large amplitude.  Indeed, we shall see that, the shadow system (\ref{47}) has solutions with interior transition layer when $b_1=0$, therefore, this assumption does not inhabit the applications of our original system (\ref{35}) and the shadow system (\ref{47}) in the mathematical modeling of interspecific segregation.

Assuming $\Phi(v)=v$ and $b_1=0$, we see that (\ref{437}) becomes
\begin{align}\label{438}
\frac{n^2\pi^2}{2L}\mathcal{K}_2 =& -\frac{L}{48c_1\big(\big(a_2\!-\!c_2\bar{v}\big)r\!-\!c_2\big)}\Big(\big(a_2\!-\!c_2\bar{v}\big)^2r^2\big(2a_1\bar{v}r^2-17a_1r^1
\!+\!8c_1\big)+\\
&\big(a_2\!-\!c_2\bar{v}\big)r\big(9a_1c_2\bar{v}r^2\!+\!41a_1c_2r\!-\!16c_1c_2\big)\!-\!
12a_1c_2^2\bar{v}r^2\!-\!24a_1c_2^2r\!+\!8c_1c_2^2\Big),\nonumber
\end{align}
where $\bar{v}=\frac{a_1}{c_1}$.  Moreover, for notational simplicity, we denote
\[\theta=\big(a_2-c_2\bar v\big)r,\]
then (\ref{413}) implies that bifurcation occurs at $(\bar v,\bar \lambda,\epsilon_n)$ only if $\theta>c_2$ and we will assume this condition throughout the rest of this section.  Now we see that (\ref{438}) becomes
\begin{equation}\label{439}
\frac{n^2\pi^2}{2L}\mathcal{K}_2=\frac{F(\theta)}{48\big(\theta-c_2\big)}=\frac{\alpha\theta^2+\beta\theta+\gamma}{48\big
(\theta-c_2\big)},
\end{equation}
where we have employed the notations
\begin{equation}\label{440}
\alpha=-2a_1\bar{v}r^2+17a_1r-8c_1,
\end{equation}
\begin{equation}\label{441}
\beta=-9a_1c_2\bar{v}r^2-41a_1c_2r+16c_1c_2,
\end{equation}
and
\begin{equation}\label{442}
\gamma=12a_1c_2^2\bar{v}r^2+24a_1c_2^2r-8c_1c_2^2.
\end{equation}
To determine the sign of $\mathcal{K}_2$ for all $\theta>c_2$, we first have through straightforward calculations that the determinant of the quadratic function $F(\theta)$ in (\ref{439}) is
\begin{equation*}
\beta^2-4\alpha\gamma=177a^2_1c^2_2r^2\Big(\big(\bar{v}r+\frac{57}{177}\big)^2+\frac{6228}{177^2}\Big)>0,
\end{equation*}
therefore $F(\theta)=0$ always have two roots
\begin{equation}\label{443}
\theta_1=\frac{-\beta-\sqrt{\beta^2-4\alpha\gamma}}{2\alpha},~\theta_2=\frac{-\beta+\sqrt{\beta^2-4\alpha\gamma}}{2\alpha}.
\end{equation}
We now present the following results concerning (\ref{439}).
\begin{proposition}\label{prop3}
Suppose that $\Phi(v)=v$ and $b_1=0$ in (\ref{47}) and the condition (\ref{12}) holds.  Denote $\theta=\big(a_2-c_2\bar v\big)r$ and assume that $\theta>c_2$.  For each $n\in \mathbb N^+$, we have the following facts about $\mathcal{K}_2$ in (\ref{418}):

(i). if $r\in(0,\frac{1}{2\bar{v}})\cup(\frac{8}{\bar{v}},\infty)$, $\mathcal{K}_2>0$ for $\theta\in(c_2,\theta_1)$, and $\mathcal{K}_2<0$, for $\theta\in(\theta_1,\infty)$;

(ii). if $r\in(\frac{1}{2\bar{v}},\frac{8}{\bar{v}})$, $\mathcal{K}_2>0$ for $\theta\in(c_2,\theta_1)\cup(\theta_2,\infty)$, and $\mathcal{K}_2<0$ for $\theta\in(\theta_1,\theta_2)$;

(iii). if $r=\frac{1}{2\bar{v}}$, $\mathcal{K}_2>0$ for $\theta\in(c_2,\frac{28c_2}{27})$, and $\mathcal{K}_2<0$ for $\theta\in(\frac{28c_2}{27},\infty)$;

(iv). if $r=\frac{8}{\bar{v}}$, $\mathcal{K}_2>0$ for $\theta\in(c_2,\frac{120c_2}{111})$, and $\mathcal{K}_2<0$ for $\theta\in(\frac{120c_2}{111},\infty)$,
where we have $\bar v=\frac{a_1}{c_1}$ in (i)-(iv).
\end{proposition}
The results in Proposition \ref{prop3} are presented in Figure \ref{fig2}.
\begin{proof} For all $\theta>c_2$, we readily see from (\ref{439}) that $\mathcal{K}_2$ has the same sign as $F(\theta)=\alpha\theta^2+\beta\theta+\gamma$.  Moreover, in light of (\ref{439})--(\ref{442}), we have from straightforward calculations that $F(c_2)=a_1c_2^2r^2>0$.  On the other hand, we notice that (\ref{440}) is equivalent as
\begin{eqnarray*}\label{445}
\alpha=-c_1\big(2\bar{v}r-1)\big(r-\frac{8}{\bar{v}}\big),
\end{eqnarray*}
and $\alpha=0$ if $r=\frac{1}{2\bar{v}}$ or $r=\frac{8}{\bar v}$, then it gives that $F(\theta)=-\frac{27c_1c_2}{4}\theta+7c_1c_2^2$ and $F(\theta)=-888c_1c_2\theta+960c_1c_2^2$ respectively.  Therefore, \emph{(iii)} and \emph{(iv)} follows from (\ref{439}) immediately.
\begin{figure}[!htb]
\centering
\minipage{.34\textwidth}\centering
  \includegraphics[width=0.8in]{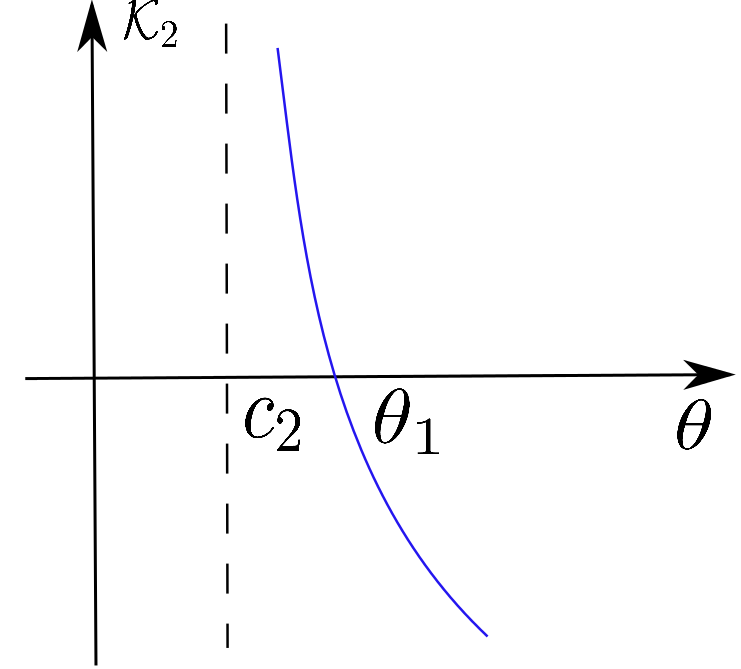}\caption*{$r\in(0,\frac{1}{2\bar v})\cup(\frac{8}{\bar v},\infty)$}
\endminipage
\minipage{.28\textwidth}\centering
  \includegraphics[width=0.8in]{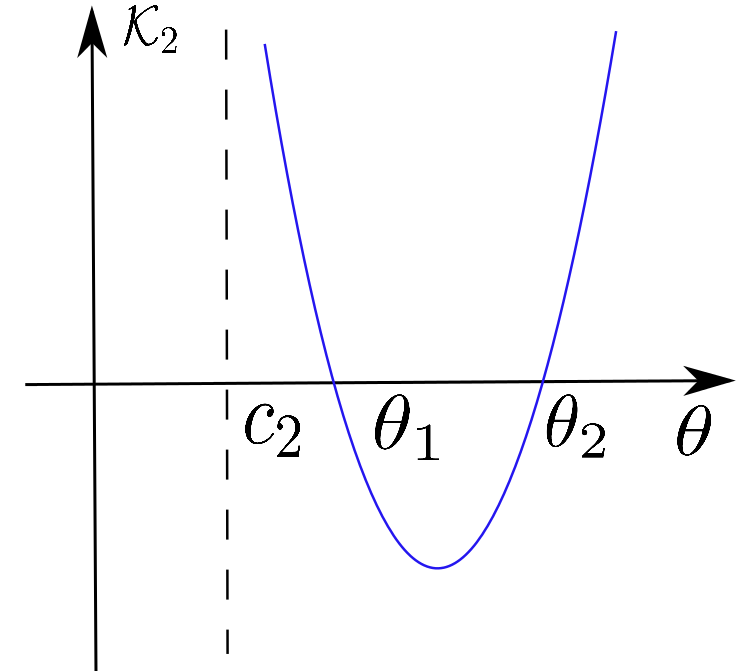}\caption*{$r\in(\frac{1}{2\bar v},\frac{8}{\bar v})$\\}
\endminipage
\minipage{.18\textwidth} \centering
  \includegraphics[width=0.8in]{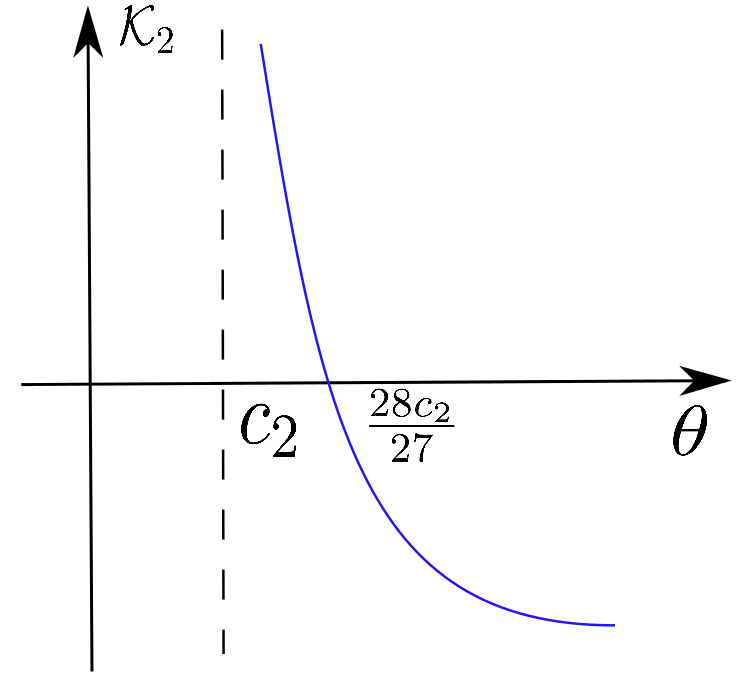}\caption*{$r=\frac{1}{2\bar v}$\\}
\endminipage
\minipage{.18\textwidth}\centering
  \includegraphics[width=0.8in]{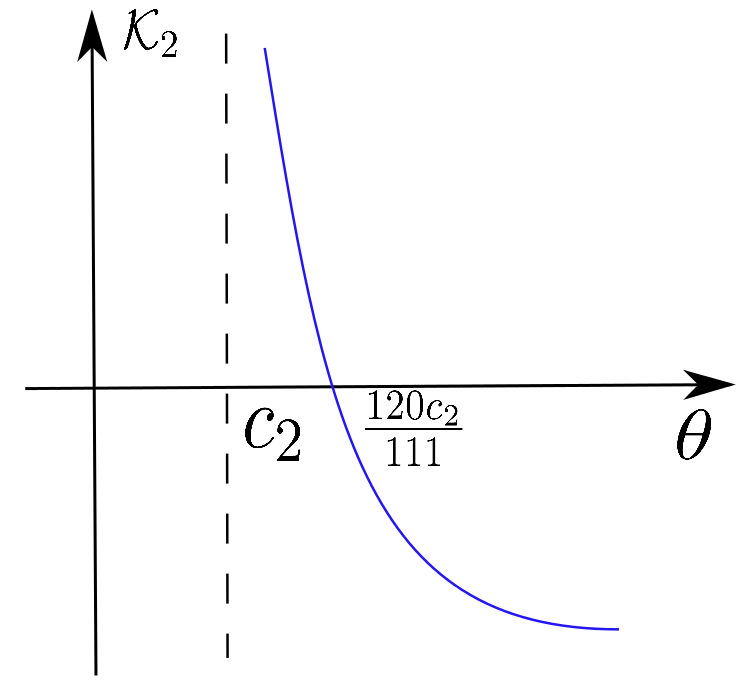}\caption*{$r=\frac{8}{\bar v}$\\}
\endminipage
\caption{Graphs of $\mathcal{K}_2$ as a function of $\theta=\big(a_2-c_2\bar v\big)r$; the assumption $\theta>c_2$ is required such that bifurcation occurs at $(\bar v,\bar \lambda,\epsilon_n)$.}\label{fig2}
\end{figure}

To prove \emph{(i)} and \emph{(ii)},  we compare axis of symmetry $\theta^*=-\frac{\beta}{2\alpha}$ of the parabola $F(\theta)$ with $c_2$
\[c_2-\theta^*=\frac{\beta+2c_2\alpha}{2\alpha}=-\frac{13a_1c_2\bar{v}r^2+7a_1c_2r}{2\alpha},\]
then it readily implies that sgn$(c_2-\theta^*)=$sgn$(-\alpha)$.  If $r\in(0,\frac{1}{2\bar{v}})\cup(\frac{8}{\bar{v}},\infty)$, we have $c_2-\theta^*>0$ since $\alpha<0$, then the parabola $F(\theta)$ opens down and with its axis of symmetry to the left of $c_2$.  Therefore, $F(\theta)>0,\mathcal{K}_2>0$, if $\theta\in(c_2,\beta_1)$, and $F(\theta)<0,\mathcal{K}_2<0$, if $\theta\in(\beta_1,\infty)$.  This completes the proof of \emph{(i)}.  By the same arguments, we can show \emph{(ii)} and this proves Proposition \ref{prop3}.
\end{proof}

Since $F(c_2)>0$, we see that $\mathcal{K}_2>0$ if $\theta>c_2$ is small, i.e, for all $(a_2-c_2\bar v )r-c_2>0$ being small.  We shall interpret this observation by its relevance with our results and mathematical modeling.  Before this, we present the following theorem on the stability of the bifurcating solution $(v_n(s,x),\lambda_n(s))$ established in Theorem \ref{thm43}.  Here the stability refers to the stability of the inhomogeneous solutions taken as an equilibrium to the time-dependent counterpart to (\ref{47}).
\begin{theorem}\label{thm44}
Suppose that all the conditions in Proposition \ref{prop3} are satisfied.  Then for each $n\in \mathbb N^+$, the bifurcation curve $\Gamma_n(s)$ at $(\bar v,\bar \lambda,\epsilon_n)$ is of pitch-fork type.  Moreover, the bifurcating solution $(v_n(s,x),\lambda_n(s))$ is always unstable for all $n\geq 2$; $(v_1(s,x),\lambda_1(s))$ is unstable if $\mathcal{K}_2>0$ and it is asymptotically stable if $\mathcal{K}_2<0$, where $\mathcal{K}_2$ is given by (\ref{439}).
\end{theorem}

The local bifurcation branches described in Theorem \ref{thm44} are schematically presented in Figure \ref{fig3}.  In contrast to the bifurcation branches in Figure \ref{fig1}, small $\epsilon$ supports stable bifurcating solutions as $\Gamma_1(s)$ turns to the left, while large $\epsilon$ tends to exclude the nontrivial solutions $\Gamma_1(s)$ turns to the right.
\begin{figure}[!htb]
\minipage{.4\textwidth}
  \includegraphics[width=1.2in]{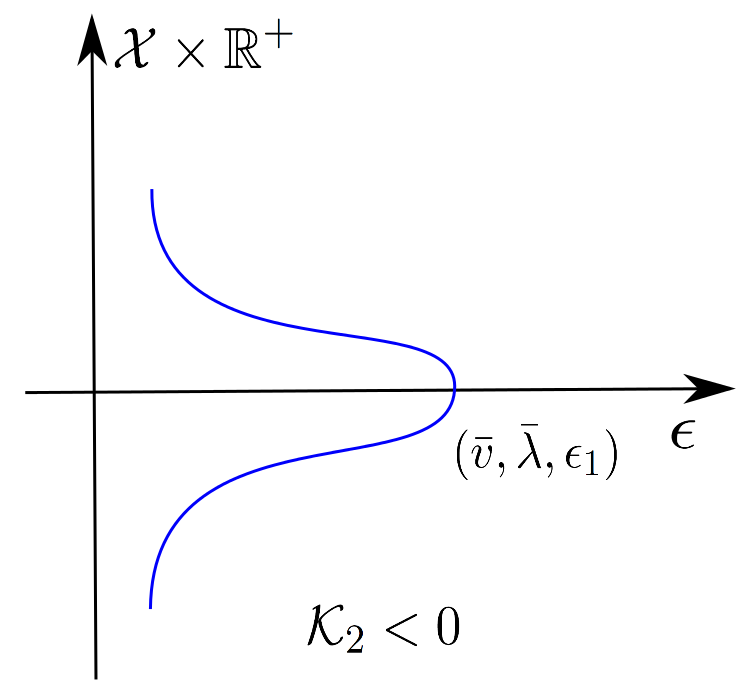}
%  \caption*{An illustration of the bifurcation branch for stable bifurcation branch}\label{fig:awesome_image1}
\endminipage
\minipage{.2\textwidth}
  \includegraphics[width=1.2in]{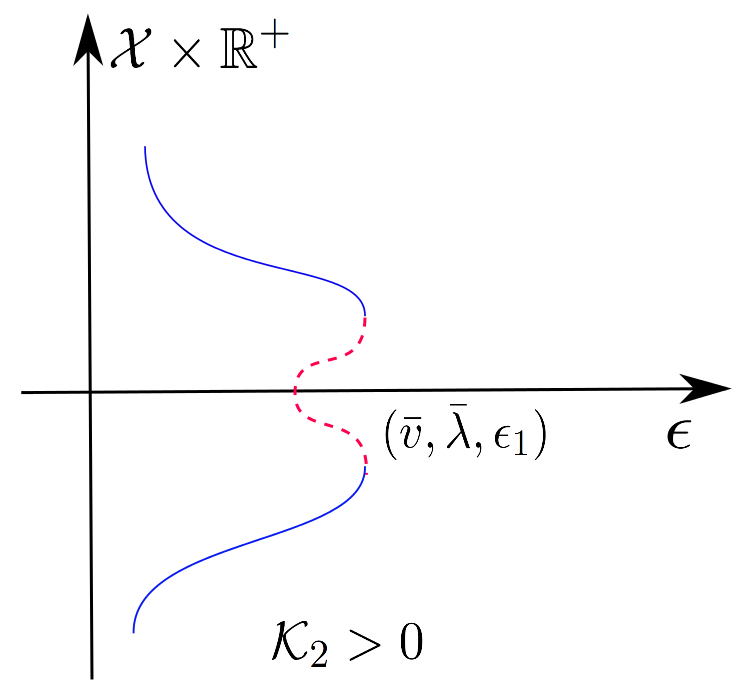}
%  \caption*{An illustration of the bifurcation branch for unstable bifurcation branch}\label{fig:awesome_image2}
\endminipage
\caption{Pitchfork type bifurcation branches.  The solid line represents stable bifurcating solution $(v_1(s,x),\lambda_1(s),\epsilon_1(s))$ and the dashed line represents unstable solution.}\label{fig3}
\end{figure}

Following the results of \cite{CR2}, we can prove this theorem by the same analysis that leads to Theorem \ref{thm32}.  For the sake of completeness, we sketch it as follows.
To study the stability of the bifurcating solution from $(\bar v,\bar \lambda,\epsilon_n)$, we linearize (\ref{47}) at $(v_n(s,x),\lambda_n(s),\epsilon_n(s))$.  By the principle of the linearized stability in Theorem 8.6 \cite{CR2}, to show that they are asymptotically stable, we need to prove that the each eigenvalue $\eta$ of the following elliptic problem has negative real part:
\begin{equation}\label{446}
D_{(v,\lambda)}\mathcal{T}(v_n(s,x),\lambda_n(s),\epsilon_n(s))(v,\lambda)=\eta (v,\lambda),~(v,\lambda)\in\mathcal{X} \times \mathbb{R}.
\end{equation}
where $v_n(s,x)$, $\lambda_n(s)$ and $\epsilon_n(s)$ are established in Theorem \ref{thm43}.  On the other hand, we observe that $0$ is a simple eigenvalue of $D_{(v,\lambda)}\mathcal{T}(\bar v,\bar \lambda,\epsilon_n)$ with an eigenspace $\text{span}\{(\cos \frac{n\pi x}{L},0)\}$. It follows from Corollary 1.13 in \cite{CR2} that, there exist an internal $I$ with $\epsilon_n\in I$ and continuously differentiable functions $\epsilon\in I\rightarrow \zeta(\epsilon),~s\in(-\delta,\delta) \rightarrow \eta(s)$ with $\eta(0)=0$ and $\zeta(\epsilon_n)=0$ such that $\eta(s)$ is an eigenvalue of (\ref{446}) and $\zeta(\chi)$ is an eigenvalue of the following eigenvalue problem
\begin{equation}\label{447}
D_{(v,\lambda)}\mathcal{T}(\bar v,\bar \lambda,\epsilon)(v,\lambda)=\zeta(v,\lambda),~(v,\lambda)\in \mathcal{X} \times \mathbb{R};
\end{equation}
moreover, $\eta(s)$ is the only eigenvalue of (\ref{446}) in any fixed neighbourhood of the origin of the complex plane (the same assertion can be made on $\zeta(\epsilon)$).  We also know from \cite{CR} that the eigenfunctions of (\ref{447}) can be represented by $(v(\epsilon,x),\lambda(\epsilon))$ which depend on $\epsilon$ smoothly and can be uniquely determined by $\big(v(\epsilon_n,x),\lambda(\epsilon_n)\big)=\big(\cos \frac{n\pi x}{L},0 \big)$.

\begin{proof} [Proof\nopunct] \emph{of Theorem} \ref{thm44}.
By the same analysis that leads to part \emph{(i)} of Theorem \ref{thm31}, we can show that (\ref{446}) has a positive eigenvalue for each $n\geq2$, therefore $(v_n(s,x),\lambda_n(s),\epsilon_n(s))$ is unstable for $n\geq2$.

Now we proceed to study the stability of $(v_1(s,x),\lambda_1(s),\epsilon_1(s))$.  Similar as the analysis that leads to the claim behind (\ref{49}), we can show that $\lambda=0$ in (\ref{447}).  Then by differentiating (\ref{447}) with respect to $\epsilon$ and setting $\epsilon=\epsilon_1$, we collect the following system from (\ref{48}) in light of $\zeta(\epsilon_1)=0$
\begin{equation}\label{448}
\left\{
\begin{array}{ll}
-\big(\frac{\pi}{L})^2\cos\frac{\pi x}{L}+\epsilon_1\dot{v}''+\big((a_2-c_2\bar{v}) r -c_2\big)\bar v\dot{v}=\dot{\zeta}(\epsilon_1)\cos\frac{\pi x}{L},\\
\int_0^L(-c_1+(a_1-c_1\bar v)r)\dot{v}dx=0,\\
\dot{v}'(0)=\dot{v}'(L)=0,
\end{array}
\right.
\end{equation}
where the dot-sign means the differentiation with respect to $\epsilon$ evaluated at $\epsilon=\epsilon_1$ and in particular $\dot{v}=\frac{\partial v(\epsilon,x)}{\partial \epsilon}\big\vert_{\epsilon=\epsilon_1}$.

Multiplying the first equations of (\ref{448}) by $\cos\frac{\pi x}{L}$ and integrating it over $(0,L)$ by parts, we apply (\ref{413}) with $\Phi'(v)=1$ and obtain from (\ref{448})
\[\dot{\zeta}(\epsilon_1)=-\big(\frac{\pi}{L}\big)^2.\]
According to Theorem 1.16 in \cite{CR2}, the functions $\eta(s)$ and $-s\epsilon'_1(s)\dot{\zeta}(\epsilon_1)$ have the same zeros and the same signs for $s\in(-\delta,\delta)$.  Moreover
\[\lim_{s\rightarrow 0,~\eta(s)\neq0}\frac{-s\epsilon'_1(s)\dot{\zeta}(\epsilon_1)}{\eta(s)}=1.\]
Now, since $\mathcal{K}_1=0$, it follows that $\lim_{s\rightarrow 0} \frac{s^2\mathcal{K}_2 }{\eta(s)}=\big( \frac{L}{\pi}\big)^2$ and we readily see that $\text{sgn}(\eta(s))=\text{sgn}(\mathcal{K}_2)$ for $s\in(-\delta,\delta)$, $s\neq0$.  Moreover, we can show that the the nonzero eigenvalue around the origin must have negative real part.  Therefore, we have proved Theorem \ref{thm31} according to the discussions above.
\end{proof}

Theorem \ref{thm44} gives wave-mode selection mechanism of the shadow system (\ref{47}).  We can easily see that all bifurcation branches other than the first one are unstable; in other words, if a stationary pattern $(v_n(s,x),\lambda_n)$ is stable, it must be on the first branch and is monotone around $(\bar v,\bar \lambda)$.  Similar as the analysis in \cite{Wq}, we can show that all solutions on the continuum of the first bifurcation branch must be monotone.  Therefore, shadow system (\ref{47}) only allows monotone stable solutions, independent of the domain size $L$.  Actually, Ni, Polascik and Yanagida \cite{NPY} have showed that any bounded (may be time-dependent) solution to an autonomous one-dimensional shadow system must be either asymptotically homogeneous or eventually monotone in $x$ in a general setting.

According to (\ref{413}), there always exist positive bifurcating solutions to (\ref{47}) for each $\epsilon>0$ being small.  However, according to Proposition \ref{prop3} and Theorem \ref{thm44}, the small-amplitude bifurcating solutions $(v_1(s,x),\lambda_1(s),\epsilon_1(s))$ are unstable in the strong competition case with $b_1=0$.  Therefore, we are motivated to find positive solutions to (\ref{47}) that have large amplitudes.

\section{Existence of transition-layer solutions to the shadow systems}\label{section5}
In this section, we study the positive solutions to a shadow system in the form of (\ref{47}) that have an interior transition layer.  Without losing much of our generality, we choose $\Phi(v)=v$ in (\ref{47}) and consider the following nonlinear boundary value problem,
\begin{equation}\label{51}
\left\{
\begin{array}{ll}
\epsilon v''+(a_2-b_2\lambda e^{-rv}-c_2v)v=0,&x \in (0,L),     \\
\int_0^L (a_1-b_1\lambda e^{-rv}-c_1v)e^{-rv} dx=0,\\
v'(0)=v'(L)=0,
\end{array}
\right.
\end{equation}
where $v=v_\epsilon(\lambda,x)$ depends $\lambda$ and $\lambda=\lambda_\epsilon$ is a positive constant to be determined.  Our aim of this section is to show that, for $\epsilon$ being sufficiently small, system (\ref{51}) admits solutions with a single transition layer which is represented by an approximation of a step-function.  The transition-layer solution can be used to model, not only the coexistence, but also a more interesting segregation phenomenon of the competing species.  Though we are concerned with $v_\epsilon(x)$ that has only single transition layer over $(0,L)$, one can construct multiple-transition-layer solutions by reflecting and periodically extending $v_\epsilon(x)$ at $x=\pm L,\pm 2L,\pm 3L,...$

Our approach is to construct at first the single transition-layer solution $v_{\epsilon}(\lambda,x)$ to (\ref{51}) for an arbitrarily predetermined $\lambda$.  Then we proceed to find $\lambda=\lambda_\epsilon$ and $v_\epsilon(\lambda_\epsilon,x)$ such that the integral condition is satisfied.  To this end, we first fix some positive $\lambda$ and consider the follow equation
\begin{equation}\label{52}
\left\{
\begin{array}{ll}
\epsilon v''+f(\lambda, v)=0,&x \in (0,L),\\
v(x)>0,& x\in(0,L),\\
v'(0)=v'(L)=0,
\end{array}
\right.
\end{equation}
where we have used the notation \[f(\lambda,v)=(a_2-b_2\lambda e^{-rv}-c_2v)v.\]
It is obvious that $\bar v_0(\lambda)=0$ is a constant solution (\ref{52}).  Integrating the first equation of (\ref{52}) over $(0,L)$, we have
\[\int_0^L f(\lambda, v) dx=0,\]
therefore $f(\lambda, v)$ changes sign in $(0,L)$.  Denoting $\tilde f(\lambda,v)=a_2-b_2\lambda e^{-rv}-c_2v$.  We can observe that $\tilde f(\lambda,v)=0$ has two positive roots if and only if $\tilde f(\lambda,0)<0$ and there exists a positive $v^*$ such that $\tilde f_v(\lambda, v^*)=0$ and $\tilde f(\lambda, v^*)>0$.  Then it follows from straightforward calculations that $f(\lambda,v)=0$ has two positive roots $\bar v_1(\lambda)$, $\bar v_2(\lambda)$ that satisfy
\begin{equation}\label{53}
f(\lambda, \bar v_1 )=f(\lambda, \bar v_2 )=0,~0<\bar v_1 <v^*=\frac{a_2}{c_2}-\frac{1}{r}< \bar v_2,
\end{equation}
provided that
\begin{equation}\label{54}
\lambda \in \Big(\frac{a_2}{b_2},~\frac{c_2}{b_2 r}e^{\frac{a_2 r}{c_2}-1} \Big),~r>\frac{c_2}{a_2}.
\end{equation}
We shall assume (\ref{54}) throughout the rest part of this section.

One can easily observe that the constant solutions $0$ and $\bar v_2$ are stable and $\bar v_1$ is unstable with respect to the time-dependent system of (\ref{52}).  Moreover, for each $\lambda$ satisfying (\ref{54}), $f(\lambda,v)$ is of Allen-Cahn type with $f_v(\lambda,0)<0$ and $f_v(\lambda,\bar v_2)<0$, therefore we have from the phase plane analysis, in  \cite{F} for example, that the following system has a unique smooth solution $V_0(z)$,
\begin{equation}\label{55}
\left\{
\begin{array}{ll}
V_0''+f(\lambda,V_0)=0,~z \in \mathbb{R},    \\
V_0(z) \in (0,\bar v_2(\lambda)),~z\in \mathbb{R},\\
V_0(-\infty)=\bar v_2(\lambda),~V_0(0)=\bar v_2(\lambda)/2,~V_0(\infty)=0;
\end{array}
\right.
\end{equation}
moreover, there exist positive constants $C$ and $\kappa$ that depend on $\lambda$ such that
\begin{equation}\label{56}
\Big\vert \frac{dV_0(z)}{dz}  \Big\vert \leq Ce^{-\kappa \vert z \vert }, ~z\in \mathbb{R}.
\end{equation}
Our construction of the transition-layer solutions $v_\epsilon$ to (\ref{52}) begins with an approximation of a step-function and we follow the idea in \cite{HS} for this purpose.
Denoting $L^*=\min\{x_0,L-x_0\}$ for each fixed $x_0\in(0,L)$, we choose the cut-off functions $\chi_0(y)$ and $\chi_1(y)$ of $C^\infty([-L,L])$ smooth to be
\begin{equation}\label{57}
\chi_0(y)\!=\!
\left\{\!\!
\begin{array}{ll}
1,\!\!&\vert y \vert \leq L^*/4,    \\
0,\!\!&\vert y \vert \geq L^*/2,\\
\in (0,1),\!\!&y\in[-L,L],
\end{array}
\right.
\text{~and~}
\chi_1(y)\!=\!
\left\{\!\!
\begin{array}{ll}
0,\!\!& y \in [-L,0],    \\
\bar v_2(\lambda)(1-\chi_0(y)),\!\!&y \in [0,L].
\end{array}
\right.
\end{equation}
Let
\begin{equation}\label{58}
V_\epsilon(\lambda,x)=\chi_0(x-x_0)V_0\Big(\lambda,\frac{x-x_0}{\sqrt{\epsilon}}\Big)+\chi_1(x-x_0),
\end{equation}
then for each $\lambda \in \Big(\frac{a_2}{b_2}, \frac{c_2}{b_2 r}e^{\frac{a_2 r}{c_2}-1} \Big)$, we want to construct a solution to (\ref{51}) that takes the form
\[v_\epsilon(\lambda,x)=V_\epsilon(\lambda,x)+\sqrt \epsilon \Psi(\lambda,x),\]
then $\Psi$ satisfies
\begin{equation}\label{59}
\mathcal{L}_\epsilon \Psi+\mathcal{P}_\epsilon+\mathcal{Q}_\epsilon=0,
\end{equation}
where
\begin{equation}\label{510}
\mathcal{L}_\epsilon=\epsilon\frac{d^2}{dx^2}+f_v(\lambda,V_\epsilon(\lambda,x)),
\end{equation}
\begin{equation}\label{511}
\mathcal{P}_\epsilon=\epsilon^{-\frac{1}{2}} \Big(\epsilon \frac{d^2 V_\epsilon(\lambda,x)}{dx^2}+f(\lambda,V_\epsilon(\lambda,x)) \Big),
\end{equation}
and
\begin{equation}\label{512}
\begin{array}{ll}
\mathcal{Q}_\epsilon&=\epsilon^{-\frac{1}{2}} \Big(f(\lambda,V_\epsilon(\lambda,x)+\sqrt \epsilon\Psi)-f(\lambda,V_\epsilon(\lambda,x))-\sqrt\epsilon \Psi f_v(\lambda,V_\epsilon(\lambda,x))\Big)\\
&=b_2\lambda e^{-r V_\epsilon}\big(V_\epsilon/\sqrt\epsilon-e^{-r \sqrt\epsilon\Psi}(V_\epsilon/\sqrt\epsilon+\Psi)-\Psi(r V_\epsilon-1)\big)-c_2\sqrt\epsilon \Psi^2.
\end{array}
\end{equation}
We readily see from (\ref{59})--(\ref{512}) that $\mathcal{P}_\epsilon$ and $\mathcal{Q}_\epsilon$ measure the accuracy that $V_\epsilon(\lambda,x)$ approximates solution $v_\epsilon(\lambda,x)$.  Our existence result is a consequence of several lemmas and we first present the following two about $\mathcal{P}_\epsilon$ and $\mathcal{Q}_\epsilon$.
\begin{lemma}\label{lem51}
Assume that $r>\frac{c_2}{a_2}$ and suppose $\lambda \in \Big(\frac{a_2}{b_2}+\delta, \frac{c_2}{b_2 r}e^{\frac{a_2r}{c_2}-1}-\delta \Big)$ for $\delta>0$ small.  Then there exist $C_1=C_1(\delta)>0$ and $\epsilon=\epsilon_1(\delta)>0$ small such that, for all $\epsilon\in(0,\epsilon_1(\delta))$
\[\sup_{x\in(0,L)} \vert\mathcal{P}_\epsilon (x)\vert \leq  C_1.  \]
\end{lemma}
\begin{lemma}\label{lem52}
Under the same conditions in Lemma \ref{lem51}.  For any $R_0>0$, there exist $C_2=C_2(\delta,R_0)>0$ and $\epsilon_2=\epsilon_2(\delta,R_0)>0$ small such that for all $\epsilon\in (0,\epsilon_2)$,
\begin{equation}\label{513}
\Vert\mathcal{Q}_\epsilon [\Psi_i]\Vert_\infty \leq C_2\sqrt \epsilon \Vert \Psi_i \Vert_\infty,~ \forall \Vert \Psi_i \Vert_\infty\leq R_0
\end{equation}
\begin{equation}\label{514}
\Vert\mathcal{Q}_\epsilon [\Psi_1]-\mathcal{Q}_\epsilon [\Psi_2]\Vert_\infty \leq C_2 \sqrt{\epsilon} \Vert \Psi_1-\Psi_2\Vert_\infty,~ \forall   \Vert \Psi_i \Vert_\infty\leq R_0
\end{equation}
\end{lemma}
We also need the following Lemma in our existence arguments.
\begin{lemma}\label{lem53}
Under the same conditions in Lemma \ref{lem51}.  For any $p\in[1,\infty]$, there exist $C_3=C_3(\delta,p)>0$ and $\epsilon_3=\epsilon_3(\delta,p)>0$ small such that for all $\epsilon \in (0,\epsilon_3(\delta,p))$, $\mathcal{L}_\epsilon$ with domain $W^{2,p}(0,L)$ has a bounded inverse $\mathcal{L}^{-1}_\epsilon$ and
\[ \Vert \mathcal{L}^{-1}_\epsilon g \Vert_p \leq C_3 \Vert g \Vert_p ,~\forall g\in L^p(0,L).\]
\end{lemma}
In light of Lemmas \ref{lem51}--\ref{lem53}, we can prove the following existence results on the transition-layer solutions to (\ref{52}).
\begin{proposition}\label{prop4}
Assume that (\ref{54}) is satisfied.  Let $x_0$ be an arbitrary point in $(0,L)$.  Then for each $\delta>0$ being small and any $\lambda \in \Big(\frac{a_2}{b_2}+\delta, \frac{c_2}{b_2 r}e^{\frac{a_2r}{c_2}-1}-\delta \Big)$, there exists a small $\epsilon_4=\epsilon_4(\delta)>0$ such that for all $\epsilon\in (0,\epsilon_4(\delta))$, (\ref{52}) has a family of solution $v_\epsilon(\lambda,x)$ such that
\[\sup_{x\in(0,L)} \vert v_\epsilon(\lambda,x)-V_\epsilon(\lambda,x) \vert \leq C_4\sqrt \epsilon,\] where $C_4$ is a positive constant independent of $\epsilon$.  In particular,
\begin{equation}\label{515}
\lim_{\epsilon \rightarrow 0^+} v_\epsilon(\lambda,x)=
\left\{
\begin{array}{ll}
\bar v_2(\lambda), \text{ compact uniformly on } [0,x_0)  ,    \\
\bar v_2(\lambda)/2,~x=x_0,\\
0,\text{ compact uniformly on } (x_0,L].
\end{array}
\right.
\end{equation}
\end{proposition}
\begin{proof}
We shall establish the existence of $v_\epsilon$ in the form of $v_\epsilon=V_\epsilon+\sqrt \epsilon \Psi$, where $V_\epsilon$ is defined in (\ref{58}).  For this purpose, it is equivalent to show the existence of smooth functions $\Psi$ and we shall to apply the Fixed Point Theorem to this end.  For all $\Psi\in \mathcal{X}$, we define
\begin{equation}\label{516}
\mathcal{S}_\epsilon[\Psi]=-\mathcal{L}^{-1}_\epsilon(\mathcal{P}_\epsilon+\mathcal{Q}[\Psi]).
\end{equation}
Then $\mathcal{S}_\epsilon$ is mapping from $C([0,L])$ to $C([0,L])$ according to elliptic regularities.  Moreover, we choose
\[\mathcal{B}=\{\Psi\in \mathcal{X} ~\vert~\Vert \Psi \Vert_\infty \leq R_0 \},\]
where $R_0\geq 2C_1C_3$.  By Lemma \ref{lem51} and \ref{lem53}, we have $\Vert \mathcal{L}^{-1} \mathcal{P}_\epsilon \Vert_\infty\leq C_1C_3$.  Therefore, it follows from Lemma \ref{lem52} that
\[\Vert \mathcal{S}_\epsilon[\Psi]\Vert_\infty \leq C_1C_3+C_2C_3\sqrt{\epsilon} R_0\leq R_0,~\forall \Psi\in \mathcal{B},\]
provided that $\epsilon$ is small.  Moreover, it follows from Lemma \ref{lem52} that for $\epsilon$ being sufficiently small,
\[\Vert \mathcal{S}_\epsilon[\Psi_1]-\mathcal{S}_\epsilon[\Psi_2]\Vert_\infty \leq \frac{1}{2} \Vert \Psi_1-\Psi_2 \Vert_\infty,~\forall \Psi_1,\Psi_2 \in \mathcal{B},\]
hence $\mathcal{S}_\epsilon$ is a contraction mapping on $\mathcal{B}$ for all small positive $\epsilon$, then it follows from the Banach Fixed Point Theorem that $\mathcal{S}_\epsilon$ has a fixed point $\Psi_\epsilon$ in $\mathcal{B}$, which is apparently a smooth solution of (\ref{59}).  Therefore $v_\epsilon=V_\epsilon+\sqrt{\epsilon}\Psi_\epsilon$ is a smooth solution of (\ref{52}).  Moreover, it is easy to verify that $v_\epsilon$ satisfies (\ref{515}) and this finishes the proof of Proposition \ref{prop4}.
\end{proof}

\begin{proof}[Proof\nopunct]\emph{of Lemma} \ref{lem51}
Substituting $V_\epsilon(\lambda,x)=\chi_0(x-x_0)V_0(\lambda,\frac{x-x_0}{\sqrt{\epsilon}})+\chi_1(x-x_0)$ into $\mathcal{P}_\epsilon (x)$ in (\ref{511}), we collect from the $V_0$ equation that
\begin{eqnarray}\label{517}
&&\epsilon\frac{d^2V_\epsilon(\lambda,x)}{dx^2}\!+\!f(\lambda,V_\epsilon(\lambda,x))\!\!\!\!\! \nonumber\\
&=&\!\!\!\!f\big(\lambda,V_\epsilon\big)\!-\!\chi_0f\big(\lambda,V_0\big)\!+\!\epsilon\chi''_0V_0+2\sqrt{\epsilon}\chi'_0V'_0\!+\!\epsilon \chi''_1\nonumber\\
&=&\!\!\!\!f\big(\lambda,V_\epsilon\big)\!-\!\chi_0f\big(\lambda,V_0\big)\!+\!\mathcal{O}(\sqrt{\epsilon})\nonumber\\
&=&\!\!\!\!f\big(\lambda,\chi_0(x-x_0)V_0(\lambda,(x-x_0)/\sqrt{\epsilon})\!+\!\chi_1(x-x_0)\big)\nonumber\\
&&-\chi_0f\big(\lambda,V_0(\lambda,(x-x_0)/\sqrt{\epsilon}) \big)+\mathcal{O}(\sqrt{\epsilon}),
\end{eqnarray}
where $\mathcal{O}(\sqrt{\epsilon})$ is taken with respect to $L^\infty$-norm.

We claim that $\vert f\big(\lambda,V_\epsilon\big)-\chi_0f\big(\lambda,V_0\big) \vert=\mathcal{O}(\sqrt{\epsilon})$ in (\ref{517}) and our discussions are divided into the following cases.  If $\vert x-x_0 \vert \leq L^*/4$, or $x-x_0\geq L^*/2$, or $x-x_0\leq -L^*/2$, we can have from (\ref{57}) that $f\big(\lambda,V_\epsilon\big)-\chi_0f\big(\lambda,V_0\big)=0$.  If $\vert x-x_0 \vert \in (L^*/4, L^*/2)$, since $V_0(z)$ decays exponentially to $0$ at $\infty$ and to $\bar v_2$ at $-\infty$, there exists $C>0$ uniform in $\epsilon$ such that
\[\vert f\big(\lambda,V_\epsilon\big)-\chi_0f\big(\lambda,V_0\big) \vert \leq C\sqrt{\epsilon}.\]
This proves our claim and Lemma \ref{lem51} follows from (\ref{517}).
\end{proof}
\begin{proof}[Proof\nopunct]\emph{of Lemma} \ref{lem52}.  First of all, we have the Taylor series
\begin{equation}\label{518}
e^{-r\sqrt{\epsilon}\Psi}=1-r\sqrt{\epsilon}\Psi+\frac{\big(r\sqrt{\epsilon}\Psi\big)^2}{2}+o(\epsilon \Psi^2)
\end{equation}
where $o(\epsilon \Psi^2)$ is taken in the $L^\infty$-norm.  Substituting (\ref{518}) into (\ref{512}), we collect
\begin{eqnarray}\label{519}
\mathcal{Q}_\epsilon \!\!\!\!\!&=&\!\!\!\!\!b_2\lambda e^{-rV_\epsilon}\Big(V_\epsilon\big(1-e^{-r\sqrt{\epsilon}\Psi}\big)/\sqrt{\epsilon}-e^{-r\sqrt{\epsilon}\Psi}
\Psi-\Psi r V_\epsilon+\Psi\Big)-c_2\sqrt{\epsilon}\Psi^2\nonumber\\
\!\!\!\!\!&=&\!\!\!\!\!b_2\lambda e^{-rV_\epsilon}\Big(V_\epsilon\big(r\sqrt{\epsilon}\Psi-\frac{r^2\epsilon\Psi^2}{2}-o(\epsilon \Psi^2)\big)/\sqrt{\epsilon} -\Psi
\big(1-r\sqrt{\epsilon} \Psi+\frac{r^2\epsilon\Psi^2}{2}+o(\epsilon \Psi^2)\big)\nonumber\\
&&-rV_\epsilon\Psi+\Psi\Big)-c_2\sqrt{\epsilon}\Psi^2\nonumber\\
\!\!\!\!\!&=&\!\!\!\!\!b_2\lambda e^{-rV_\epsilon}\big(r\sqrt{\epsilon}\Psi^2-\frac{r^2\sqrt{\epsilon}\Psi^2}{2}V_\epsilon\big)-c_2\sqrt{\epsilon}\Psi^2+o(\sqrt{\epsilon} \Psi)\nonumber\\
\!\!\!\!\!&=&\!\!\!\!\!\sqrt{\epsilon}\Big(b_2\lambda e^{-rV_\epsilon}\big(r\Psi^2-\frac{r^2\Psi^2}{2}V_\epsilon\big)-c_2\Psi^2\Big)+o(\sqrt{\epsilon} \Psi),
\end{eqnarray}
then we have from (\ref{519}) that
\begin{eqnarray}\label{520}
\Vert \mathcal{Q}_\epsilon [\Psi_i]\Vert_\infty \!\!\!\!\!&\leq&\!\!\!\!\! 2\sqrt{\epsilon}\Big(b_2\lambda r\Vert \Psi^2 \Vert_\infty+\frac{r^2}{2}\Vert V_\epsilon\Psi^2 \Vert_\infty+c_2\Vert \Psi^2\Vert_\infty\Big)\nonumber\\
\!\!\!\!\!&\leq& \!\!\!\!\!2\sqrt{\epsilon}\big(b_2\lambda r+\frac{r}{2}\Vert V_\epsilon\Vert_\infty+c_2\big)R_0\Vert\Psi_i\Vert_\infty;
\end{eqnarray}
on the other hand, for any $\Psi_1$ and $\Psi_2$ in $\mathcal B$, we have that
\begin{align}\label{521}
&\Vert\mathcal{Q}_\epsilon [\Psi_1]\!-\!\mathcal{Q}_\epsilon [\Psi_2]\Vert_\infty\nonumber\\
=&\Vert b_2\lambda e^{-rV_\epsilon}\Big( \big(\frac{V_\epsilon}{\sqrt\epsilon}+\Psi_2\big)  e^{-r\sqrt {\epsilon} \Psi_2}-\big(\frac{V_\epsilon}{\sqrt{\epsilon}}+\Psi_1\big)  e^{-r\sqrt {\epsilon} \Psi_1}+(r V_\epsilon-1)(\Psi_2-\Psi_1)\Big)\nonumber\\
  &+c_2\sqrt {\epsilon } (\Psi_2^2-\Psi_1^2)\Vert_\infty \nonumber\\
\leq& b_2\lambda \Vert  \big(\frac{V_\epsilon}{\sqrt\epsilon}+\Psi_2\big)  e^{-r\sqrt {\epsilon} \Psi_2}-\big(\frac{V_\epsilon}{\sqrt{\epsilon}}+\Psi_1\big)  e^{-r\sqrt {\epsilon} \Psi_1} +(r V_\epsilon-1)(\Psi_2-\Psi_1)\Vert_\infty \nonumber \\
  &+c_2\sqrt{\epsilon} \Vert\Psi_2+\Psi_1 \Vert_\infty \Vert\Psi_2-\Psi_1 \Vert_\infty  \\
\leq& b_2\lambda \max_{\Psi \in \mathcal B}\Vert (1-rV_\epsilon)(-r \Psi+\frac{\sqrt{\epsilon} r^2}{2}\Psi^2+o(\sqrt{\epsilon}\Psi^2))-r \Psi e^{-r\sqrt {\epsilon} \Psi} \Vert_\infty  \sqrt{\epsilon} \Vert\Psi_2-\Psi_1 \Vert_\infty\nonumber \\
&+ 2c_2R_0\sqrt{\epsilon}\Vert\Psi_2-\Psi_1 \Vert_\infty, \nonumber
\end{align}
then we see that (\ref{513}) and (\ref{514}) follow from (\ref{520}) and (\ref{521}) respectively.
\end{proof}

\begin{proof}[Proof\nopunct]\emph{of Lemma} \ref{lem53}.
To show that $\mathcal{L}_\epsilon$ in (\ref{510}) is invertible, it is sufficient to prove that $\mathcal{L}_\epsilon$ defined on $L^p(0,L)$ with the domain $W^{2,p}(0,L)$ has only trivial kernel.  Our proof is quite similar as that of Lemma 5.4 in \cite{LN2} given by Lou and Ni.  We argue by contradiction.  Choose a sequence $\{(\epsilon_i,\lambda_i)\}_{i=1}^\infty$ with $\epsilon_i \rightarrow 0$ and $\lambda_i \rightarrow \lambda \in \Big(\frac{a_2}{b_2}+\delta, \frac{c_2}{b_2 r}e^{\frac{a_2}{c_2 r}-1}-\delta \Big)$ as $i\rightarrow \infty$.  Suppose that there exists $\Phi_i\in W^{2,p}(0,L)$ satisfying
\begin{equation}\label{522}
\left\{
\begin{array}{ll}
\epsilon_i \frac{d^2\Phi_{i}}{dx^2}+f_v(\lambda_i,V_{\epsilon_i}(\lambda_i,x))\Phi_i=0,x\in(0,L),\\
\Phi_i'(0)=\Phi_i'(L)=0,\\
\sup_{x\in(0,L)}   \Phi_i(x)  =1.
\end{array}
\right.
\end{equation}
Let us denote
\[\tilde{\Phi}_i(z)=\Phi_i(x_0+\sqrt{\epsilon_i}z),~\tilde V_{\epsilon_i}(\lambda_i,z)=V_{\epsilon_i}(\lambda_i,x_0+\sqrt{\epsilon_i}z),\]
for all $z\in\big(x_0-\frac{1}{\sqrt{\epsilon_i}},x_0+\frac{1}{\sqrt{\epsilon_i}}\big)$, $i\in \mathbb N^+$, then we see that
\[\frac{d^2\tilde\Phi_{i}}{dz^2}+f_v(\lambda_{i},\tilde V_{\epsilon_i}(\lambda_{i},z))\tilde\Phi_{i}=0,~z\in \Big(x_0-\frac{1}{\sqrt{\epsilon_i}},x_0+\frac{1}{\sqrt{\epsilon_i}}\Big).\]
Thanks to (\ref{54}) and (\ref{57}), it is easy to see that both $f_v(\lambda_i, \tilde V_{\epsilon_i})$ and $\tilde\Phi_i$ are bounded for all $i\in \mathbb N^+$, therefore, we can apply the standard elliptic regularity and the diagonal argument that, after passing to a subsequence if necessary as $i \rightarrow \infty$, $\tilde \Phi_i \rightarrow \tilde \Phi_0$ in $C^1(\mathbb{R}_c)$ for any compact subset $\mathbb{R}_c$ of $\mathbb{R}$; moreover, $\tilde\Phi_0$ is a $C^\infty$-smooth function that satisfies
\begin{equation}\label{523}
\frac{d^2\tilde\Phi_0 }{dz^2}+f_v(\lambda,V_0(\lambda,z))\tilde\Phi_0=0,z\in \mathbb{R},
\end{equation}
where $V_0(\lambda,z)$ is the unique solution of (\ref{55}).  Let $x_i\in[0,L]$ such that $\Phi_i(x_i)=1$.  Then we have that $f_v(\lambda_i,V_{\epsilon_i}(\lambda,x_i))\geq 0$ according to the Maximum Principle.  We claim that $\vert z_i\vert =\vert \frac{x_i-x_0}{\sqrt{\epsilon_i}}\vert \leq C_0 $ for some $C_0$ independent of $\epsilon_i$.  Suppose otherwise and there exists a sequence $z_i \rightarrow \pm \infty$ as $\epsilon_i\rightarrow 0$, then it follows from the definition of $V_{\epsilon_i}$ in (\ref{58}) that
$\tilde V_{\epsilon_i}\big(\lambda, z_i \big) \rightarrow \bar v_2$ and $0$ respectively.  On the other hand, we recall that $f_v(\lambda, \bar v_2)<0$ and $f_v(\lambda, 0)<0$, therefore $f_v(\lambda,V_{\epsilon_i} (\lambda, x_i ))<0$ for $\epsilon_i$ being sufficiently small and we reach a contradiction, hence $\vert z_i\vert$ is bounded for all $\epsilon_i$ being sufficiently small.  Now as $\epsilon_i \rightarrow 0$, one has that $z_i\rightarrow z_0$ for some $z_0\in\mathbb{R}$ such that
\[\tilde\Phi_0(z_0)=\sup_{z\in\mathbb{R}}\tilde\Phi_0(z)=1,~\tilde\Phi_0'(z_0)=0.\]

Differentiating equation (\ref{523}) with respect to $z$, we can obtain
\begin{equation}\label{524}
\frac{d^2 V'_0}{dz^2}+f_v(\lambda,V_0(\lambda,z))V'_0=0,
\end{equation}
where $V'_0=\frac{dV_0}{dz}$.  Multiplying (\ref{524}) by $\Phi$ and (\ref{523}) by $V'_0$ respectively and then integrating them over $(-\infty,z_0)$ by parts, we obtain that
\[0=\int_{-\infty}^{z_0} \tilde\Phi_0''V'_0-\tilde\Phi_0 (V'_0)''dz=\tilde\Phi_0'(z)V'_0(z)\big \vert_{-\infty}^{z_0}-\tilde\Phi_0(z)V''_0(z)\big \vert_{-\infty}^{z_0}.\]
Since $V_0''(z_0)\neq0$, we conclude from (\ref{56}) that $\tilde\Phi_0(z_0)=0$, however this is a contradiction to the assumption that $\tilde\Phi_0(z_0)=1$.  Therefore, we have proved the invertibility of $\mathcal{L}_\epsilon$ and we denote its inverse operator by $\mathcal{L}^{-1}_\epsilon$ from now on.

To show that $\mathcal{L}^{-1}_\epsilon$ is uniformly bounded on $L^p(0,L)$ for all $p\in[1,\infty]$, it suffices to prove it for $p=2$ thanks to the Marcinkiewicz interpolation Theorem.
We consider the following eigenvalue problem
\begin{equation}\label{525}
\left\{
\begin{array}{ll}
\mathcal{L}_\epsilon \varphi_{i,\epsilon}=\mu_{i,\epsilon} \varphi_{i,\epsilon},~x\in(0,L),\\
\varphi_{i,\epsilon}'(0)=\varphi_{i,\epsilon}'(L)=0,\\
\sup_{x\in(0,L)}   \varphi_{i,\epsilon}(x) =1.
\end{array}
\right.
\end{equation}
By applying the same analysis as above, we can show that for each $\lambda\in\Big(\frac{a_2}{b_2}+\delta, \frac{c_2}{b_2 r}e^{\frac{a_2 r}{c_2}-1} -\delta\Big)$, there exists a constant $C(\lambda)>0$ independent of $\epsilon$ such that $\mu_{i,\epsilon}\geq C(\lambda)$ for all $\epsilon$ sufficiently small.  Therefore
\[\Vert \mathcal{L}_\epsilon^{-1} g\Vert_{2} =\big\Vert \sum_{j=0}^\infty \frac{<g,\varphi_{i,\epsilon}>}{\mu_{i,\epsilon}} \varphi_{i,\epsilon} \big\Vert_{2}\leq C^{-1} \Vert g\Vert_{2},\]
where $<\cdot,\cdot>$ denotes the inner product in $L^2$.  This finishes the proof of Lemma \ref{lem53}.
\end{proof}

We proceed to employ the solution $v_\epsilon(\lambda,x)$ of (\ref{52}) as obtained in Proposition \ref{prop4} to construct solutions of (\ref{51}).  Therefore, we only need to show that there exists $(v_\epsilon(\lambda_\epsilon,x), \lambda_\epsilon)$ to (\ref{52}) such that the integral condition in (\ref{51}) is satisfied.

First of all, we can have from simple calculations that $\lambda(\bar{v}_2)<\frac{a_2}{b_2}$ for all $\bar{v}_2\in(\frac{a_1}{c_1},\frac{a_2}{c_2})$ if $r\in(\frac{c_2}{a_2},\frac{c_1}{a_1}\ln\frac{a_2c_1}{a_2c_1-a_1c_2}]$, this contradicts (\ref{54}) and no transition layer occurs under this condition.  Therefore, we shall assume $r>\frac{c_1}{a_1}\ln\frac{a_2c_1}{a_2c_1-a_1c_2}$ from now on.  Now we present the following necessary conditions on the existence of transition-layer solutions to (\ref{52}).
\begin{proposition}\label{proposition5}
Denote $\lambda(\bar{v}_2)=\frac{a_2-c_2\bar{v}_2}{b_2}e^{r\bar{v}_2}$.  We assume that $\frac{a_1}{c_1}<\frac{a_2}{c_2}$ and $ \bar{v}_2\in (\frac{a_1}{c_1},\frac{a_2}{c_2})\cap(\frac{a_2}{c_2}-\frac{1}{r},\frac{a_2}{c_2})$ with $r\in(\frac{c_1}{a_1}\ln\frac{a_2c_1}{a_2c_1-a_1c_2},\infty)$.  Then there exists a unique $\bar{v}_2^{**}$ such that $\lambda(\bar{v}_2^{**})=\frac{a_2}{b_2}$; moreover, if (\ref{51}) exists a transition-layer solution $v_\epsilon$ that satisfies (\ref{515}), we must have the followings:

(i) if $r\in\big(\frac{c_1}{a_1}\ln\frac{a_2c_1}{a_2c_1-a_1c_2},\frac{c_1c_2}{a_2c_1-a_1c_2}\big)$, then $\lambda(\bar{v}_2)\in\big(\frac{a_2}{b_2},\frac{a_2c_1-a_1c_2}{b_2c_1}e^{\frac{a_1r}{c_1}}\big)$ for all $\bar{v}_2\in(\frac{a_1}{c_1}, \bar{v}_2^{**})$;

(ii) if $r\in[\frac{c_1c_2}{a_2c_1-a_1c_2},\infty)$, then $\lambda(\bar{v}_2)\in\big(\frac{a_2}{b_2},\frac{c_2}{b_2r}e^{\frac{a_2r}{c_2}-1}\big)$ for all $\bar{v}_2\in(\frac{a_2}{c_2}-\frac{1}{r},\frac{a_2}{c_2})$.
\begin{proof}
It follows from (\ref{53}) and (\ref{54}) that $\bar{v}_2\in(\frac{a_1}{c_1},\frac{a_2}{c_2})$; moreover, we can easily show that $\lambda(\bar{v}_2)$ is monotone increasing when $\bar{v}_2\in(\frac{a_1}{c_1},\frac{a_2}{c_2}-\frac{1}{r})$ and is monotone decreasing when $\bar{v}_2\in(\frac{a_2}{c_2}-\frac{1}{r},\frac{a_2}{c_2})$.  On the other hand, since $\lambda(\frac{a_2}{c_2}-\frac{1}{r})=\frac{c_2}{b_2r} e^{\frac{a_2}{c_2}r-1}>\frac{a_2}{b_2}$ from(\ref{54}) and $\lambda(\frac{a_2}{c_2})=0$, there exists a unique $\bar{v}_2^{**}$ such that $\lambda(\bar{v}^{**}_2)=\frac{a_2}{b_2}$.  We also want to point out that, if $\frac{a_1}{a_2}<\frac{c_1}{c_2}$,
\[ \frac{c_2}{a_2}<\frac{c_1}{a_1}\ln\frac{a_2c_1}{a_2c_1-a_1c_2}<\frac{c_1c_2}{a_2c_1-a_1c_2}.  \]
Sending $\epsilon$ to zero in (\ref{51}), we have from Lebesgue Dominated Convergence Theorem that $\lambda=\frac{a_2-c_2\bar{v}_2}{b_2}e^{r\bar{v}_2}$.  Now we finish the proof by considering the following two cases.

If $\frac{a_2}{c_2}-\frac{1}{r}<\frac{a_1}{c_1}$  and $\lambda(\frac{a_1}{c_1})>\frac{a_2}{b_2},$ then $\bar{v}_2\in(\frac{a_1}{c_1},\frac{a_2}{c_2})$ and $r\in(\frac{c_1}{a_1}\ln\frac{a_2c_1}{a_2c_1-a_1c_2},\break\frac{c_1c_2}{a_2c_1-a_1c_2}).$  Moreover, we have that $\lambda(\frac{a_1}{c_1})=\frac{a_2c_1-a_1c_2}{b_2c_1}e^{\frac{a_1r}{c_1}}>\frac{a_2}{b_2}$ for all $\bar{v}_2\in(\frac{a_2}{c_2}-\frac{1}{r},\frac{a_2}{c_2})$, thanks to the monotonicity of $\lambda(\bar{v}_2)$, therefore we have finished the proof of \emph{(i)}.

If $\frac{a_2}{c_2}-\frac{1}{r}>\frac{a_1}{c_1},$  we have that $r\in(\frac{c_1c_2}{a_2c_1-a_1c_2},\infty)$ and correspondingly $\bar{v}_2\in(\frac{a_2}{c_2}-\frac{1}{r},\frac{a_2}{c_2})$.  Then we can also prove \emph{(ii)} by the monotonicity of $\lambda(\bar{v}_2)$ for $r\in(\frac{c_1c_2}{a_2c_1-a_1c_2},\infty)$
\end{proof}
\end{proposition}

For the simplicity of notations, we write
\[r^* =\frac{c_1}{a_1}\ln\frac{a_2c_1}{a_2c_1-a_1c_2}\] and denote
\begin{equation}\label{526}
I_0=
\left\{
\begin{array}{ll}
 (\frac{a_1}{c_1},v^{**}), &\text{ if } r\in(r^{*},\frac{c_1c_2}{a_2c_1-a_1c_2}) ,\\
 (\frac{a_2}{c_2}-\frac{1}{r},v^{**}), &\text{ if } r\in(\frac{c_1c_2}{a_2c_1-a_1c_2},\infty);
\end{array}
\right.
\end{equation}
Now we are ready to present another main result of our paper.
\begin{theorem}\label{thm54}
Suppose $r\in(r^*,\infty)$ and $\frac{a_1}{a_2}<\frac{c_1}{c_2}$ with $b_1 \rightarrow 0$
as $\epsilon \rightarrow 0$.  .  Then for each $\bar {v}_2 \in I_0$ defined in (\ref{526}), there exists $\epsilon_0=\epsilon_0(\bar {v}_2)>0$ such that (\ref{51}) admits positive solutions $(\lambda_\epsilon ,v_\epsilon(\lambda_\epsilon,x))$ for all $\epsilon\in(0,\epsilon_0)$; moreover
\begin{equation}\label{527}
\lim_{\epsilon \rightarrow 0^+} v_\epsilon(\lambda_\epsilon,x)=
\left\{
\begin{array}{ll}
\bar v_2,&\text{ compact uniformly on } [0,x_0),\\
\bar v_2/2,& x=x_0,\\
0, &\text{ compact uniformly on } (x_0,L],
\end{array}
\right.
\end{equation}
and
\begin{equation}\label{528}
\lim_{\epsilon \rightarrow 0^+} \lambda_\epsilon = \bar \lambda_0=\frac{(a_2-c_2\bar v_2)e^{r\bar v_2}}{b_2},
\end{equation}
where in (\ref{527})
\begin{equation}\label{529}
x_0=\frac{a_1L}{a_1-(a_1-c_1  \bar{v}_{2})e^{-r\bar{v}_{2}}}.
\end{equation}
\end{theorem}
\begin{proof}
We shall apply the Implicit Function Theorem for our proof.  First of all, we see from Proposition \ref{proposition5} that, $\lambda(\bar{v}_{2})=\frac{a_2-c_2 \bar{v}_{2}}{b_2}e^{r \bar{v}_{2}}$ is a one-to-one function of $\bar v_2$ in $I_0$ and $\lambda(\bar{v}_{2})\in \Big(\frac{a_2}{b_2},\frac{c_2}{b_2 r}e^{\frac{a_2}{c_2}r-1} \Big)$.

For each $\lambda(\bar{v}_{2})\in \Big(\frac{a_2}{b_2},\frac{c_2}{b_2 r}e^{\frac{a_2}{c_2}r-1} \Big)$ with $\bar v_2\in I_0$, we take $\delta>0$ and $\epsilon_0(\delta)>0$ small and then define for all $\epsilon\in(0,\epsilon_0)$
\begin{equation}\label{530}
\mathcal{I}(\epsilon,\lambda)=\int_0^L \big(a_1-b_1\lambda e^{-rv_\epsilon(\lambda,x)}-c_1v_\epsilon(\lambda,x)\big)e^{-rv_\epsilon(\lambda,x)} dx,
\end{equation}
where $\lambda \in \big(\bar \lambda_0-\delta , \bar \lambda_0+\delta \big)$, $\delta>0$ small and $\bar \lambda_0$ is to be determined.  For $\epsilon\leq 0$, we set $v_\epsilon(\lambda,x)=\bar v_2(\lambda)$ if $x\in [0,x_0)$ and $v_\epsilon(\lambda,x)=0$ if $x\in(x_0,L]$.  Then we have that
\[\mathcal{I}(\epsilon,\lambda)\equiv  x_0(a_1-b_1\lambda e^{-r\bar v_2}-c_1\bar v_2)e^{-r\bar v_2}+(L-x_0)(a_1-b_1\lambda),~\forall \epsilon \leq 0\]
On the other hand, for $\epsilon>0$, we have from (\ref{530}) that
\begin{eqnarray*}
\frac{\partial \mathcal{I}(\epsilon,\lambda)}{\partial \lambda }\!\!\!\!\!&=&\!\!\!\!\!\int_0^L \Big(2b_1\lambda r e^{-2rv_\epsilon}-a_1re^{-rv_\epsilon}+c_1rv_\epsilon e^{-rv_\epsilon}-c_1e^{-rv_\epsilon}\Big)\frac{\partial v}{\partial \lambda}- b_1e^{-2rv_\epsilon} dx  \nonumber\\
\!\!\!\!\!&=&\!\!\!\!\!\int_0^L \Big(b_1\lambda r e^{-2rv_\epsilon}-c_1e^{-rv_\epsilon}\Big)\frac{\partial v}{\partial \lambda}- b_1e^{-2rv_\epsilon} dx;
\end{eqnarray*}
moreover, we see from Proposition \ref{proposition5} that $\lim_{\epsilon \rightarrow 0^+} \frac{\partial v}{\partial \lambda}<0$ for $x\in[0,x_0)$ and $\lim_{\epsilon \rightarrow 0^+} \frac{\partial v}{\partial \lambda} \equiv 0$ for $x\in(x_0,L]$ pointwisely.  By the Lebesgue Dominated Convergence Theorem, we readily see that $\lim_{\epsilon \rightarrow 0^+} \frac{\partial \mathcal{I}(\epsilon,\lambda)}{\partial \lambda } \neq0$.  Therefore $\frac{\partial \mathcal{I}(\epsilon,\lambda)}{\partial \lambda }$ is continuous in a neighborhood of $(0, \bar \lambda_0)$ for all $\bar \lambda_0 \in \big(\frac{a_2}{b_2},\frac{c_2}{b_2 r}e^{\frac{a_2}{c_2}r-1} \big)$.  Finally, it concludes from the Implicit Function Theorem that, there exist solutions $(v_\epsilon(\lambda_\epsilon,x),\lambda_\epsilon)$ to system (\ref{51}) $\lambda_\epsilon \rightarrow \lambda_0$.

Sending $\epsilon$ to zero in (\ref{51}), we conclude from the Lebesgue Dominated Convergence Theorem that
\begin{equation*}
x_0(a_1-c_1\bar v_2( \lambda_0))e^{-r\bar v_2( \lambda_0)}+a_1(L-x_0)=0,
\end{equation*}
then we see that (\ref{529}) follows from straightforward calculations.  This completes the proof of Theorem \ref{thm54}.
\end{proof}
Thanks to Proposition \ref{proposition5}, we have that in (\ref{28}), $\lambda \in \Big(\frac{a_2}{b_2},\frac{a_2c_1-a_1c_2}{b_2c_1}e^{\frac{a_1 r}{c_1}} \Big)$ if $r\in(r^*,
\frac{c_1c_2}{a_2c_1-a_1c_2})$ and $\lambda_0 \in \Big(\frac{a_2}{b_2},\frac{c_2}{b_2 r}e^{\frac{a_2 r}{c_2}-1} \Big)$ if $r\in(\frac{c_1c_2}{a_2c_1-a_1c_2},\infty)$.

According to (\ref{527}) and Theorem \ref{thm42}, both $u$ and $v$ admit a single transition layer at $x=x_0$ if $D_1$ and $\chi$ are sufficiently large.  Moreover, we can construct infinitely many such transition-layer solutions to the shadow system (\ref{51}) for any $x=x_0$ given by (\ref{529}), provided that $\bar v_2\in I_0$.  These transition-layer solutions can be used to model the segregation phenomenon in interspecific competition.

\section{Conclusion and discussion}\label{section6}
In this paper, we propose and study the $2\times 2$ reaction-advection-diffusion Lotka-Volterra system (\ref{16}) that models the population dynamics of two competing species.  For the one-dimensional domain $\Omega=(0,L)$, the global classical solutions have been proved to exist and be uniformly bounded for all $t\in(0,\infty)$.  Same results are also found for multi-dimensional domain for a parabolic-elliptic system.  For the $1D$ stationary problem (\ref{35}), we show that the constant solution $(\bar u,\bar v)$ becomes unstable for $\chi> \min_{k\in \mathbb N^+} \chi_k$ in the sense of Turing's instability driven by advection.  And then we apply the Crandall-Rabinowitz bifurcation theories to establish its nonconstant positive solutions.  The stability or instability of these bifurcating solutions has also been obtained when the diffusion rates $D_1$ is sufficiently large and $D_2$ is sufficiently small.    By sending $\chi \rightarrow \infty$ with $\frac{\chi}{D_1}=r\in(0,\infty)$, i.e, $D_1$ and $\chi$ being comparably large, we show that $uv$ converges to a constant $\lambda$ uniformly over $[0,L]$ and $v$ converges to $v_\infty$ in $C^1([0,L])$ such that $v_\infty$ satisfies (\ref{47}), a shadow system to (\ref{35}).  The existence and stability of $v_\infty$ have also been obtained through bifurcation theories.  The bifurcating solutions of (\ref{35}) and (\ref{47}) are small perturbations from the constant solutions to the corresponding system and they have only small oscillations.  Our results also provide a wave-mode selection mechanism of stable patterns.  For the full system (\ref{35}), the only stable bifurcation branch $\Gamma_k$ is that of $k_0$ when $\chi_k$ is minimized.  For the shadow system (\ref{47}), only monotone bifurcation solution is stable, which is on the first bifurcation branch that $\epsilon_k$ is maximized.  Finally, for the strong competition case $\frac{a_1}{a_2}<\frac{c_1}{c_2}$ with $b_1=0$, we construct positive transition-layer solutions to the shadow system (\ref{47}), or (\ref{51}).

We are interested in the mathematical modeling of interspecific segregation by the one-dimensional stationary problem (\ref{35}).  According to Theorem \ref{thm42} and Theorem \ref{thm54}, the steady state $v$ has an interior transition layer at $x_0$ if $\min\{\chi,D_1,\frac{1}{D_2}\}$ is sufficiently large with $\chi$ and $D_1$ being comparable, therefore $u\approx \frac{\lambda}{e^{rv}}$ also admits a transition layer at $x_0$.  These transition-layer solutions model the phenomenon that $v$ dominates the habitat $[0,x_0)$ and $u$ the region $(x_0,L]$.  We see that the formation of the transition-layers is a combining effect of diffusion rate $D_i$, $i=1,2$ and advection rate $\chi$.  However, neither the diffusion nor the advection determines the size of the dominating habitat for each species.  Therefore, coexistence and species segregation depend on the kinetic terms rather than species dispersal rates.  Biologically, fast-diffuser $u$ (with large $D_1$) can take the fast-escaping mechanism ($\chi$ large) to avoid interspecific competition and this can be an effective and active way for species $u$ to coexist with and eventually form spatial segregation with slower-diffusers $v$ (with small $D_2$).  We want to compare our results on (\ref{35}) and the results obtained by Lou and Ni \cite{LN,LN2} on the stationary system of (\ref{12}).  It is shown that, as $D_1$ approaches to infinity with $\rho_{12}/D_1=r\in(0,\infty]$, the steady states $(u,v)$ of (\ref{12}) approximate certain shadow systems if $r\in(0,\infty)$ and $v\rightarrow 0$ if $r=0$.  Therefore, the species $v$ extinguish through competition under the influence of large advection rate of species $u$.  This coincides with our results on (\ref{35}) as $D_1 \rightarrow \infty$ with $r=\infty$.  Therefore, large diffusion and advection can protect the species for survival and eliminate the competitors from their habitat.

To the global existence of (\ref{16}) over multi-dimensional domain, we require $\tau=0$ and some other conditions on $D_2$, $b_1$, $b_2$ and $\chi$.  Though these assumptions are not entirely unrealistic from the view points of mathematical and ecological modeling, we surmise that they are not necessary for our mathematical analysis.  To this end, we still need to estimate $\Vert \nabla v(\cdot ,t) \Vert_{L^p}$ for some $p>N$.  This can be a mathematically challenging problem.  The crowding term with $b_1$ in the $u$-equation helps to prevent $\Vert u(\cdot ,t) \Vert_{L^\infty}$ from blowing up over finite or infinite time period, however, whether or not it sufficient for this purpose is unknown, in particular when $\chi$ is large.  Similar problems have been proposed and discussed for general reaction-advection-diffusion systems by survey paper of Cosner \cite{C} and the chemotaxis model by \cite{Winkler2}.  Though the global bounded solutions (\ref{16}) has been obtained for in both $1D$ and $nD$ domains, the understanding of its global dynamics is way from complete.  The dynamics of the double-advection model (\ref{14}) is a more realistic but harder problem that deserves future explorations.

There are also a few important and interesting unsolved questions regarding the stationary system (\ref{35}).  The existence of non-existence of nontrivial solutions in multi-dimensional domain deserves exploring.  The stability of the interior transition-layer solution is another challenging problem that worths future attentions.  Our calculations become extremely complicated and difficult if one wants to remove the constraints on the parameters in our stability analysis of the bifurcating solutions in Theorem \ref{thm32} and Theorem \ref{thm44}.  Last but not least, the steady states of the double-advection problem (\ref{14}) can be another important but also challenging problem that one can pursue in the future.

% We want to note that the traveling wave solutions of (\ref{2}) is another fascinating and also important question to probe.  We refers our reader to \cite{LNW,N2,D,H,WZ} and the references therein for the results and recent developments in this direction.

\medskip
% The data information below will be filled by AIMS editorial staff
Received  January 2014; revised   August  2014.
\medskip

\end{document}